\documentclass[11pt]{article}
\usepackage[margin=1in]{geometry}

\usepackage[title]{appendix}
\usepackage[utf8]{inputenc}
\usepackage[T1]{fontenc}
\usepackage{authblk}
\usepackage{amsthm}
\usepackage{amssymb}
\usepackage{amsmath}
\usepackage{float}
\usepackage{tikz}
\usetikzlibrary{arrows,calc,shapes,decorations.pathreplacing}
\usepackage{hyperref}
\hypersetup{colorlinks=true,allcolors=blue}
\usepackage{caption}
\usepackage{pgfplots}
\usepackage{subcaption}
\usepackage{enumitem}

\def\addlegendimage{\csname pgfplots@addlegendimage\endcsname}

\newtheorem{theorem}{Theorem}
\newtheorem{lemma}[theorem]{Lemma}
\newtheorem{proposition}[theorem]{Proposition}
\newtheorem{corollary}[theorem]{Corollary}

\newtheorem*{theorem*}{Theorem}
\newtheorem{definition}{Definition}
\newtheorem{example}{Example}
\newtheorem{remark}{Remark}

\newtheorem{assumption}{Assumption}
\renewcommand\theassumption{A\arabic{assumption}}

\def\P{\mathrm{P}}
\def\E{\mathrm{E}}

\def\Var{\mathrm{Var}}
\def\Cov{\mathrm{Cov}}
\def\det{\mathrm{det}}
 
\def\Ep{\E_{\mathbf{p}}}
\def\Varp{\Var_{\mathbf{p}}}
\def\Covp{\Cov_{\mathbf{p}}}
\DeclareMathOperator*{\argmin}{arg\,min}
\DeclareMathOperator*{\argmax}{arg\,max}
\newcommand{\apost }{'}
\newcommand{\parrow}{\:\to_{\mathrm P}\:}
\newcommand{\darrow}{\:\to_{\rm d}\:}
\newcommand{\asarrow}{\:\to_{\rm as}\:}

\allowdisplaybreaks

\pgfplotsset{compat=1.17}

\begin{document}

\title{Stochastic Approximation of Symmetric Nash Equilibria in Queueing Games\footnote{To appear in Operations Research.}}

\author[1]{Liron Ravner}
\author[2]{Ran I. Snitkovsky}
\affil[1]{\small{Department of Statistics, University of Haifa}}
\affil[2]{\small{Coller School of Management, Tel Aviv University}}

\date{\today}
\maketitle

\begin{abstract}
We suggest a novel stochastic-approximation algorithm to compute a symmetric Nash-equilibrium strategy in a general queueing game with a finite action space. The algorithm involves a single simulation of the queueing process with dynamic updating of the strategy at regeneration times. Under mild assumptions on the utility function and on the regenerative structure of the queueing process, the algorithm converges to a symmetric equilibrium strategy almost surely. This yields a powerful tool that can be used to approximate equilibrium strategies in a broad range of strategic queueing models in which direct analysis is impracticable. 
\end{abstract}


\section{Introduction}

This paper presents a simulation-based algorithm that computes, or learns, a symmetric Nash equilibrium in a general class of queueing games. Our purpose is to suggest a robust, easy-to-implement equilibrium approximation scheme, mainly targeted at queueing games in which identifying equilibrium strategies is difficult to pursue analytically. In particular, we consider service systems to which rational customers arrive according to a renewal process, each chooses between finitely many actions with the goal of maximizing their expected utility. Customers in our model assume that the system operates in a steady state, however, their assessment of the utility depends on other customers' actions, giving rise to a game-theoretic equilibrium problem: Identifying a behavior profile that prescribes an optimal play for customers, such that the induced steady state is consistent with each customer's belief regarding the behavior of the others. 

Though the initiation of the literature about strategic customer behavior in queues is often accredited to Naor \cite{N1969}, the most basic, yet non-degenerated queueing game was first introduced by Edelson and Hildebrand \cite{EH1975}, and was coined the unobservable M/M/1 model. It discusses a single-server queue were customer utility decreases linearly with the waiting time, and customers need to decide whether to join the queue or balk, without observing the system state. 
Since then, various different queueing games have been studied in the literature, to model the behavior of callers in  call centers, users of communication networks, passengers in a transportation system, and more (see  overviews and exhaustive surveys in \cite{HH2003} and \cite{H2016}). 

Traditionally, this literature distinguishes between `unobservable' and `observable' models, although this conceptual distinction is rather obscure. By unobservable, one usually means that the strategy of a customer cannot rely on any state information but the premise that this state is drawn from the stationary distribution. Conversely, observable (and partially observable) models are commonly used to describe games in which customers are endowed with information about the state to which they arrive, hence, making their actions state dependent. In that sense, the observable framework arguably allows a richer description of customer behavior, although technically, many observable queueing games discussed in the literature can be reframed as unobservable models via a suitable modification of the action set.

In most queueing games, observable or unobservable, customer expected utility (ex ante) is a function of their own strategy, in which the expectation is with respect to the stationary distribution arising from the cumulative behavior of all players. When customers are homogeneous, the conventional solution concept is the symmetric Nash equilibrium. Characterizing a symmetric Nash equilibrium means finding a fixed point of the best-response function, which is the function that maps a strategy profile to the best strategy to play against it. 

Unlike the Markovian setting in \cite{EH1975}, for non-elementary queueing models, the stationary state distribution and the resulting customer expected utility rarely admit closed form expressions. Think, for example, of a GI/G/1 queue, with customer utility depending on their delay. In this example, as well as in many others, there is no explicit formula for the expected sojourn time of a customer (let alone its distribution), hence, approaching an equilibrium solution with only analytic tools is intractable. Tackling such problems often brings the need for sophisticated numeric calculation and simulation schemes.

We present a stochastic-approximation (SA) algorithm that converges to an equilibrium solution for a general class of models. Specifically, the algorithm involves simulating the queueing process, and based on the realizations, updating the strategy at carefully chosen regeneration times. The update relies on estimating the deviation of the strategy from its best response, using a smoothing transformation of the best-response correspondence. Under mild regularity assumptions it is shown that the algorithm converges almost surely to a Nash equilibrium as the simulation length goes to infinity. The convergence conditions are verified for several examples, and numerical results are presented.

Simulation is a very popular technique among queueing theorists and practitioners as a tool to approximate the performance of a model that cannot be approached analytically. The literature studying the theory and applications of simulation methods for queueing networks is extensive (see \cite{AG2007} for a comprehensive overview), and algorithms can efficiently learn the performance of quite intricate networks, given the set of primitives. However, in the strategic queueing domain, simulation is hardly ever used, primarily due to the difficulties that arise in the equilibrium search process: 
In queueing games, the strategy profile governs the dynamics of the underlying system. Thus, to identify its best response (or an approximation thereof), the strategy should be fed to the simulator up front.
For a strategy given a priori, simulation can be used to verify, with high certainty, whether this strategy approximately meets the equilibrium criterion. But when an equilibrium strategy is to be found, with no simple solution at hand, this will require traversing through the set of possible strategies,
and performing a separate simulation at each iteration to check each strategy. This procedure can be tedious and even impractical, especially when the simulation time at each iteration is exceedingly long. Moreover, the outcome of each iteration is subject to uncertainty, and so, if not carefully implemented, such a process is not guaranteed to converge at all, or perhaps worse -- converge to an undesired limit.

It is possible that for these reasons the literature of strategic queueing is often limited to stylized models, of which performance measures can be easily expressed. In their own right, stylized models play an important role in understanding the theoretical properties of queueing games. Yet, for the purpose of departing from a merely theoretical framework, it is crucial to have the ability to compute the equilibrium outcome in elaborate systems as well. 

We introduce a novel simulation scheme that jointly learns the system's statistical characteristics and customer best-response dynamics, such that the process efficiently converges to a Nash equilibrium. Our framework allows for customer strategies to depend on state information, thus, it is applicable for a wide range of models. Still, the key results and intuition supporting these results are more easily conveyed using the so-called `unobservable' class of models. We first discuss our method in the context of unobservable queueing games, and later complement the discussion by suggesting a refinement for observable models as well. We further discuss several practical and theoretical aspects of implementation, such as rate of convergence.

Our method relies on the SA algorithm commonly known as the Robbins-Monro algorithm \cite{RM1951}, a fundamental building block underpinning various stochastic optimization techniques, among which is the renowned Stochastic Gradient Descent method. However, as opposed to standard optimization problems, we do not seek the root of a gradient. In our framework, the equilibrium condition needs to satisfy an indifference principle between randomly chosen actions. Under regularity conditions, an equilibrium is found by the SA algorithm in the limit with probability one. In other words, a long enough (single) simulation is guaranteed to converge to an equilibrium strategy. This idea draws inspiration from our understanding of how equilibrium emerges in real world systems: customers adapt their strategy based on observing past empirical performance of the system, and eventually converge to an equilibrium.

\subsection{A motivating example -- Two unobservable GI/G/1 queues in parallel}\label{sec:motivation}

The main purpose of the paper is to describe a general approximation technique for a broad range of applications, and therefore the fundamental queueing game formulated in Section \ref{sec:model}, and its extension in Section \ref{sec:observable}, are relatively abstract. Yet, to demonstrate the competency of our method we first look at a special setup that is easily explained on one hand, but is  analytically intractable on the other hand.

Consider a network, termed the \emph{system}, consisting of two FCFS queues (stations), indexed by $m\in\{1,2\}$, with dedicated servers  working in parallel, one for each queue. Service times at  station $m=1,2$ are independent and follow a general light-tailed distribution $F_m$ with mean $1/\mu_m$, such that w.l.o.g. $1/\mu_1\geq 1/\mu_2$. Potential customers arrive at the system according to a renewal process with inter-arrival time distribution $H$ and mean $1/\lambda$. The state of the system at an arbitrary point in time is characterized by a vector $X=(X^{[1]}, X^{[2]})$, with $X^{[m]}$ being the workload at station $m$. 

Upon arrival, potential customers make decisions based on a prior belief regarding the mean waiting time at each queue (e.g., relying on historic observations), yet, they cannot observe the current system state. Each customer strategically chooses one of three possible actions: (1) Join station $1$; (2) Join station $2$; or (3) Balk. Respectively, we denote the set of actions for a customer by $\mathcal{A}=\{ 1,  2,  3 \}$. A strategy  $\mathbf{p}=(p_1, p_2, p_3)\in\Delta(\mathcal{A})$ is a distribution over the action set with $p_i$ being the probability of taking action $i\in \mathcal{A}$. We assume the utility from joining is linear in the waiting time: Following the conventional notation of \cite{HH2003}, 
let $R>0$ denote the customer reward for service and $C>0$ be the customer waiting-time cost. When a customer finds the system at state $X$ and joins queue $m\in \{1,2\}$, their utility is given by $v_m = R-C\cdot (X^{[m]} + Y_m)$, where $Y_m\sim F_m$ is a r.v. representing the customer's service time at server $m$. The utility from balking, $v_3$, is normalized to 0.

Customers assume when they arrive, that the state $X$ admits its stationary distribution  (embedded at arrival instants), which is determined by the population strategy $\mathbf{p}$. To avoid diverting the discussion towards issues of stability, we assume that $\lambda < \mu_1$, implying that $X$ indeed admits a stationary distribution for every $\mathbf{p}$. A detailed discussion of the stability conditions for generalizations of this model can be found in Section~\ref{sec:parallel_GG1}.

 \textbf{Equilibrium strategy.} We are interested in characterizing an equilibrium strategy $\mathbf{p}^e = (p^e_1, p^e_2, p^e_3 )\in\Delta(\mathcal{A})$, such that under the stationary state distribution induced by $\mathbf{p}^e$, each customer's choice is utility maximizing in expectation. In other words, $\mathbf{p}^e$ is a strategy such that for each $i \in \mathcal{A}$
\[ 
p^e_i > 0 \Rightarrow i \in \argmax_{j\in \mathcal{A}} \E [v_j],
\]
where the expectation is taken jointly w.r.t the random variables $Y_1, Y_2$, and \textcolor{black}{the stationary system state at arrival instants, $X$, induced by the strategy $\mathbf{p}^e$}. The following lemma, whose proof follows standard arguments (see Section \ref{lem:GG2_existence_proof} in the appendix), establishes existence and uniqueness of the  equilibrium strategy in focus.

\begin{lemma}\label{lem:GG2_existence} Assume the inter-arrival distribution $H$ is continuous. For any pair of distributions $(F_1,F_2)$ and  parameters $R$ and $C$, there exists a unique symmetric Nash equilibrium strategy $\mathbf{p}^e$.
\end{lemma}

However, except for some special cases, an explicit characterization of this equilibrium strategy is not available, because the values of $\E [v_1]$ and $\E [v_2]$ given strategy $\mathbf{p}$ are inaccessible. Whether a strategy $\mathbf{p}=(p_1, p_2, p_3)\in\Delta(\mathcal{A})$ satisfies the equilibrium condition is a question that cannot be answered accurately, but rather approximately, based on one's ability to approximate $\E [v_1]$ and $\E [v_2]$ for the underlying strategy $\mathbf{p}$. For example, one can determine whether $\mathbf{p}$ satisfies an $\epsilon$-equilibrium condition, with $\epsilon$ depending on the ability to bound the approximation error for $\E [v_j]$.

 \textbf{Stochastic approximation.} For a given strategy $\mathbf{p}$, assume a customer arrives at the system when both queues are empty. Then we can consider this moment of arrival as an instant of system regeneration. Our method relies on simulating regeneration cycles, i.e., simulating the queueing process between two successive instants of system regeneration. During each cycle we assume that the customer strategy $\mathbf{p}$ is fixed, and keep track of the system state at every arrival instant. Let the r.v. $L$ denote the number of arrivals during a cycle (including the arrival that initiates the cycle). \textcolor{black}{Using discrete-event simulation}, we generate a single cycle consisting of $L\geq 1$ samples of the system state, $X_1, \dots, X_L$, with $X_j = (X_j^{[1]}, X_j^{[2]})$, and combine them together by defining
\[ 
\mathbf{G} = \begin{pmatrix} G_1 \\ G_2 \\ G_3 \end{pmatrix} = \sum_{j=1}^L\begin{pmatrix} R-C\cdot (X_j^{[1]} + 1/\mu_1) \\ R-C\cdot (X_j^{[2]} + 1/\mu_2) \\ 0 \end{pmatrix} .
\]
In words, we sum up the (conditional) expected utility over all arriving customers of that cycle, for each action they could have taken, had they observed the system state upon arrival. 
The motivation for the construction of $\mathbf{G}$ above is that for any  $\mathbf{p}$, the term $G_i/ \E[L]$, $i\in\{1,2,3\}$, can be regarded as an unbiased point estimator for $\E[v_i]$. Yet the values $\E[L]$ and $\E[v_i]$ themselves are assumed unknown, and obviously are not given as inputs to the algorithm. This particular form of the estimator $\mathbf{G}$ is tailored to correct for the length bias that often arises when estimating performance measures of a stochastic processes from random samples.

An application of our SA algorithm works as follows: Setting an arbitrary initial strategy $\mathbf{p}^{(1)}$, at each iteration $n=1,2, \dots$, we construct an estimator $\mathbf{G}^{(n)}$ assuming the strategy $\mathbf{p}^{(n)}$ is fixed. Given a constant $\gamma_0>0$, we then apply the following update rule:
\[ 
\mathbf{p}^{(n+1)} = \pi_{\Delta}\left(\mathbf{p}^{(n)}+\frac{\gamma_0}{n}\mathbf{G}^{(n)} \right),
\]
where $\pi_{\Delta}$ denotes the projection onto the standard (2-)simplex.
In Section \ref{sec:converg} of the paper we prove that under mild regularity of the primitives the proposed scheme converges (almost surely) to the true equilibrium, $\mathbf{p}^{(n)}\asarrow \mathbf{p}^{e}$. The parameter $\gamma_0$ is measured in the reciprocal of utility units, and clearly, the choice of its value will impact the performance of the algorithm, but we defer the discussion of this issue to later sections.

\textbf{Simulation results.} Below are results from a simulation, in which $F_1$ is $\rm{Beta}(10, 10)+0.5$, $F_2$ is $\rm{Bernoulli}(0.1) \cdot 10$, and inter-arrivals are distributed according to $\rm{Gamma}(0.1, 11)$, implying altogether that $\lambda = 10/11 <\mu_1=\mu_2=1$. The reward and cost parameters are given by $R=5$ and $C=1$. To set up the simulation we initialize $\mathbf{p}^{(1)}=(1/3, 1/3, 1/3)$ and $\gamma_0$=0.1, and run the algorithm for $N=10^6$ iterations. The convergence of the sequence $\{\mathbf{p}^{(n)}\}$ is depicted in Figure \ref{fig:multi-GG1}. After terminating at iteration $N=10^6$, the algorithm produces an output $\mathbf{p}^{(N)} = (0.525, 0.330, 0.145)$. We note that the coefficient of variation of $F_2$ is significantly higher than that of $F_1$ ($3$ compared to $0.11$), and generally speaking, higher service-time variation leads to longer waiting time and in turn to smaller utility. Thus, it is expected that in equilibrium, the negative impact of the higher variance in Server 2 will be effectively balanced by less arrivals, namely, that $p^e_1 > p^e_2$. A short numerical study described in Appendix \ref{sec:appVER} suggests that with high certainty (>99\%), our approximated solution $\mathbf{p}^{(N)}$ satisfies the criterion for an $\epsilon$-approximate Nash equilibrium (\cite{DMP2009}) for $\epsilon \leq 0.028$.

\begin{figure}[H]
\centerline{\includegraphics[scale=.5]{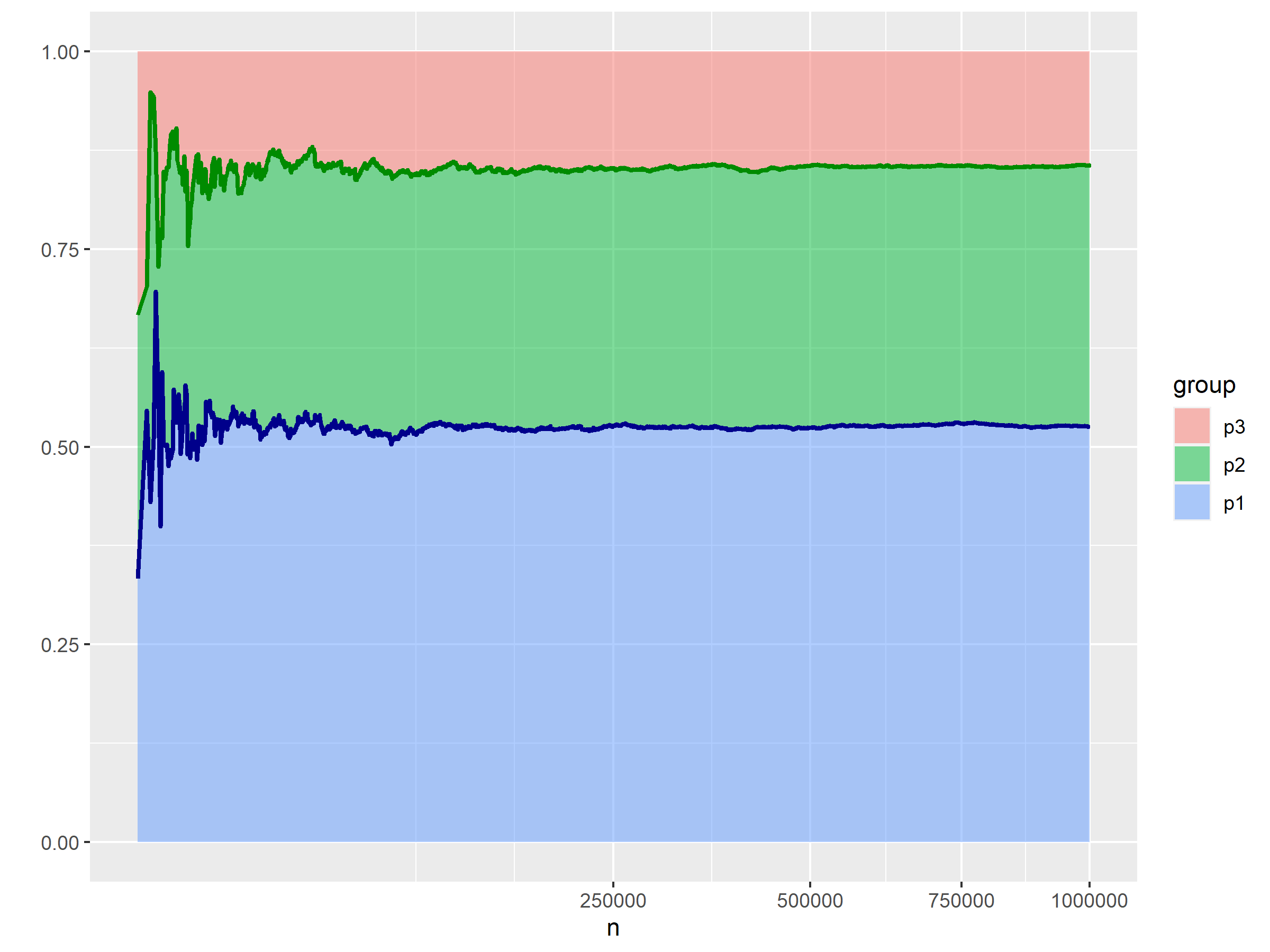}}
\caption{A stacked-area plot showing the convergence of the SA algorithm as $n\to\infty$. The coordinates of $\mathbf{p}^{(n)}=({p}^{(n)}_1, {p}^{(n)}_2, {p}^{(n)}_3)$ are plotted vs. $n$ on a square-root scale. The blue curve depicts ${p}^{(n)}_1$, and the green curve depicts ${p}^{(n)}_1+{p}^{(n)}_2$. Thus, the blue, green and orange shaded areas correspond with  ${p}^{(n)}_1$,  ${p}^{(n)}_2$, and  ${p}^{(n)}_3$, respectively. } \label{fig:multi-GG1}
\end{figure}

\subsection{Outline and main results}\label{sec:outline}
As discussed, our main result is the derivation of a simulation-based approximation technique to compute equilibria. The method is robust in the sense that it can be applied to a wide range of strategic queueing models. Below we summarize the organization of the paper and its results:

\textbf{Theoretical results.} In Section \ref{sec:model} we introduce a unified framework for unobservable queueing games with homogeneous customers and finitely many actions. As a byproduct we derive sufficient conditions for the existence of a symmetric Nash equilibrium in this class of games (Lemma \ref{lem:existence-of-equilibiurm}), which we show how to reduce down to a fixed-point problem of a real continuous function. This naturally yields an iterative (deterministic) mechanism to approximate equilibrium solutions, had customer stationary utility been known. In Section \ref{sec:SA} we explain how this stationary utility can be estimated based on samples drawn from simulated regeneration cycles. The proposed estimator serves in place of the unknown original utility, in combination with the iterative fixed-point scheme from Section \ref{sec:model}, to form the SA algorithm. The SA algorithm is shown, in Theorem \ref{thm:convergence}, to converge almost surely to an equilibrium under mild regularity conditions. Our main convergence result is accompanied with a thorough discussion of the assumptions and convergence-rate results in Sections \ref{sec:disc} and \ref{sec:conv}. In Section \ref{sec:observable} we then extend the method to account for state-dependent strategies in order to deal with observable models.

\textbf{Examples and implementation.} All along Sections \ref{sec:model} and  \ref{sec:SA} we use the canonical unobservable M/G/1 queue as an instructive tool to exemplify and generate intuition from the general model. In Section \ref{sec:applications} we discuss relevant applications, namely multiple GI/G/1 queues in parallel (Section \ref{sec:parallel_GG1}), and selective routing between queues with different buffer capacity (Section \ref{sec:CR}). We verify the sufficient convergence conditions for these applications, and use them to present refinements for run-time improvement. In Section \ref{sec:observable} we discuss the observable GI/G/1 model as an example of observable queueing games and provide results from numeric experiments. 


\subsection{Related literature}\label{sec:lit}

\textbf{Fundamentals of queueing games.} The game of joining or balking from an unobservable queue was first introduced by Edelson and Hildebrand \cite{EH1975}. They consider an M/M/1 queue with strategic customers wishing to maximize a utility function comprised of a fixed reward and a linear cost for waiting. 
As a simple formula is readily available for the expected waiting time in an M/M/1 queue, the equilibrium joining probability can be computed explicitly. This model easily extends to general service times \cite{BS1983}. Numerous other queueing games have been studied over the years and we refer interested readers to the books by Hassin and Haviv \cite{HH2003} and Hassin \cite{H2016} for complete overview and surveys.
Recently, Haviv and Oz \cite{HO2021} suggested a unified approach to formulate unobservable queueing games by analyzing the interactions between the (random) set of customers who visit the system during the same busy period. Their formulation is related to our work in that our SA algorithm updates the strategy in between busy periods. Broadly speaking, both works rely on the understanding that, given the strategy, all the statistical properties of the underlying queueing process are encapsulated in a single regeneration cycle.

\textbf{Stochastic optimization in queueing systems.} A comprehensive overview of the theory of stochastic approximations, which was initiated by the seminal work of Robbins and Monro \cite{RM1951}, can be found in Kushner and Yin \cite{KY2003} and in Borkar \cite{B2008}. An adaptation of the Robbins-Monro algorithm for the purpose of optimization was first presented by Kiefer and Wolfowitz \cite{KW1952}, who laid the groundwork for the rapid development of stochastic optimization techniques. Among their notably wide range of applications, these techniques have been applied to problems of optimizing the steady-state performance of queueing systems, to which we give emphasis next. 

The computation of an optimal service rate for a single-server queue using  stochastic approximations was initially studied by Suri and Zazanis \cite{SZ1988} where the gradient estimation problem is addressed via the method of Infinite Perturbation Analysis. Empirical experiments for the M/M/1 queue (building on the machinery developed in \cite{SZ1988}) were carried  by Suri and Leung \cite{SL1989}. Fu \cite{F1990} later extended this theory to the GI/G/1 queue, exploiting the regenerative structure underlying it. Further convergence properties are given for the GI/G/1 model in L'Ecuyer and Glynn \cite{LG1994}, for different estimators of the gradient such as finite-difference and likelihood-ratio estimators. A numerical study of the techniques discussed in \cite{LG1994} for the M/M/1 case can be found in L'Ecuyer et. al \cite{LGG1994}. A review of optimization techniques using simulation, including applications to GI/G/1 queues, is provided by Fu \cite{F1994}, and a unified framework for stochastic optimization of the steady-state performance in regenerative systems is established in Andradóttir \cite{A1996}. In all these settings the objective function is continuously differentiable in the parameter and the focus is on estimating the gradient towards the ultimate goal of computing its root. Our setting is different in that we do not focus on optimizing a given objective, but rather on identifying Nash equilibrium solutions -- a problem that is not dealt with in this line of literature. 

Recently, Chen et. al \cite{CLH2020} proposed an SA algorithm for jointly optimizing price and capacity in a GI/G/1 queue, and formulated conditions for convergence to the optimal solution. 
Chen et. al  \cite{CLH2020} further perform thorough analysis of the regret, drawing connections between stochastic approximations and reinforcement learning in the context of revenue management in queues. Reinforcement learning has been employed in recent studies to approach optimal control policies as a means of optimizing steady-state performance in queueing networks, by Liu et. al \cite{LXM2019} and by Dai and Gluzman \cite{DG2021}.

\textbf{Simulation and learning in queueing games.} \textcolor{black}{Stochastic approximation and adaptive simulation algorithms have been discussed in different branches of Game Theory. In particualr, some SA schemes were shown to converge to Nash equilibrium in certain classes of stochastic games with repeated interaction and decision making (see, for example, \cite{BH1999}, and \cite[Ch.~10.4]{B2008}).} However, in the area of strategic queueing, literature involving simulation-based methods to approach equilibrium solutions is rather scarce. Altman and Shimkin \cite{AS1998} explicitly analyze an observable processor-sharing queue with strategic customers, exponential services, and linear waiting cost. They  suggest an ad-hoc simulation-based learning algorithm, to demonstrate the convergence of customers' iterative decision-making process to a symmetric Nash equilibrium. This was extended to a model with heterogeneous customers in Ben Shahar et al. \cite{BOS2000}. Buche and Kushner \cite{BK2000} verify convergence of the algorithm presented in \cite{AS1998} to a Nash equilibrium and explain how it extends to non-linear cost functions. In the context of customers strategically timing their arrivals to a transient queueing processes, Sakuma et. al \cite{SMF2020} construct a heuristic dynamic algorithm to approximate the Nash equilibrium arrival strategy. We note, however, that the framework of strategic arrival scheduling in queues is significantly different than ours -- in the former, the focus is on transient queueing processes, and additionally, customer action space is usually assumed continuous (see more details in \cite{HR2021}).

Indirectly related to our work is the study of best-response dynamics; iterative updating of strategies by maximizing utility given the strategy in the previous iteration. Best-response algorithms are known to converge in some queueing games, such as S-modular games as defined by Yao \cite{Y1995}. For example, the S-modular framework was applied to a problem of decentralized control of a wireless network by Altman and Altman \cite{AA2003}. However, in most settings the known conditions required for convergence of best-response dynamics are not satisfied, and typically cannot even be tested due to the intractability of the stationary performance measures of the system. 

\textbf{Potential applications.} In the queueing game we introduce, customers are modeled as short-lived entities who arrive according to a renewal process to a general, regenerative system, and choose one out of finitely many actions, with the objective of maximizing their utility. Ever since Naor's \cite{N1969} seminal work, this modeling structure has been accepted as a standard approach to study customer behavior in service systems, and has been applied widely in the operations-management literature. While in \cite{N1969} the existence of a dominant  strategy deems the equilibrium analysis trivial, here we focus our attention on cases where the equilibrium search problem is non-degenerated. Below is a noninclusive list of few representative examples of such models.

Among the classical problem themes in the study of strategic queueing are join-or-balk decisions in observable and unobservable queues \cite{EH1975, BE2007, DTVW2008, GH2011, K2011}; decentralized selective routing in queueing networks \cite{BS1983, H1996, PK2004, HS2017}; paying to reduce wait via priority or premium service classes \cite{AY1974, HH1997, WCW2019, CWY2020}; and provision and  acquisition of information in queues \cite{GZ2007, XH2013, HRG2017, HS2021, HRG2021}.
Concrete applications that motivate the study of such models traditionally include communication and computer information services \cite{M1985, AS1998, MZ2003, JMMZ2012} and transportation networks \cite{MEK2014, MCK2017}. Some very recent business applications that are increasingly emerging in this line of literature also extend across ride-sharing platforms \cite{T2018, JRG2021, HWW2020, CFMNY2020}; food delivery and curbside pickup in restaurants \cite{CHW2022, SLY2020}; as well as omnichannel services \cite{BCL2022, RGY2020, GBL2020, CFMNY2020}.

In many of these papers, solving for customer equilibrium is a preliminary, yet a crucial step towards answering questions concerning pricing, policy selection, capacity planning and the like. Due to the intricacy of the equilibrium problem, the majority of these examples aim to capture the main characteristics of the equilibrium behavior through parsimonious models, on which direct analysis can be carried out. Two exceptions here are \cite{XH2013}, who adopt a mean-field approach, and \cite{MZ2003} who resort to heavy-traffic approximations (in the Halfin-Whitt regime). In \cite{MZ2003}, the authors further highlight the need for exhaustive simulations in tackling the pre-limit analog of their model.
The framework we study is fairly general, covers both observable and unobservable games, and is easily extended to multiple (finitely many) customer types (see, e.g., Appendix \ref{sec:heterogeneous-customers}). It is rich enough to capture the customer-behavior model employed in the examples listed above (with the restriction to finite heterogeneity types when considered). Thus, our results potentially offer a useful method to approach many of these examples' variants and extensions that are not amenable to direct analysis.

\subsection{Notation}\label{sec:notation}
In the paper, the domain of interest over which we define our problem is $\mathbb{R}^k$, $k>1$. Thus, we designate vectors in $\mathbb{R}^k$ by bold letters, and so we do for functions whose image is in $\mathbb{R}^k$. We denote by $\mathbf{e}_i, i=1,\dots,k$, the standard-basis unit vector with $1$ in its $i$-th coordinate, and denote by $\mathbf{e}$ the all-1 $k$-dimensional vector. By default, we define vectors as column vectors and use $'$ to denote their transpose. The non-negative real half line is denoted by $\mathbb{R}_+$. Given a function $g$ with a domain in $\mathbb{R}$ we denote its limit from the left at $t$ by $g(t-)=\lim_{s\uparrow t}g(s)$. For a real vector space, $\Vert \cdot \Vert$ is used by default to refer to the $L^2$-norm, and for $p\in[1,\infty)\cup\{\infty\}$, $\Vert \cdot \Vert_p$ denotes the $L^p$-norm. We further use $\Vert \cdot \Vert_0$ to denote the $L^0$-``norm'', i.e., the number of non-zero elements. For any non-empty set $\mathcal{S}$ and a point $x$ both defined in a real vector space we denote by $\pi_\mathcal{S}(x)$ the projection of $x$ onto $\mathcal{S}$, namely, $\pi_\mathcal{S}(x)=\argmin_{y\in\mathcal{S}} \Vert x-y\Vert$. With a slight abuse of notation, when $\pi_\mathcal{S}(x)$ is a singleton we will refer to it as a point in $\mathcal{S}$.
For $k$ being the dimension of the problem, we denote by $\Delta$ the $(k-1)$-simplex (i.e., the unit simplex with $k$ vertices), and its relative interior by $\Delta^\mathrm{o}$; $\Delta^\mathrm{o}=\{\mathbf{x}+\theta\mathbf{y} \mid \mathbf{x,y}\in\Delta, 0<\theta<1\}$.  
Given a probability space we denote by $\mathit{1}(A)$ the r.v. representing the indicator of an event $A$. For a random variable $X$, $\sigma(X)$ is the $\sigma$-algebra generated by $X$. The symbols $\asarrow$, $\parrow$ and $\darrow$ indicate convergence of a sequence of random variables almost surely, in probability, and in distribution, respectively. Almost-sure convergence of a random sequence $X_n$ to a set $\mathcal{S}$, denoted by $X_n\asarrow \mathcal{S}$, implies that $\inf_{x\in\mathcal{S}}\Vert X_n-x\Vert\asarrow 0$ as $n\to\infty$.
When comparing random variables, $\leq_{\rm st}$ symbolizes inequality in first-order stochastic dominance.

\section{Model and preliminaries}\label{sec:model}
We begin by introducing a general formulation of an unobservable queueing game. By unobservable, we mean that each customer commits to a specific action (which is possibly chosen randomly) prior to observing any information about the system's state. However, we treat the concept of an action here in rather general abstraction, which allows us to cover many observable models as well
(see discussion in \ref{sec:comment-obs}).
Later, in Section \ref{sec:observable}, we will allow customers to rely on state-information in their decisions, thereby extending our current formulation, providing a more refined treatment of observable queueing games. \textcolor{black}To ease the exposition, we assume throughout the paper that customers are homogeneous, however the general methods can be generalized to deal with finitely many heterogeneity types of customers through only incremental changes. {An instructive example of a model with multiple types of customers is given in Appendix \ref{sec:heterogeneous-customers}.}

 \textbf{The model.} Consider a service system in which the state at time $t$ represents some (possibly multi-dimensional) buffer content, taking values in a state space $\mathcal{X}\subset \mathbb{R}^d_+$. Potential customers arrive at the system according to a renewal process with inter-arrivals $\{A_n\}_{n\geq 1}$ (i.e., $A_n$ are iid), where we interpret $A_n$ as the time between the arrivals of the $(n-1)$-st and the $n$-th customers. The arrival epoch of the $n$-th customer, $n\geq 1$, is therefore given by $T_n=\sum_{i=1}^{n} A_i$. 

Every arriving customer chooses one out of $k\geq 2$ possible actions, with the action set denoted by $\mathcal{A}=\{a_1, \dots, a_k\}$. Thus, the $(k-1)$-simplex $\Delta$ represents the set of possible strategies (namely, distributions over $\mathcal{A}$). Given a strategy $\mathbf{p}\in \Delta$, we denote its $i$-th coordinate by $p_i$ which is the probability assigned to action $a_i$. 

We denote by $\{X(t; \mathbf{p})\}_{t\geq 0}$ the stochastic process representing the state of the system at time $t$ when the strategy employed by customers is $\mathbf{p}$. By convention, $\{X(t; \mathbf{p})\}_{t\geq 0}$ is assumed to be right-continuous with left-hand limits. The evolution of $X(t; \mathbf{p})$ depends on the strategy $\mathbf{p}$ played by the customers, however, in the general setup we suppress the specifics of how new arrivals and their corresponding actions change the state of the buffer content because that depends on the specific application. 

\begin{example}\label{example:MG1}
To generate intuition, one can think of the canonical Unobservable M/G/1 queue (see \cite{BS1983}), in which customers choose whether or not to join a single-server queue, thus $k=2$. For concreteness, we assume that $\lambda < \mu$, where $\lambda$ is the Poisson arrival rate and $1/\mu$ is the mean service time. In this example, a strategy is characterized by $\mathbf{p}=(p_1, p_2)=(p,1-p) \in \Delta$ for some $p\in[0,1]$, prescribing the joining and balking probabilities, $p$ and $1-p$, respectively. The content process $X(t; \mathbf{p})$ corresponds to the workload (or virtual waiting time), hence when a customer arrives at the system, with probability $p$ they join the queue and add a random job size (with mean $1/\mu$) to the workload.   \hfill $\diamond$ 
\end{example}

Let $X_n(\mathbf{p})=X(T_n-;\mathbf{p})$ denote the system state just before the $n$-th arrival given that customers adopt the strategy $\mathbf{p}$, and further assume that $X_1(\mathbf{p})=0^d$. In other words the system starts empty, so that the first customer arrives (at time $T_1$) to an empty system. The strategy $\mathbf{p}\in \Delta$ defines a probability measure $\mathbb{P}_\mathbf{p}$ for $\{X_n(\mathbf{p})\}_{n\geq 1}$. Expectation with respect to the measure $\mathbb{P}_\mathbf{p}$ is denoted by $\E_\mathbf{p}$. Define the \emph{cycle-length} r.v. $L(\mathbf{p})=\inf\{n\geq 1\mid  X_{n+1}(\mathbf{p})=0^d\}$. With the assumption $X_1(\mathbf{p})=0^d$, $L(\mathbf{p})$ is a r.v. describing the number of arrivals during a typical regenerative cycle, provided customers play according to $\mathbf{p}$. Note that separate cycle lengths are iid, that is, for any $n_1, n_2, \dots$ satisfying $X_{n_i}(\mathbf{p})=0^d$, $i=1,2,\dots$, the sequence $\inf\{n\geq 1\mid X_{n_i+n+1}(\mathbf{p})=0^d\}$ for $i=1,2,\dots$ consists of iid random variables. 
Let $\ell(\mathbf{p})=\E_\mathbf{p}L(\mathbf{p})$ denote the mean cycle length when the strategy is given by $\mathbf{p}$, and let $\ell^r(\mathbf{p})=\E_\mathbf{p}L^r(\mathbf{p})$ denote its $r$-th moment. By embedding our system at arrival epochs we impose that $L(\mathbf{p})\geq 1$ with probability 1, therefore $\ell(\mathbf{p})\geq 1$ for every $\mathbf{p}\in\Delta$. 

Considering the strategy $\mathbf{p}$ as given, it is known that if the first moment of the cycle length is finite, i.e., if $\ell(\mathbf{p})<\infty$, then $\{X_n(\mathbf{p})\}_{n\geq 1}$ is positive Harris recurrent; for further details see \cite[Ch.~VI]{A2003}. This implies $X_n(\mathbf{p})\darrow X(\mathbf{p})$ as $n\to\infty$, where $X(\mathbf{p})$ is a random variable corresponding to the stationary distribution at arrival times. Note that $X(\mathbf{p})$ may be different from the stationary, \emph{time-averaged} distribution of $X(t;\mathbf{p})$. The conditions for convergence of the algorithm presented below  demand that the cycle length has a finite second moment, implying that its first moment, $\ell(\mathbf{p})$, is also finite.

The utility of a customer depends on the system state at the time of their arrival, their chosen action $a_i$ and possibly some random outcome. Formally, given a \emph{realization} of the state $x\in\mathcal{X}$ and a realization of the random outcome $y$, we let $v_i(x, y)$, $i=1,\dots,k$ denote the value associated with action $a_i$. We define
\begin{equation}\label{eq:vXY}
\mathbf{v}(x, y)=\big( v_1(x, y), \dots, v_k(x, y) \big).
\end{equation}
The function $\mathbf{v}$ is considered a model's primitive.
Note that the random outcome may depend on the action chosen, in which case $y$ can be modeled as a vector $y\in\mathbb{R}^k$ with each coordinate corresponding to an action $a_i\in\mathcal{A}$. The dimension of $y$ is in fact irrelevant to our analysis and for the sake of simplicity $y$ can be thought of as single valued. The function $\mathbf{v}$, as well as the distribution of the random outcome, are allowed to depend on the strategy $\mathbf{p}$, but for brevity we suppress this dependence in the notation.

Our general interest is in treating the state and the random outcome as (possibly dependent) random variables, say $X$ and $Y$. The value vector $\mathbf{v}(X, Y)$ then constitutes a $k$-dimensional random variable. We introduce the vector of expected stationary values as
\begin{equation}\label{eq:up}
\mathbf{u}(\mathbf{p})=\E_\mathbf{p} \bigg[ \mathbf{v}\big(X(\mathbf{p}), Y\big)  \bigg],
\end{equation}
where we recall that $X(\mathbf{p})$ is the random stationary state given $\mathbf{p}$. We assume throughout that a stationary distribution exists for every $\mathbf{p}\in\Delta$, and furthermore that the function $\mathbf{v}(X, Y)$ is integrable with respect to the measure corresponding to that stationary distribution. This ensures that $\mathbf{u(p)}$ is well defined given $\mathbf{p}$, thus, it describes a function $\mathbf{u}:\Delta\to\mathbb{R}^k$; For each $i$, $u_i(\mathbf{p})$ takes the interpretation of the mean utility of some ``controlled'' customer who arrives at a stationary system and is told to play $a_i$. To clarify, we highlight that $\mathbf{u(p)}$ does \emph{not} represent the (single-valued) average utility obtained when customers employ the strategy $\mathbf{p}$ -- the latter can be expressed as $\mathbf{u(p)}'\mathbf{p}$.

\setcounter{example}{0}
\begin{example}\textbf{{\rm (Continued)}}
Consider once again the example of the unobservable M/G/1, with arrival and service rates $\lambda$ and $\mu$, respectively, satisfying $\lambda < \mu$. As explained, a strategy is given by $\mathbf{p} = (p, 1-p)$ where $p\in[0,1]$ depicts the joining probability, thus, we interpret $X(\mathbf{p})$ as the stationary virtual workload at arrival epochs associated with customers joining at rate $p\lambda$, and $Y$ is the service time of a tagged arrival. The possible actions are joining ($a_1$) or balking ($a_2$), hence $v_1(x,y)=R-C\cdot(x+y)$ for $R$ and $C$ being the reward from service and cost for unit of delay, and $v_2(x,y)=0$. In addition, letting $w(\cdot)$ denote the mean virtual workload in the system as a function of the arrival rate, we have $\E_\mathbf{p} [ X(\mathbf{p}) ]=w(p\lambda)$ and therefore  $\mathbf{u}(\mathbf{p}) = (R-C\cdot(w(p\lambda)+\mu^{-1}), 0)$.  \hfill$\diamond$
\end{example}

As in Example 1, it is oftentimes the case that $Y$ and $X(\mathbf{p})$ are independent. However, in general the distribution of $Y$ may depend on the observed state, $x$, as well as on the strategy of others, $\mathbf{p}$. One classic example is a processor-sharing system in which the waiting time depends on the strategy of future arrivals (see \cite{AS1998}). 

 \textbf{Equilibrium strategy.} Next, we define an equilibrium, the desired solution concept at the center of attention in this work.

\begin{definition}
The best-response set for a strategy $\mathbf{p}\in\Delta$ is the set 
\[\mathcal{BR}(\mathbf{p})=\argmax_{\mathbf{q}\in\Delta} \mathbf{u(p)}'\mathbf{q}.\]
\end{definition}
For a strategy $\mathbf{p}\in\Delta$, each of the elements of $\mathcal{BR}(\mathbf{p})$ is commonly termed a \emph{best-response strategy} (or \emph{best response} in short) for $\mathbf{p}$. Thus, $\mathbf{p}$ is a symmetric equilibrium if it is a best response to itself, or synonymously, if it is a fixed point of $\mathcal{BR}$, when the latter is viewed as a set-valued map, $\mathcal{BR}:\Delta\to 2^\Delta$. Hence,
\begin{definition}\label{def:SNE}
A Symmetric Nash Equilibrium strategy is a strategy $\mathbf{p}\in\Delta$ such that 
\[\mathbf{p}\in \mathcal{BR}(\mathbf{p}) = \argmax_{\mathbf{q}\in\Delta} \mathbf{u(p)}'\mathbf{q}.\]
\end{definition}
 \noindent Nonetheless, describing an equilibrium as a fixed point of the best-response function $\mathcal{BR}$ does not yield any simple method to compute it in general. This is because $\mathcal{BR}$ is a set-valued mapping in a continuous space, and is naturally hard to work with. Moreover, even if $\mathcal{BR}$ is a singleton almost everywhere in the domain $\Delta$ and continuous at every such point, it is often the case that the fixed point $\mathbf{p}^e\in\mathcal{BR}(\mathbf{p}^e)$ is a point of jump discontinuity, in the sense that the limits approaching $\mathbf{p}^e$ from different directions do not agree.
\setcounter{example}{0}
\begin{example}\textbf{{\rm (Continued)}}
In the unobservable M/G/1, recall that a strategy is given by $\mathbf{p} = (p, 1-p)$ with $p\in[0,1]$ and that $\E_\mathbf{p} [ X(\mathbf{p}) ]=w(p\lambda)$ which is the mean virtual workload when the arrival rate is $p\lambda$. Suppose that $R,C$ and $\mu$ are such that $0< R/C -1/\mu < w(\lambda)$. Then there exists a unique equilibrium joining probability $p^e \in (0,1)$, which is characterized by the unique solution to the equation $R-C\cdot(w(p^e \lambda)+1/\mu)=0$, and the corresponding equilibrium strategy is $\mathbf{p}^e=(p^e, 1-p^e)$. The best-response function (plotted in Figure~\ref{fig:MG1_BR}) for a strategy $\mathbf{p} = (p, 1-p)$ takes the form:
\[
\mathcal{BR}(\mathbf{p}) = \begin{cases} \mathbf{e}_1 \quad &\mbox{if } p<p^e;\\
\Delta &\mbox{if } p= p^e;\\
\mathbf{e}_2 &\mbox{if } p>p^e.
\end{cases}
\]
Denoting $\tilde{\mathbf{e}}=\mathbf{e}_1-\mathbf{e}_2$, it can be seen that for any $\epsilon>0$, 
\[\Vert \mathcal{BR}(\mathbf{p}^e+\epsilon\tilde{\mathbf{e}}) -  \mathcal{BR}(\mathbf{p}^e-\epsilon\tilde{\mathbf{e}}) \Vert  = \Vert \mathbf{e}_1-\mathbf{e}_2 \Vert =\Vert \tilde{\mathbf{e}} \Vert =\sqrt{2}.\] 
\hfill$\diamond$ 
\end{example}
Example \ref{example:MG1} shows that the set-valued mapping $\mathcal{BR}$ in general is not lower-hemicontiuous. Thus, even in the fundamental unobservable M/G/1 setting, it is not obvious how to approach the equilibrium strategy based on fixed-point iterations.
To overcome this difficulty, we define next a modified version of the best-response function, which is a vector-valued function, 
$\mathbf{f}:\Delta \to\Delta $, expressed as:
\begin{equation}
\mathbf{f}(\mathbf{p})= \pi_\Delta\big( \mathbf{p}+\mathbf{u(p)}\big), \label{EQN:f-def}
\end{equation} 
recalling that $\pi_\Delta(\mathbf{x})$ denotes the projection of $\mathbf{x}\in\mathbb{R}^k$ onto the simplex $\Delta$. This vector-valued function can replace the raw definition of $\mathcal{BR}$ for our purpose, and in addition, under appropriate conditions on the primitives, also possesses desirable attributes like continuity and smoothness.

The intuition behind the definition of $\mathbf{f}$ is that $\mathbf{f}(\mathbf{p})$ maps $\mathbf{p}$ to a strategy obtained by deviating from $\mathbf{p}$ in the direction of one of its (actual) best responses. This can be noticed observing that $\mathbf{p}+\mathbf{u(p)}$ is the unique solution of the following (convex) optimization problem:
\[ 
\max_{\mathbf{q}\in\mathbb{R}^k} \left\{\mathbf{u(p)}'\mathbf{q}- \frac{1}{2}\Vert \mathbf{p-q}\Vert^2 \right\}.
\]
Whenever the best response for $\mathbf{p}$ is not unique, $\mathbf{f(p)}$ ``pushes'' $\mathbf{p}$ in the direction of the best response that is the closest to $\mathbf{p}$ in the standard euclidean sense (note that the projection of $\mathbf{p}$ onto $\mathcal{BR}(\mathbf{p})$ is unique because the latter is a convex polyhedron in $\mathbb{R}^k$). 
It is intuitive therefore that a strategy $\mathbf{p}$ is a symmetric equilibrium if and only if it is a fixed point of $\mathbf{f}$.  This statement is rigorously formulated in the next lemma, whose proof appears in Appendix \ref{lem:equilibrium-equivalence-proof}:
\begin{lemma}\label{lem:equilibrium-equivalence} 
A strategy $\mathbf{p}\in\Delta$ is a Symmetric Nash Equilibrium, i.e., $\mathbf{p} \in \mathcal{BR}(\mathbf{p})$, if and only if it satisfies $\mathbf{p}=\mathbf{f}(\mathbf{p})$.
\end{lemma}

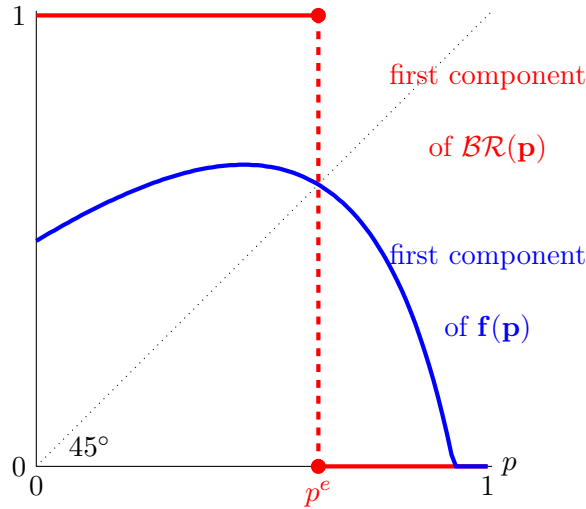
\begin{figure}[H]
\centering
\begin{tikzpicture}[xscale=6,yscale=6]
  \def\xmin{0}
  \def\xmax{1.01}
  \def\ymin{0}
  \def\ymax{1.01}
    \draw[-] (\xmin,\ymin) -- (\xmax,\ymin) node[right] {$p$} ;
    \draw[-] (\xmin,\ymin) -- (\xmin,\ymax);
    \draw[dotted] (\xmin,\ymin) -- (\xmax,\ymax);
    \node[draw=none, above right ] at (\xmin+.05,\ymin) {$45^\circ$};
    \foreach \x in {0,1}
    \node at (\x,\ymin) [below] {\x};
    \foreach \y in {0,1}
    \node at (\xmin,\y) [left] {\y};
    \draw[red, ultra thick] (0,1)--(0.625,1) node[minimum size=0.2cm,inner sep=0pt,circle,red,fill]{};
    \draw[red, ultra thick, dashed] (0.625,1)--(0.625,0) node[minimum size=0.2cm,inner sep=0pt,circle,red,fill]{};
    \draw[red, ultra thick] (0.625,0)--(1,0);
    \node[draw=none, below, red] at (0.625,-0.01) {$p^e$};
    \node[draw=none,right,color=red] at (0.725,0.8)  { \begin{tabular}{c} first component \\ of $\mathcal{BR}(\mathbf{p})$\end{tabular} };
    
        \node[draw=none,right,color=blue] at (0.725,0.4)  { \begin{tabular}{c} first component \\ of $\mathbf{f}(\mathbf{p})$\end{tabular} };
    
    \draw[blue,  ultra thick] ( 0 , 0.5 )-- ( 0.01 , 0.506 )-- ( 0.02 , 0.512 )-- ( 0.03 , 0.518 )-- ( 0.04 , 0.523 )-- ( 0.05 , 0.529 )-- ( 0.06 , 0.535 )-- ( 0.07 , 0.54 )-- ( 0.08 , 0.546 )-- ( 0.09 , 0.551 )-- ( 0.1 , 0.557 )-- ( 0.11 , 0.562 )-- ( 0.12 , 0.567 )-- ( 0.13 , 0.572 )-- ( 0.14 , 0.577 )-- ( 0.15 , 0.582 )-- ( 0.16 , 0.587 )-- ( 0.17 , 0.591 )-- ( 0.18 , 0.596 )-- ( 0.19 , 0.6 )-- ( 0.2 , 0.605 )-- ( 0.21 , 0.609 )-- ( 0.22 , 0.613 )-- ( 0.23 , 0.617 )-- ( 0.24 , 0.621 )-- ( 0.25 , 0.625 )-- ( 0.26 , 0.629 )-- ( 0.27 , 0.632 )-- ( 0.28 , 0.636 )-- ( 0.29 , 0.639 )-- ( 0.3 , 0.642 )-- ( 0.31 , 0.645 )-- ( 0.32 , 0.648 )-- ( 0.33 , 0.651 )-- ( 0.34 , 0.653 )-- ( 0.35 , 0.656 )-- ( 0.36 , 0.658 )-- ( 0.37 , 0.66 )-- ( 0.38 , 0.662 )-- ( 0.39 , 0.663 )-- ( 0.4 , 0.665 )-- ( 0.41 , 0.666 )-- ( 0.42 , 0.667 )-- ( 0.43 , 0.668 )-- ( 0.44 , 0.668 )-- ( 0.45 , 0.669 )-- ( 0.46 , 0.669 )-- ( 0.47 , 0.669 )-- ( 0.48 , 0.668 )-- ( 0.49 , 0.668 )-- ( 0.5 , 0.667 )-- ( 0.51 , 0.665 )-- ( 0.52 , 0.664 )-- ( 0.53 , 0.662 )-- ( 0.54 , 0.66 )-- ( 0.55 , 0.657 )-- ( 0.56 , 0.654 )-- ( 0.57 , 0.651 )-- ( 0.58 , 0.647 )-- ( 0.59 , 0.643 )-- ( 0.6 , 0.638 )-- ( 0.61 , 0.633 )-- ( 0.62 , 0.628 )-- ( 0.63 , 0.622 )-- ( 0.64 , 0.615 )-- ( 0.65 , 0.608 )-- ( 0.66 , 0.601 )-- ( 0.67 , 0.592 )-- ( 0.68 , 0.584 )-- ( 0.69 , 0.574 )-- ( 0.7 , 0.564 )-- ( 0.71 , 0.553 )-- ( 0.72 , 0.541 )-- ( 0.73 , 0.528 )-- ( 0.74 , 0.515 )-- ( 0.75 , 0.5 )-- ( 0.76 , 0.484 )-- ( 0.77 , 0.468 )-- ( 0.78 , 0.45 )-- ( 0.79 , 0.431 )-- ( 0.8 , 0.411 )-- ( 0.81 , 0.39 )-- ( 0.82 , 0.367 )-- ( 0.83 , 0.342 )-- ( 0.84 , 0.316 )-- ( 0.85 , 0.287 )-- ( 0.86 , 0.257 )-- ( 0.87 , 0.225 )-- ( 0.88 , 0.191 )-- ( 0.89 , 0.154 )-- ( 0.9 , 0.114 )-- ( 0.91 , 0.072 )-- ( 0.92 , 0.026 )-- ( 0.93 , 0 )-- ( 0.94 , 0 )-- ( 0.95 , 0 )-- ( 0.96 , 0 )-- ( 0.97 , 0 )-- ( 0.98 , 0 )-- ( 0.99 , 0 )-- ( 1 , 0 );
\end{tikzpicture}
\caption{The first component of the best-response function  $\mathcal{BR}(\mathbf{p})$ (red), and the first component of the best-response surrogate  $\mathbf{f(p)}$ (blue), for the unobservable M/G/1 model. At $p^e$ customers are indifferent between joining and balking, thus, every joining probability $p\in[0,1]$ constitutes a best-response strategy, $(p, 1-p)\in\mathcal{BR}(\mathbf{p}^e)$. The red curve's left and right limits at $p^e$ do not coincide. The blue curve is a single-valued continuous function over $[0,1]$ and is smooth at $p^e$. }\label{fig:pp_BR_a}
\label{fig:MG1_BR}
\end{figure}

In Figure \ref{fig:MG1_BR}, both the best-response function, $\mathcal{BR}$, and its surrogate, $\mathbf{f}$, are illustrated for the unobservable M/G/1 model of Example \ref{example:MG1}. It can be seen that both functions admit a unique fixed point at the equilibrium strategy. However, $\mathcal{BR}$ is discontinuous (in the sense of lower-hemicontinuity) at that point, while $\mathbf{f(p)}$ is indeed continuous. 

\begin{remark}\label{rem:choice-of-ell}
The result of Lemma \ref{lem:equilibrium-equivalence} would still prevail if we extended
our definition of $\mathbf{f(p)}$ to $\pi_\Delta\big( \mathbf{p}+h(\mathbf{p})\mathbf{u(p)}\big)$,
for any positive real function $h:\Delta\to\mathbb{R}$. In using our suggested equilibrium-approximation scheme, it is useful to choose a function $h$ that is bounded away from 0.
\end{remark}

 \textbf{Iterative equilibrium-approximation scheme.} Equipped with Lemma \ref{lem:equilibrium-equivalence} we focus our interest on characterizing a solution to $\mathbf{p}=\mathbf{f}(\mathbf{p})$. To ensure that such a fixed point exists, we assume that $\mathbf{u}$ is continuous, in which case we have, as an implication of the maximum theorem, that $\mathbf{f}$ is continuous as well. We can then state the following lemma.

\begin{lemma}\label{lem:existence-of-equilibiurm} Assume that $\mathbf{u(p)}$ is continuous for all $\mathbf{p}\in\Delta$. Then a symmetric equilibrium strategy $\mathbf{p}^e\in\Delta$ exists, and this strategy satisfies $\mathbf{p}^e = \mathbf{f}(\mathbf{p}^e)$.
\end{lemma}
Lemma \ref{lem:existence-of-equilibiurm} follows directly from Lemma \ref{lem:equilibrium-equivalence} and the continuity of $\mathbf{f}$ by applying Brouwer's fixed-point theorem. As a consequence of Lemma \ref{lem:existence-of-equilibiurm}, when $\mathbf{u(p)}$ is accessible and can be evaluated directly, then a simple fixed-point iteration algorithm can be used to approximate an equilibrium; To this end,
we think of $\mathbf{u(p)}$ as the vector that determines the direction of the step at each iteration, 
giving rise to the following (deterministic) adaptive-iteration scheme:
\begin{equation} \label{eq:FI_iteration}
{\mathbf{p}}^{(n+1)} = \pi_\Delta \left( {\mathbf{p}}^{(n)}+\gamma_n \mathbf{u}({\mathbf{p}}^{(n)}) \right), 
\end{equation}
where $\{\gamma_n\}_{n\geq 1}$ is a positive sequence called the \textit{step-size sequence}, and satisfies $\sum_{n= 1}^\infty \gamma_n=\infty$ (in applications, the requirement $\gamma_n\to0$ is further imposed for regulation). Assuming $\mathbf{u}(\cdot)$ is continuous, any limit of the sequence in \eqref{eq:FI_iteration} corresponds to an equilibrium strategy. 
Yet, the latter statement alone does not make any conclusion about whether or not convergence takes place, and in some cases the iterate may indeed end up oscillating between non-equilibrium strategies. 
Further regulating assumptions on $\mathbf{u}(\cdot)$ are needed (for example, Assumption \ref{assum:finite_u} below) to guarantee convergence of the algorithm.
General conditions under which \eqref{eq:FI_iteration} converges is a rather technical issue that exceeds the scope of our discussion here; some relevant information is provided in  Appendix~\ref{sec:app_convergence-proof}.
\begin{remark}
\textcolor{black}{Consider a point $\mathbf{p}\in\Delta$, and let $\gamma>0$ be an arbitrary step size. Define
$$\gamma\mathbf{z} = \pi_\Delta \left( \mathbf{p}+ \gamma\mathbf{u}(\mathbf{p}) \right) - (\mathbf{p}+ \gamma\mathbf{u}(\mathbf{p})),$$ 
which is the shortest vector needed to bring $\mathbf{p}+ \gamma\mathbf{u}(\mathbf{p})$ back into $\Delta$. Starting at $\mathbf{p}$, applying one iteration of \eqref{eq:FI_iteration} with the step size $\gamma$ results in the updated iterate $\mathbf{p} + \gamma\mathbf{u}(\mathbf{p})+\gamma\mathbf{z}$.
Thus, $\mathbf{p}$ is a fixed point for (\ref{eq:FI_iteration}) if (and only if) $\mathbf{z}=-\mathbf{u}(\mathbf{p})$, namely, 
the shortest trajectory from $\mathbf{p}+\gamma\mathbf{u}(\mathbf{p})$ back to the simplex has to be along the opposite direction of $\mathbf{u}(\mathbf{p})$. }
When $\mathbf{p} \in \Delta^\mathrm{o}$, i.e., $\mathbf{p}$ is in the (relative) interior of $\Delta$, the latter condition means that
$\mathbf{u(p})$ has to be orthogonal to the hyperplane spanned by $\Delta$. In other words, all entries of $\mathbf{u(p})$ have to be equal. This observation falls in line with the well-known \emph{indifference principle}: If $\mathbf{p}$ is a strategy such that when customers play $\mathbf{p}$, an individual customer is indifferent to choosing any action in $\mathcal{A}$, then $\mathbf{p}$ induces equilibrium. The latter argument and its intuitive interpretation can be extended to any strategy $\mathbf{p}=(p_1, \dots, p_k)$ that lies in a facet of $\Delta$, with its corresponding subset of supported actions, $\{a_i \in \mathcal{A} \mid i \text{ \rm{ s.t.} } p_i>0 \}$.
\end{remark}

Clearly, the deterministic algorithm just described is only useful when the utility function $\mathbf{u(p)}$ can be computed.
In many settings, however, the entries of $\mathbf{u(p)}$ may not be given in closed form, and even numerical approximations can be difficult to obtain. Therefore we devise a method of approaching the desired equilibrium using an SA algorithm. 

\section{Simulation and stochastic approximation}\label{sec:SA}

The stochastic approximation procedure is based on the Robbins-Monro algorithm, which aims to mimic the deterministic fixed-point iteration of \eqref{eq:FI_iteration}, replacing the progression direction $\mathbf{u(p)}$ with a noisy estimator based on a realization of  the queueing process. However, estimating $\mathbf{u(p)}$ naively by taking a sample average over a (finitely long) simulation path, often results in a biased estimator. Using such biased estimators in the SA scheme, one risks the algorithm will converge to an undesirable limit. In light of Remark \ref{rem:choice-of-ell}, we thus replace $\mathbf{u(p)}$ in \eqref{eq:FI_iteration} with $\mathbf{g(p)}=\ell(\mathbf{p})\mathbf{u(p)}$, which we can easily estimate, without bias, based on realizations of regeneration cycles.
A typical iteration of the SA algorithm then runs as follows.  At iteration $n$, assume that the strategy is $\mathbf{p}^{(n)}$, and that we are provided with an estimator $\mathbf{G}^{(n)}$ for $\mathbf{g}(\mathbf{p}^{(n)})$. Recall that $\gamma_n$ is the $n$-th step size. The update at the next iteration of the SA scheme is done by
\begin{equation}\label{eq:SA_iteration}
    \mathbf{p}^{(n+1)} = \pi_\Delta \left( \mathbf{p}^{(n)} + \gamma_n \mathbf{G}^{(n)}\right),\ n\geq 1,
\end{equation}
where $\mathbf{p}^{(1)}\in\Delta$ is an arbitrary initial strategy. 
If $\mathbf{G}^{(n)}$ is unbiased for $\mathbf{g}(\mathbf{p}^{(n)})$, then under mild regularity conditions the algorithm will converge to a fixed point of $\mathbf{f}$ (see Theorem \ref{thm:convergence}), i.e., to a Nash equilibrium. The remainder of this section is dedicated to the construction of an unbiased estimator and to presenting the necessary assumptions and steps for our main convergence result.

\subsection{Unbiased estimator}

Assume that the strategy $\mathbf{p}$ is given. To avoid using cumbersome notation we suppress in some places the dependency on $\mathbf{p}$, however all the distributions of the random variables to be defined next should be regarded as functions of $\mathbf{p}$. To construct the estimator, we simulate a single regenerative cycle starting with the arrival of a customer at an empty system; $X_1(\mathbf{p})=0^d$. As before, we denote by $L=\inf\{n\geq 1 \mid X_{n+1}(\mathbf{p})=0\}$ the number of arrivals during that cycle. For each arrival $j=1, \dots, L$, let $X_j$ be the state and $Y_j$ be the corresponding random outcome. For the $j$-th customer, the vector $\mathbf{v}(X_j, Y_j)$ as defined in \eqref{eq:vXY} reflects the potential utility that each action will yield to that customer. We clarify that calculating $\mathbf{v}(X_j, Y_j)$ for each customer $j$ requires to assess their realized value for \emph{every} action in $\mathcal{A}$ \emph{as if} they chose it, yet in the simulation process, the action taken by the $j$-th customer is fully governed by $\mathbf{p}$.

Denote the expectation of the utility conditional on observing state $X$ by
\begin{equation}\label{eq:Vx}
\mathbf{\overline{v}}(X)= \E_Y \left[\mathbf{v}\big(X, Y\big) \: \big| \: X \right].
\end{equation}
Throughout the paper we assume, for all $\mathbf{p}\in\Delta$, that $\mathbf{\overline{v}}$ is integrable with respect to the distribution of $X(\mathbf{p})$. Then the vector of expected stationary utilities, $\mathbf{u}(\mathbf{p})$, can be written as
$\mathbf{u}(\mathbf{p})=\E_\mathbf{p} \big[ \mathbf{\overline{v}}(X(\mathbf{p})) \big]$.
Define 
\begin{equation}\label{eq:G}
\mathbf{G} =  \sum_{j=1}^{L} \mathbf{\overline{v}}(X_j).
\end{equation}
The following lemma, which is a direct result of a classical formula of regenerative processes (see, for example, \cite[Prop.~A.3]{CI1975}), states that $\mathbf{G}$ is an unbiased estimator for $\mathbf{g(p)}$.
\begin{lemma}\label{lemma:unbiased_G}
Suppose that $\ell^2(\mathbf{p})<\infty$ for all $\mathbf{p}\in\Delta$, then
\begin{equation}\label{eq:EG}
\E_{\mathbf{p}}\mathbf{G}=\ell(\mathbf{p})\mathbf{u(p)}=\mathbf{g(p)}.
\end{equation}
\end{lemma}

Lemma \ref{lemma:unbiased_G} explains the rational behind the choice of an estimate for $\mathbf{g(p)}$ over an estimate for $\mathbf{u(p)}$. As discussed, one may attempt to estimate $\mathbf{u(p)}$ directly, for example by considering the naive estimator 
\[\mathbf{Z} = \frac{1}{L}\sum_{j=1}^L\mathbf{\overline{v}}(X_j),\] 
however, such an estimator is generally biased. To see why, we provide an example.

\setcounter{example}{0}
\begin{example}\textbf{{\rm (Continued)}}
Recall the unobservable M/G/1 model with service rate $\mu$, reward $R$, and delay cost $C$, where we use $\mathbf{p}=(p,1-p)$ to denote a strategy and $X_j$ to describe the virtual workload at the $j$-th arrival, so that $\E_\mathbf{p} [ X(\mathbf{p}) ]=w(p\lambda)$ is the mean virtual workload when the arrival rate is $p\lambda$. In particular, the estimator $\mathbf{Z}$ when constructed from a single regeneration cycle (i.e., one busy period) is a 2-dimensional vector, $\mathbf{Z} = (R-C\cdot (\bar{X}+1/\mu),0 )$, where $\bar{X}=(1/L)\sum_{j=1}^L X_j$ is the average workload at arrival instants over all customers of that cycle. It is known that $\bar{X}$ is a biased estimator of $w(p\lambda)$, i.e., that $\E_\mathbf{p} [ \bar{X}]\neq w(p\lambda)$ (see, for example, \cite{CI1975}), and so, $\mathbf{Z}$ is biased for $\mathbf{u(p)}$. The fact that $\bar{X}$ is biased can be explained as a ramification of the well-known \emph{length bias}: The average virtual workload over long busy periods will naturally tend to be larger than that over short busy periods. Yet, during a long busy period more customers will typically arrive at the system than in a short one. By averaging the sampled workloads one erroneously normalizes the length of the  busy period, and as a result,  gives disproportional importance to short busy periods in the estimation, resulting in an underestimate of the mean workload. 
By contrast, the estimator $\mathbf{G}$ presented above, for the unobservable M/G/1 model takes the form 
\[ \mathbf{G}=\sum_{j=1}^L \mathbf{\overline{v}}(X_j)= \sum_{j=1}^L \begin{pmatrix} R-C\cdot(X_j+1/\mu) \\ 0 \end{pmatrix}.\]
This estimator is indeed unbiased for $\mathbf{g(p)}$, and in particular, $\E_{\mathbf{p}}\left( \sum_{j=1}^L X_j \right) = \ell(\mathbf{p})w(\lambda p)$.
\hfill $\diamond$
\end{example}

It is important to point out the difference between using the realized value vector for customer $j$, $\mathbf{v}(X_j, Y_j)$, and using the conditional expectation $\mathbf{\overline{v}}(X_j)$, as prescribed in (\ref{eq:G}). The latter has two main advantages. Firstly, using the expected value instead of the realization reduces the variance of the estimator. Secondly, $\mathbf{\overline{v}}$ is merely a function of the state, and hence is more convenient to work with. 
Yet, in some settings it may not be trivial to express this function explicitly. 

In principle, one can directly use $\mathbf{v}(X_j, Y_j)$ in Equation (\ref{eq:G}) instead of $\mathbf{\overline{v}}(X_j)$. However, since $\mathbf{v}(X_j, Y_j)$ is a function of both the observed state $X_j$ and the random outcome $Y_j$, in order to apply the result of Lemma \ref{lemma:unbiased_G}, one must make assumptions about the dependence between $X_j$ and $Y_j$. For example, the following assumption is sufficient: At the $j$-th arrival with the observed state $x_j$, the random outcome conditioned on the state observed, $Y_j\mid X_j =x_j$, should be independent of $X_{j'}$ and $Y_{j'}$ for all $j'\neq j$. This assumption may not hold in systems with elaborate service policies such as priority queues or processor sharing: For instance, in processor sharing, the waiting of the $j$-th arriving customer indeed depends on the service demand (as well as the action) of later-arriving customers. Furthermore, the value $\mathbf{v}(X_j, Y_j)$ is not a byproduct of a single busy period simulation because it must be computed for every possible action of a customer, including those not taken in the realization of the simulation. 
A possible solution is, for each $j=1\ldots,L$, to obtain an independent sample of $\mathbf{v}(x_j, Y_j)$ using subordinate simulations, independently of the main simulation path.  
To generate a sample of any component $v_i(x_j, Y_j)$ of the vector $\mathbf{v}(x_j, Y_j)$, one can run an independent subordinate simulation that starts with the arrival of a customer at state $x_j$ who takes action $a_i$, and terminates when the utility of that customer is realized. When possible, the same subordinate simulation can be used to evaluate as many components of the vector $\mathbf{v}(x_j, Y_j)$ as needed, as long as this is done independently of the main simulation path.

\subsection{Convergence and main result}\label{sec:converg}

In the approximation process, given a strategy $\mathbf{p}$, we make use of the estimator $\mathbf{G}$ at each iteration of the algorithm as a ``proxy'' for the true value $\mathbf{g(p)}$. Following the framework of \cite[Ch.~5]{KY2003} we present conditions for this algorithm to converge almost surely to a set of equilibrium points. In particular, we assume that the variance of the unbiased estimator $\mathbf{G}$ is finite for any strategy $\mathbf{p}$. This is combined together with the assumptions required for existence of an equilibrium, as discussed in Section \ref{sec:model}, to obtain the convergence result described in Theorem \ref{thm:convergence}. The proof of that theorem appears in Appendix \ref{sec:app_convergence-proof}.

To state our assumptions, we
first introduce the hyperplane $\mathcal{H}=\{\mathbf{x}\in \mathbb{R}^k \mid \mathbf{e}'\mathbf{x}=0\}$ which is a $(k-1)$-dimensional subspace of $\mathbb{R}^k$, parallel to $\Delta$. We further note that the orthogonal projection onto $\mathcal{H}$, $\pi_{\mathcal{H}}:\mathbb{R}^k\to \mathcal{H}$, is a linear transformation, and can be represented by a symmetric matrix $\mathbf{H}\in\mathbb{R}^{k \times k}$, namely, $\pi_{\mathcal{H}}(\mathbf{x})=\mathbf{Hx}$. 

\begin{assumption}\label{assum:finite_El2}
The second moment of the cycle length $L(\mathbf{p})$ is bounded for every strategy;  $\ell^2(\mathbf{p})<\infty$ for all $\mathbf{p}\in \Delta$.
\end{assumption}

\begin{assumption}\label{assum:finite_EG2}
The second moment of the estimator $\mathbf{G}$ is uniformly bounded on $\Delta$;\\ $\sup_{\mathbf{p}\in\Delta}\E_{\mathbf{p}} \Vert \mathbf{G} \Vert^2<\infty$. 
\end{assumption}

\begin{assumption}\label{assum:finite_u} The projected utility function $\mathbf{Hu}(\cdot)$ is continuous, and there exists a differentiable function $u^*:\Delta\to\mathbb{R}$ such that $-\mathbf{Hu}(\mathbf{p})$ is the gradient of $u^*$ at $\mathbf{p}$, i.e., $-\mathbf{Hu}(\mathbf{p})=\nabla u^*(\mathbf{p})$.
\end{assumption}

\begin{assumption}\label{assum:gamma_n}
The step-size sequence $\{\gamma_n\}_{n\geq 1}$ satisfies 
\begin{equation}\label{eq:gamma_condition}
\sum_{n=1}^\infty \gamma_n=\infty, \quad \sum_{n=1}^\infty \gamma_n^2<\infty ,
\end{equation}
\end{assumption}

\begin{theorem}\label{thm:convergence}
Suppose Assumptions \ref{assum:finite_El2}--\ref{assum:gamma_n} are satisfied. As $n\to\infty$, $\mathbf{p}^{(n)}\asarrow \mathcal{S}^e$, where $\mathcal{S}^e\subseteq\Delta$ is a set of equilibrium strategies, namely, $\mathbf{f}(\mathbf{p})=\mathbf{p}$ for every $\mathbf{p}\in \mathcal{S}^e$.
\end{theorem}

\begin{corollary}\label{eq:corrol_conv}
Under the conditions of Theorem \ref{thm:convergence}, if the set of equilibrium strategies is finite, then as $n\to\infty$, $\mathbf{p}^{(n)}\asarrow \mathbf{p}^e$, where  $\mathbf{p}^e$ is an equilibrium.
\end{corollary}


\subsection{Discussion of assumptions}\label{sec:disc}

 Assumption~\ref{assum:finite_El2} is required for positive recurrence of the underlying regenerative process, which in turn ensures the existence of a stationary distribution for any strategy.
The step-size condition of Assumption~\ref{assum:gamma_n} is standard and enables the application of the law of large numbers for martingale differences, which underlies the convergence result in Theorem \ref{thm:convergence}. Assumptions~\ref{assum:finite_EG2}--\ref{assum:finite_u} impose regularity conditions on the utility function which are specific to the framework of queueing games, and therefore deserve more attention, as discussed below.

A straightforward generalization of Lemma \ref{lem:existence-of-equilibiurm} shows that the continuity of $\mathbf{Hu}(\cdot)$ implies that an equilibrium exits, in accordance with the first part of Assumption~\ref{assum:finite_u}. The second part of the assumption, which relates to the existence of some potential function $u^*$,  states, in different words, that $\mathbf{Hu}(\cdot)$ should be a conservative vector field on $\Delta$. Still, the function $u^*$ need not be specified. Its existence prevents the iterate from circulating a perpetual loop of off-equilibrium strategies. Importantly, without this assumption there is no guarantee  that the deterministic algorithm converges to an equilibrium, even when $\mathbf{u(p)}$ is explicitly given. For a continuous utility $\mathbf{u}(\cdot)$, Assumption \ref{assum:finite_u} holds, for example, when $k=2$, or alternatively when $k>2$ and $\mathbf{u}$ itself is a conservative vector field on $\Delta$ (see Remark \ref{rem:conservative-u} in Appendix \ref{sec:app_convergence-proof}). Further elaboration on Assumption \ref{assum:finite_u} is provided in Appendix \ref{sec:app_convergence-proof}.

Assumption  \ref{assum:finite_EG2} imposes that the variance of the estimator $\mathbf{G}$ is finite. Finite variance of the increments is quite a common assumption in the framework of stochastic approximation. However, verifying it  for specific queueing games with only limited knowledge about the stationary distribution can be challenging. In this paper we provide specific examples in which we rely on  queueing-analytic tools to verify this assumption. Example \ref{example:MG1} below  provides an analogous condition to Assumption \ref{assum:finite_EG2} for the unobservable M/G/1 queue, which is formulated in terms of moments of the service-time distribution.

\setcounter{example}{0}
\begin{example}\textbf{{\rm (Continued)}}
Consider once more the unobservable M/G/1 model described in Sections \ref{sec:model} and \ref{sec:SA}: Denote the generic service time random variable by $Y$ with $\E Y=1/\mu$. As before, we assume that $\lambda<\mu$. 
We further assume $\E Y^4<\infty$, which clearly implies $\E Y^2<\infty$. In this case Assumptions \ref{assum:finite_El2} and \ref{assum:finite_u}
can be verified by applying known properties, which is done in Section \ref{sec:parallel_GG1}. In the particular unobservable M/G/1 model, $\E Y^4<\infty$ is also a necessary condition for Assumption \ref{assum:finite_EG2};
Recall that $G_1=\sum_{i=1}^L(R-C(X_i+1/\mu))$, where $X_i$ is the virtual workload observed by customer $i$. Hence, $G_1$ is an affine transformation of the sum of waiting times during a busy period. An explicit formula of the variance of the cumulative workload during a busy period was given in \cite{DJ1969} as a function of $\lambda$ and $\E Y^r$ for $r=1,2,3,4$. Therefore, for the strategy $\mathbf{p}=\mathbf{e}_1$ (all customers join), if $\lambda<\mu$ and $\E Y^4<\infty$, then
\[
\E_{\mathbf{e}_1} \Vert \mathbf{G} \Vert^2=\E_{\mathbf{e}_1} \Bigg\Vert \sum_{j=1}^{L}  \mathbf{\overline{v}}(X_j(\mathbf{e}_1))\Bigg\Vert^2<\infty .
\] 
Furthermore, by standard coupling arguments it can be verified that for every strategy $\mathbf{p}\in\Delta$,
\[
\E_{\mathbf{p}} \Vert \mathbf{G} \Vert^2\leq \E_{\mathbf{e}_1} \Vert \mathbf{G} \Vert^2.
\]
We conclude that $\lambda<\mu$ together with $\E Y^4<\infty$ form a necessary and sufficient condition for Assumption \ref{assum:finite_EG2}. Theorem~\ref{thm:convergence} then implies that under these conditions the SA algorithm converges almost surely to the unique equilibrium. \hfill$\diamond$ 
\end{example}

The conditions in Assumptions~\ref{assum:finite_El2}--\ref{assum:finite_u} are imposed on all strategies in the simplex. However, any strategy in $\Delta$ for which the system is not stable, in general, refutes the assumptions. This issue can be alleviated in some cases using a small modification of the algorithm. If there exists a known closed convex set of strategies $\mathcal{S}\subseteq\Delta$ that contains an equilibrium, and conditions \ref{assum:finite_El2}--\ref{assum:finite_EG2} are satisfied for all $\mathbf{p}\in\mathcal{S}$, then the projection $\pi_\mathcal{S}$ can be used in \eqref{eq:SA_iteration} and all our results carry over. In practice, such a set can often be found because stability conditions are known for many models even without explicit derivation of the stationary distribution. In absence of a-priori knowledge of the stability region, the algorithm has to be modified to incorporate `stability-enforcing' dynamics. Appendix~\ref{sec:app_stability} presents a heuristic modification of the algorithm that terminates a busy cycle if its length exceeds a threshold. Of course, this introduces a bias in the estimation step. To eliminate this bias, the termination threshold is increased gradually, with the goal of making the bias asymptotically negligible once the algorithm is absorbed in the stability region.

That said, we acknowledge that the task of detecting stability regions using simulation is challenging. In fact, the problem of determining whether a queueing system is stable under a general schedueling policy was shown in \cite{G2002} to be undecidable even in a simple system configuration. A  simulated-annealing method to identify stability regions in some practical setups is suggested in \cite{MPW2017}. Some ad-hoc solutions have been discussed in recent papers that involve learning and control algorithms for queueing systems (see \cite{DG2021} and \cite{CLH2020}).

\subsection{Rate of convergence and approximate equilibrium}\label{sec:conv}
An important question from a practical standpoint is how many iterations are needed to obtain a `good' approximation. 
This question has been the subject of interest for a large body of works in the literature of stochastic gradient-descent algorithms, where it is often implied by appropriate conditions on the primitives that the iterate $\mathbf{p}^{(n)}$ converges (almost surely) to a unique point $\mathbf{p}^e$. Then, the above question is addressed via studying different notions of convergence of the scaled approximation error, $(\mathbf{p}^{(n)}-\mathbf{p}^e)/\sqrt{\gamma_n}$.
For a specific choice of the step-size sequence $\{\gamma_n\}_{n\geq 1}$,  \cite{NJLS2009} established $L^2$-convergence of the scaled error, by imposing strong convexity and smoothness on the potential function ${u}^*$. Under these assumptions, the equilibrium exists uniquely (as it is the minimum point of $u^*$), and the vector $-\mathbf{Hu}$ is the gradient for which we have noisy observations. In what follows we leverage this framework to bound our algorithm's equilibrium-approximation error with high probability.

Rather than studying the scaled error per se, we find it more meaningful in our setup to understand, given the iteration index $n$, how ``close" the strategy $\mathbf{p}^{(n)}$ is to meeting the equilibrium condition. We therefore adopt the concept of \emph{$\epsilon$-approximate Nash equilibrium} (see \cite{DMP2009}). Recall that a strategy $\mathbf{p}$ is a (symmetric) $\epsilon$-approximate Nash equilibrium, or an $\epsilon$-equilibrium for short, if no deviation from this strategy can increase a customer's expected utility by more than $\epsilon$:
\begin{equation}\label{eq:epsilon_Nash}
\mathbf{u}(\mathbf{p})\apost\mathbf{p}\geq  \max_{\mathbf{q}\in\Delta}\mathbf{u}(\mathbf{p})\apost\mathbf{q}-\epsilon.
\end{equation}
Proposition \ref{prop:eps_NE} below relates the convergence of the scaled errors to $\epsilon$-equilibria. The proof is provided in Appendix~\ref{sec:app_norm}. 

\begin{proposition}\label{prop:eps_NE}
Suppose that Assumptions \ref{assum:finite_El2}--\ref{assum:finite_u} hold. Suppose also that $\mathbf{u}$ is (locally) Lipschitz on $\Delta$, and that $u^*$ in \ref{assum:finite_u} is strongly convex on $\Delta$, i.e., there exists a constant $C>0$ such that
\begin{equation}\label{eq:convex_potential}
u^*(\mathbf{p})\geq u^*(\mathbf{q})+\mathbf{Hu}(\mathbf{q})\apost(\mathbf{p}-\mathbf{q}) + \frac{C}{2}\Vert \mathbf{p}-\mathbf{q} \Vert^2, \ \forall \mathbf{p,q}\in\Delta.
\end{equation}
Let $\gamma_n=\eta/n$, where $\eta> 1/(2C)$. Then there exists a constant $M$, such that for all $n\geq 1$ and $\delta\in(0,1)$,
\begin{equation}\label{eq:epsilon_nash_P}
        \P\left(\mathbf{u}(\mathbf{p}^{(n)})\apost\mathbf{p}^{(n)} \geq         \max_{\mathbf{q}\in\Delta}\mathbf{u}(\mathbf{p}^{(n)})\apost \mathbf{q} -\sqrt{n^{-\delta}}\right)
        \geq 1-M \cdot  n^{\delta-1};
\end{equation}
that is, for all $n$ and  $\delta\in(0,1)$, the $n$-th iterate $\mathbf{p}^{(n)}$ is a $\sqrt{n^{-\delta}}$-equilibrium with probability (at least) $1-M \cdot  n^{\delta-1}$.
\end{proposition}

Notice that the choice of the step size $\gamma_n=\eta/n$ in Proposition \ref{prop:eps_NE} satisfies \ref{assum:gamma_n}. Following \cite{NJLS2009}, under the conditions of Proposition~\ref{prop:eps_NE}, this choice of $\gamma_n$ implies that $\E\Vert \mathbf{p}^{(n)} - \mathbf{p}^e\Vert^2=\mathrm{o}(1/n)$. 
Proposition~\ref{prop:eps_NE} shows this also yields an $\epsilon_n$-equilibrium with high probability ($1-\mathrm{o}(n^{\delta-1})$), where $\epsilon_n=\sqrt{n^{-\delta}}$. For example, with $\delta=1/2$,  the result ensures that an $\epsilon$-equilibrium is reached, with $1-\mathrm{o}(\epsilon^{2})$ certainty, after $\epsilon^{-4}$ iterations.

There are various other conditions in the literature for convergence of the scaled error, including almost-sure convergence rates (e.g., \cite{SGD2021} for unconstrained stochastic optimization) and weak convergence to the normal distribution (e.g., \cite[Ch.~8.4]{B2008}). All of these results require assumptions on the utility vector $\mathbf{u}$ and the penitential function $u^*$ that ensure the algorithm is ``well behaved'', so to speak, in the neighborhood of equilibrium points. In Proposition~\ref{prop:eps_NE}, these assumptions take the form of \eqref{eq:convex_potential}  and the corresponding step-size constant $\eta>1/(2C)$. Though in applications these conditions can be hard to verify, our result sheds light on the obtainable rates when some limited structure of the stationary expected utility is available. For a detailed discussion of the convergence rate for constrained and unconstrained stochastic-approximation algorithms, see \cite{BK2002} and \cite{KY2003}. 

\subsection{Accounting for observable information in customer's decisions}\label{sec:comment-obs}

In the context of queueing games, it is natural to think of an `action' as representing a possible terminal outcome of the customer's decision, e.g., `join' or `balk', as described in Example \ref{example:MG1}. However, in our framework, the term `action' is in fact synonymous with `pure strategy' (this interpretation slightly deviates from the terminology used in extensive-form games, see \cite[Ch.~3]{MSZ2013}). \textcolor{black}{Thus, an element in $\mathcal{A}$ can potentially represent a comprehensive decision rule, prescribing what to do at each possible event, with respect to the information available to the customer during their sojourn in the system. For instance, assuming customers observe the queue length, a pure strategy can prescribe joining at some queue-lengths while balking at others. It can further capture the evolution of the system state after the arrival, e.g., for $x\in\{1,2,3\}$, specify whether to abandon the queue after $x$ minutes if the queue length did not decrease. }

\textcolor{black}{ One way to account for observable information in the customer decision-making process is to extend the action space, i.e., the set of possible pure strategies, to encompass all possible decision rules.
Note that the information can also pertain to individual customer features, and not necessarily the system state. In Appendix~\ref{sec:heterogeneous-customers} we implement this approach in a game with two customer types.
Below we explain through a brief example how to amend the action set when state information is available upon arrival. }

Consider an observable GI/G/1/$N$ queue: Each customer observes the number $n$ of customers in the queue, and assuming it is not full ($0\leq n\leq N-1$), they choose whether to join or balk. Then a pure strategy is an element in $\{0,1\}^N$ determining which queue lengths to join, and we associate it with a single action in the (finite) set $\mathcal{A}$. 

Observable queueing games with finitely many pure strategies are covered by the formulation in Section~\ref{sec:model} and hence our results are applicable to such games. However it can be noticed, particularly in the example above, that the dimension of the strategy space grows exponentially with the number of signals potentially observed. When the number of signals is large, this approach may not be practical. In Section \ref{sec:observable} we discuss an extension of the algorithm that traverses a strategy space whose dimension grows proportionally (as opposed to exponentially) in the number of signals. 

\section{Applications,  implementation and refinements}\label{sec:applications}

This section presents two applications of the SA algorithm. The first, discussed in Section \ref{sec:parallel_GG1}, is a general game of choosing a queue from a pool of parallel single-server queues, or balking from the system altogether. This model, inspired by \cite{BS1983}, generalizes the canonical M/G/1 model of Example~\ref{example:MG1} and also the motivating example of two parallel queues in Section~\ref{sec:motivation}. The second example, in Section \ref{sec:CR} is a generalization of the model introduced in \cite{HS2017}, considering a system comprised of two parallel servers; one with an infinite buffer queue and one with no queueing buffer, and customers have the option to probe (sense) the no-buffer server at a cost before joining the infinite buffer queue.
For both applications the sufficient convergence conditions of Theorem~\ref{thm:convergence} are verified in detail, and some numerical results are reported. The related proofs appear in Appendix \ref{sec:appA}. 
Based on the second example we introduce and illustrate two practical refinement techniques of the algorithm, one that reduces the variance of the estimator and one that dynamically chooses the step size. 

\subsection{Unobservable GI/G/1 queues in parallel}\label{sec:parallel_GG1}
The main purpose of the section is to rigorously verify Assumptions \ref{assum:finite_El2}--\ref{assum:finite_u}, which imply the convergence of the SA algorithm, in the queueing game described below. The proof of Proposition \ref{prop:GGn} is provided in Appendix~\ref{prop:GGn-proof} and relies on a coupling construction, which we believe can be found useful in verifying the assumptions for further generalizations and variations of the model here.

For some integer $k>1$, consider a system of $k-1$ parallel GI/G/1 queues where customers can choose which queue to join, with the possibility of balking, i.e., not joining any queue at all. Arrivals to the system are generated according to a renewal process with a continuous  inter-arrival distribution $H$ and mean $1/\lambda$. Each arriving customer chooses, without observing the system state, whether at all to join a queue, and if so, which of the $k-1$ queues to join.  Denote the generic service time variable in station $m\in\{1,\ldots,k-1\}$ by $Y_m\sim F_m$, and its mean by $1/\mu_m$. A strategy $\mathbf{p}=(p_1, \dots, p_k)\in\Delta$, for each $m\in\{1,\ldots,k-1\}$ expresses the probability $p_m$ of joining queue $m$, with $p_k$ being the probability of balking.
Arriving at a realization of the state $x = (x^{[1]}, \dots, x^{[k-1]})$, the net value for a customer joining queue $m=1, \dots, k-1$, is given by ${v}_m(x, y_m)=\nu_m (x^{[m]}+y_m)$, where $y_m$ is a realization of $Y_m$. The value from balking is normalized to zero; ${v}_{k}(x)=0$.

Clearly, this is a generalization of the motivating example presented in Section~\ref{sec:motivation} and the unobservable M/G/1 model discussed in Example \ref{example:MG1}. The next proposition presents sufficient conditions for convergence of the SA algorithm to a Nash equilibrium.
\begin{proposition}
\label{prop:GGn}
Assume that for every $m\in\{1,,\ldots,k-1\}$: (i) $\lambda < \mu_m$; (ii) $\E [Y_m^4]<\infty$; and (iii) the function $\nu_m:\mathbb{R}_+\to\mathbb{R}$ is Lipschitz continuous. Then for any initial strategy $\mathbf{p}^{(1)}\in\Delta$ and step-size sequence $\{\gamma_n\}_{n\geq 1}$ satisfying Assumption  \ref{assum:gamma_n}, the SA algorithm defined in \eqref{eq:SA_iteration} converges almost surely to a set of equilibrium strategies, in the sense of Theorem \ref{thm:convergence}.
\end{proposition}


\subsection{Selective routing in queues with different buffer capacities}\label{sec:CR}
Consider a queueing network comprised of two servers, Server 1 that has no waiting room, and Server 2 that offers unlimited queuing capacity. Customers arrive at the system following a Poisson process with rate $\lambda$, with iid  service-length requirements distributed according to $F$. Let $Y\sim F$ represent some arbitrary service duration, and let $\E[Y]=1/\mu$. 

Upon arrival a customer cannot observe the system's state, and is given the option to \emph{sense}, that is, to make a costly attempt to obtain service by Server 1, at a fixed cost $c_s$. A customer that chooses to sense and finds Server 1 idle, immediately commences service, otherwise they are instantaneously sent to the end of the infinite-capacity queue and waits to be served by Server 2. The sensing option is offered only upon the customer's arrival, and the cost is incurred regardless of whether Server 1 is found idling or not. Customers who do not choose the sensing option do not incur a sensing cost and are sent directly to the end of the queue served by Server 2. All customers queueing for Server 2 incur a linear waiting cost at rate $c_w$ per unit time.

Let $\mathcal{A}=\{a_1, a_2\}$, where $a_1$ and $a_2$ represent the actions ``sense'' and ``not sense'' respectively. A strategy in this game is given by $\mathbf{p}=(p,1-p)$, where $p>0$ is the probability of choosing $a_1$. It is shown in \cite{HS2017} that when customers' service demand is exponentially distributed, the game admits a unique equilibrium strategy which can be efficiently approximated. Yet even for the special case of exponential services, a simple expression for the equilibrium strategy is not available, because the solution in \cite{HS2017} for the stationary distribution given a strategy $\mathbf{p}$ makes use of Cardano's formula. It is worth mentioning in this context that in general, the queue served by Server 2 cannot be treated as a standard M/G/1, because the effective arrival process to Server 2 is not renewal, as it depends on the state of Server 1. 

To approach the equilibrium using simulation we define the state as a pair $X=(X^{[1]},X^{[2]})$ where each component $X^{[m]}$ stands for the workload at Server $m\in\{1,2\}$. Given a realization $x=(x^{[1]},x^{[2]})$ of $X$, the components of the vector function $\overline{\mathbf{v}}(x)$ take the form
\[ \begin{split}
\overline{v}_1 (x) &= -c_s -c_w\cdot \mathit{1}(x^{[1]} > 0) \cdot x^{[2]} , \\
\overline{v}_2 (x) &= -c_w\cdot  x^{[2]}.
\end{split}\]

\begin{proposition}\label{prop:CR}
Assume the service-time distribution $F$ is chosen such that (i) $\lambda < \mu$; (ii) $\E [Y^4]<\infty$; and (iii) $\mathbf{u(p)}$ is continuous in $\mathbf{p}\in\Delta$. Then for any initial strategy $\mathbf{p}^{(1)}\in\Delta$ and step-size sequence $\{\gamma_n\}_{n\geq 1}$ satisfying Assumption  \ref{assum:gamma_n}, the SA algorithm defined in \eqref{eq:SA_iteration} converges almost surely to a set of equilibrium strategies, in the sense of Theorem \ref{thm:convergence}.
\end{proposition}

The simplicity of the model here sets up a convenient platform to demonstrate some practical heuristics that can enhance the algorithm's performance, which is what we focus on next.

\subsection{Variance reduction using control variates}

The variance reduction technique discussed here is a fundamental control-variate method based on \cite{RR1985}, that exploits our knowledge of the strategy $\mathbf{p}$ at the beginning of each iteration, as well as some model parameters used to construct the simulation. Adapted to our framework, the goal is to modify the estimator $\mathbf{G}$ in a way that does not introduce bias, but reduces its generalized variance. The underlying principle is the following: Suppose one can characterize an $l$-dimensional random vector $\mathbf{C}$, referred to as the control vector, with zero mean (entrywise), such that it is also strongly correlated with $\mathbf{G}$. Then an appropriate matrix of control coefficients $\mathbf{\Psi}\in \mathbb{R}^{l\times k}$ can be found, for which the new estimator $\mathbf{G}-\mathbf{\Psi}'\mathbf{C}$ is unbiased for the original unknown quantity, but has smaller generalized variance than that of $\mathbf{G}$. Mathematically, this means that we would like to choose $\mathbf{\Psi}$ such that $\mathbf{G}-\mathbf{\Psi}'\mathbf{C}$ satisfies
\[ 
\Ep[\mathbf{G}-\mathbf{\Psi}'\mathbf{C}] = \Ep[\mathbf{G}] = \mathbf{g(p)}, \quad\text{and}\quad \det(\Varp[\mathbf{G}-\mathbf{\Psi}'\mathbf{C}]) < \det(\Varp[\mathbf{G}]).
\]
Observe that unbiasedness is obtained irrespective of the choice of $\mathbf{\Psi}$. 
It is shown in \cite{RR1985} that the optimal choice for $\mathbf{\Psi}$ (in terms of minimal variance) is given by $\mathbf{\Psi}^* = \mathbf{\Sigma}_\mathbf{GC} \mathbf{\Sigma}^{-1}_\mathbf{CC}$, where $\mathbf{\Sigma}_\mathbf{CC} = \Varp[\mathbf{C}]$ is the variance-covariance matrix of the control vector $\mathbf{C}$, and $\mathbf{\Sigma}_\mathbf{GC} = \Covp(\mathbf{G,C})$ is the cross-covariance matrix of $\mathbf{G}$ with $\mathbf{C}$. In most practical applications, both $\mathbf{\Sigma}_\mathbf{CC}$ and $\mathbf{\Sigma}_\mathbf{GC}$ are unknown, and so $\mathbf{\Psi}^*$ has to be estimated from the data. Moreover, in the setting here the optimal $\mathbf{\Psi}^*$ varies with the strategy $\mathbf{p}$. Luckily, our numerical results show that significant variance reduction is obtained even when using a rough approximation of $\mathbf{\Psi}^*$ which is based on naive sampling. Our guiding intuition is that as the iteration index $n$ increases, the change in the strategy $\mathbf{p}^{(n)}$ becomes smaller. Given a constant strategy, it is known that the estimator for $\mathbf{\Psi}^*$ constructed from the sample covariance and cross-covariance matrices is likelihood maximizing (see \cite{A1962}, Section 4.3.1). 
 
At iteration $n>1$, let $\mathbf{G}^{(n)}$ be the estimator defined in (\ref{eq:G}), and $\mathbf{C}^{(n)}$ be the control vector. We use a naive choice for the matrix of control coefficients $\hat{\mathbf{\Psi}}^{(n)}$ by defining (note that $\mathbf{C}^{(1)}, \dots \mathbf{C}^{(n)}$ have mean zero):
\[ \hat{\mathbf{\Sigma}}^{(n)}_\mathbf{CC} = \frac{1}{n-1}\sum_{m=1}^n \mathbf{C}^{(m)} \left.\mathbf{C}^{(m)}\right.^{\prime}, \quad \hat{\mathbf{\Sigma}}^{(n)}_\mathbf{GC} = \frac{1}{n-1}\sum_{m=1}^n \mathbf{G}^{(m)} \left.\mathbf{C}^{(m)}\right.^{\prime}, \text{ and} \quad \hat{\mathbf{\Psi}}^{(n)} = \left.{\hat{\mathbf{\Sigma}}^{(n)}_\mathbf{GC}}\right.^{-1} \hat{\mathbf{\Sigma}}^{(n)}_\mathbf{CC}. 
\]
Thus, at every iteration $n>1$ we replace $\mathbf{G}^{(n)}$ in Equation (\ref{eq:SA_iteration}) by $\mathbf{G}^{(n)}-\left.\hat{\mathbf{\Psi}}^{(n)}\right.^{\prime} \mathbf{C}^{(n)}$. 

\setcounter{example}{1}\begin{example}
Consider the model presented in Section \ref{sec:CR}. We choose a 3-dimensional control vector $\mathbf{C}=(C_1, C_2, C_3)$. As prescribed in Section \ref{sec:converg}, given a strategy $\mathbf{p}=(p, 1-p)$, we generate $L$ observations $X_1, \dots, X_L$ based on a single regeneration cycle, and construct the original estimator $\mathbf{G} = \sum_{j=1}^{L}  \mathbf{\overline{v}}(X_j)$ introduced in (\ref{eq:G}). Let $D_1, \dots, D_L$, be a sequence of \emph{sensing indicators}, namely for $j=1, \dots, L$, $D_j\sim \text{Bernoulli}(p)$ is the indicator of the event ``the $j$-th customer chooses action $a_1$ (sense)''. Let $Y_j$ be the service demand of that $j$-th customer. Our control vector of choice, $\mathbf{C}$, is given by:
\[ 
\mathbf{C} = \begin{pmatrix} C_1 \\ C_2 \\ C_3 \end{pmatrix} = \sum_{j=1}^L\begin{pmatrix}  D_j - p \\ Y_j- 1/\mu \\ \mathit{1}(X_j^{[1]} > 0) - \frac{ \lambda p}{\mu + \lambda p} \end{pmatrix}.
\]
Because the underlying process is regenerative, we have for every $\mathbf{p}$,
\[ 
\Ep[\mathbf{C}] = \ell(\mathbf{p}) \cdot \begin{pmatrix} \Ep[D] - p \\ \Ep[Y]- 1/\mu \\ \Ep[\mathit{1}(X^{[1]} > 0)] - \frac{ \lambda p}{\mu + \lambda p} \end{pmatrix} = \mathbf{0},
\]
where $D\sim \text{Bernoulli}(p)$ is the sensing indicator of an arbitrary customer, $Y\sim F$ is an arbitrary service demand, and $X=(X^{[1]}, X^{[2]})$ is the stationary system state, given strategy $\mathbf{p}=(p, 1-p)$. It is therefore immediate that $\Ep[D] = p$ and  $\Ep[Y]= 1/\mu$, and because Server 1 can essentially be modeled as an M/G/1/1 loss-system with arrival rate $\lambda p$ and service rate $\mu$, we also have in stationarity that $\Ep[\mathit{1}(X^{[1]} > 0)] = \frac{ \lambda p}{\mu + \lambda p}$. This specific choice of the control vector $\mathbf{C}$ is led by the understanding that customers' actions, their realized service demand, as well as the state of Server 1 observed by the different customers during the simulated cycle, are all affecting the utility sum $\mathbf{G}$ in that cycle.
 \hfill$\diamond$\end{example}

Figure \ref{fig:cr-game} depicts the performance of the algorithm incorporating control variates and dynamic step sizes (described below) for the model described in Section \ref{sec:CR}. Further information about the efficiency and selection of control coefficients for the method of control variates in vectored-output simulations can be found in \cite{RR1985} and \cite{YN1992}.

\subsection{Dynamic step size selection}
Next we consider the selection of the step-size sequence, $\{ \gamma_n \}_{n\geq 1}$. Our main idea underlying the SA algorithm is that instead of estimating $\mathbf{u(p)}$, to which unbiased estimators in general are not known, one can estimate  $\ell(\mathbf{p})\mathbf{u(p)}$ with no bias, based on a single regeneration cycle, and use this estimator to determine the progression of the algorithm. Yet, as noted in Remark \ref{rem:choice-of-ell}, the function $\ell(\mathbf{p})$ chosen is not exclusive; If $h:\Delta\to\mathbb{R}_+$ is a real continuous function such that $h(\mathbf{p})\geq a$ for some constant $a>0$, and in addition if one is endowed with a finite-variance unbiased estimator for  $h(\mathbf{p})\mathbf{u(p)}$, then this estimator can be used instead of the one suggested in Section \ref{sec:SA}, and the convergence result in Theorem \ref{thm:convergence} carries over. The exact form of $h(\mathbf{p})$ need not be expressed or specified, neither should it be introduced as input to the algorithm. 

In theory, if $\ell(\mathbf{p})$ can be computed, then the estimator $\mathbf{G}/\ell(\mathbf{p})$ can be taken to directly estimate $\mathbf{u(p)}$ without bias. This is essentially equivalent to dynamically choosing the step size at iteration $n$ to be $\gamma_n/\ell(\mathbf{p}^{(n)})$ instead of the original $\gamma_n$. The dynamic choice of strategy-dependent step size has the advantage of ``neutralizing'', in some sense, the unwanted impact of the cycle lengths on the algorithm's progression.
Of course, in applications, exact knowledge of $\ell(\mathbf{p})$ is an unrealistically strong requirement. However, even a coarse approximation $\tilde{\ell}(\mathbf{p}) \approx \ell(\mathbf{p})$ can impose a similar neutralizing effect, leading to substantial efficiency improvement. Thus, instead of using the estimator $\mathbf{G}$ to estimate $\ell(\mathbf{p})\mathbf{u(p)}$, given an approximation $\tilde{\ell}(\mathbf{p})$ we can use the estimator $\mathbf{G}/\tilde{\ell}(\mathbf{p})$ to estimate $(\ell(\mathbf{p})/\tilde{\ell}(\mathbf{p}))\mathbf{u(p)}$, or equivalently, divide the step size by $\tilde{\ell}(\mathbf{p})$. We believe this method to be effective especially in heavily loaded systems, in which the mean cycle length is typically large, and in addition, many approximation techniques tend to become more accurate under this regime. 

\begin{remark}
Natural approximations for ${\ell}(\mathbf{p})$ can be derived from available approximations for mean busy periods in queues. For instance, many approximations, in addition to some special-case exact expressions, are available for the mean busy period in the GI/G/1 queue (see, for instance, \cite{BN1992}). The length of an arbitrary regeneration cycle $L$ of a GI/G/1 can be interpreted as the number of service completions during the corresponding busy period $B$. Thus, the relation between $L$ and $B$ is expressed by $B=\sum_{j=1}^L Y_j$, where $Y_1, \dots, Y_L$ are the service lengths during the cycle. This relation implies $\E[L] = \E[B]/E[Y]$, where $Y$ is the length of an arbitrary service.
\end{remark}

\setcounter{example}{1}\begin{example}\textbf{{\rm (Continued)}}
In the model of Section \ref{sec:CR}, a back-of-the-envelope approximation for $\ell(\mathbf{p})$ is obtained by (erroneously) treating the two nodes, Server 1 and 2, as two independent Markovian systems with service rate $\mu$ in both systems. Thus, given $\mathbf{p}=(p,1-p)$, we have one M/M/1/1 system with arrival rate $\lambda^{[1]} = p\lambda$, and another independent M/M/1 system with arrival rate $\lambda^{[2]}=\lambda \cdot (1-p+\lambda^{[1]}/(\lambda^{[1]}+\mu))$. The stationary probability of finding each server idling is $\pi_0^{[1]} = \mu/(\mu+\lambda^{[1]})$ in the M/M/1/1 system and $\pi_0^{[2]} = 1-\lambda^{[2]}/\mu$ in the M/M/1 system, and by independence the joint distribution of the events is $\pi_0^{[1]}\pi_0^{[2]}$. By the PASTA principle, this coincides with the probability that a customer finds both servers idle upon arrival, therefore the mean number of arrivals between two subsequent such incidents is given by $\tilde{\ell}(\mathbf{p})=1/(\pi_0^{[1]}\pi_0^{[2]})$, and we shall use this quantity to approximate the true mean cycle length $\ell(\mathbf{p})$.
\hfill$\diamond$\end{example}

In the following numerical example, we employ variance reduction and dynamic step-size selection on the model described in Section \ref{sec:CR}. Services are exponentially distributed, $\mu$ and $c_w$ are normalized to 1, $\lambda=0.99$ and $c_s=5$. The initial strategy for the SA algorithm is $\mathbf{p}^{(1)}=(1/2, 1/2)$. In Figure \ref{fig:cr-game} we compare, over the first $N=5\times 10^4$ iterations, the performance of the crude implementation of the algorithm with that of the refined version. 
For the sequence of strategies, $\{ \mathbf{p}^{(n)} \}_{n\geq 1} = \{ ({p}^{(n)}, 1-{p}^{(n)}) \}_{n\geq 1}$, the figure depicts the convergence of the sequence ${p}^{(n)}$ to the (unique) equilibrium probability ${p}^e$, that corresponds with the unique equilibrium strategy $\mathbf{p}^e=(p^e, 1-p^e)$. The estimator used at iteration $n$ in the refined version (orange) is $\mathbf{G}^{(n)}-\left.\hat{\mathbf{\Psi}}^{(n)}\right.^{\prime} \mathbf{C}^{(n)}$, and the step size at iteration $n$ is $(1/n)\tilde{\ell}(\mathbf{p}^{(n)})$.
In the crude version (blue), we keep our original estimator $\mathbf{G}^{(n)}$, with comparable, non-dynamic step size given by $(1/n)\tilde{\ell}(\mathbf{p}^{(1)})$ for each $n$. For reference, we plot the correct equilibrium (red dashed horizontal line), which, for exponential services can be approached numerically up to any arbitrary precision level using the machinery developed in \cite{HS2017}.
\begin{figure}[H]
\centerline{\includegraphics[scale=.5]{"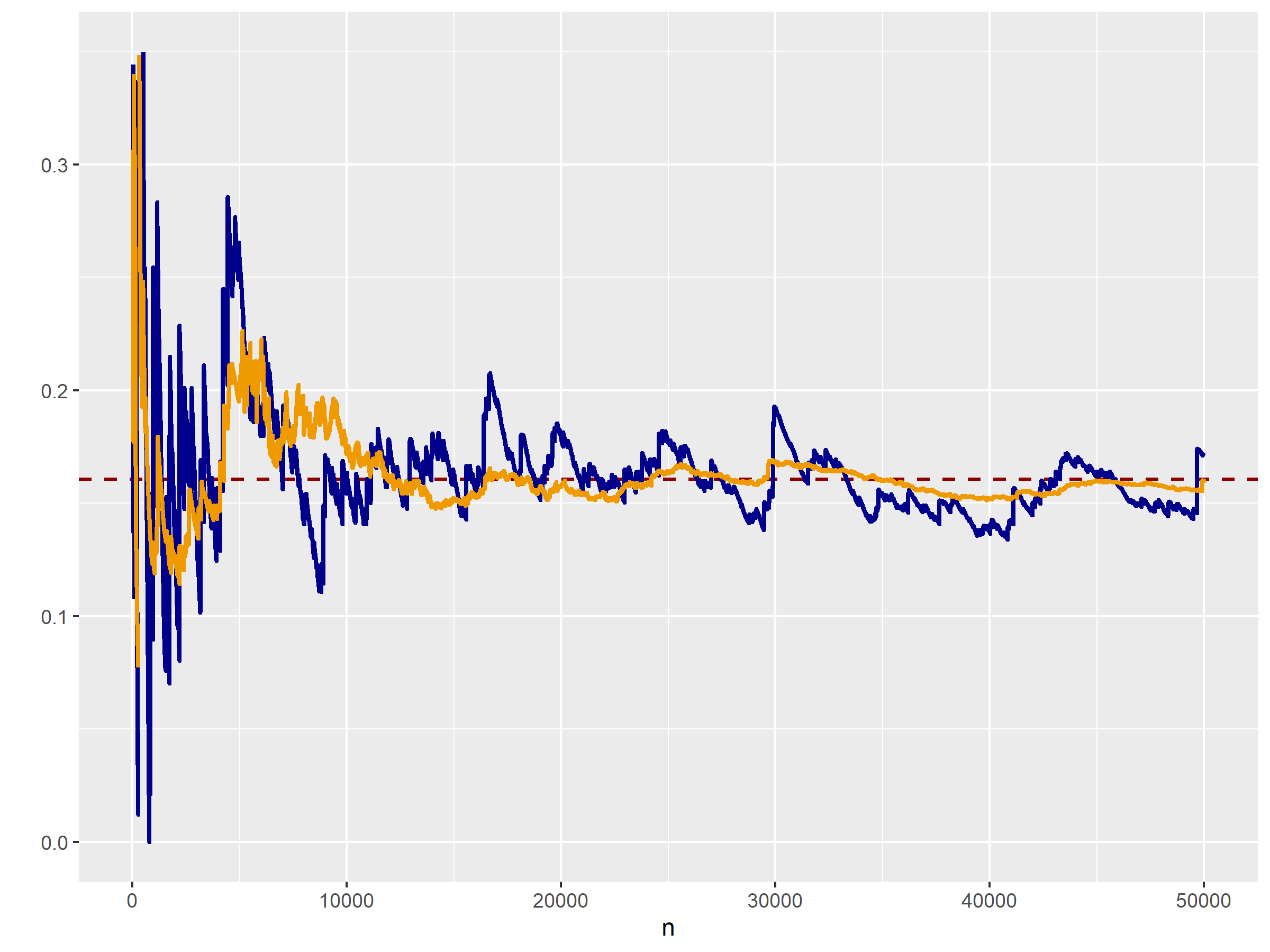"}}
\caption{Convergence of the SA algorithm as $n\to\infty$. The first coordinate of $\mathbf{p}^{(n)}=({p}^{(n)}, 1- {p}^{(n)})$, i.e., the sensing probability at stage $n$, is plotted vs. $n$. The blue curve corresponds to the crude algorithm's progression, and the orange corresponds to the refined version. The red dashed line depicts the correct equilibrium sensing probability, $p^e$, computed numerically in 8-digit precision.} \label{fig:cr-game}
\end{figure}

\section{Games with state information} \label{sec:observable}

In this section, we explain how to amend our SA algorithm to cope with so-called observable models. More precisely, we assume that customers, at the moment they arrive to the system and make their decision, have access to an information signal, which is a function of the system state. Unlike the approach described in \ref{sec:comment-obs}, the one we take here does not suffer from the same issue of dimensionality.

From a mathematical perspective, the previous formulation described in Section \ref{sec:model} can be seen as a special case of the model below when the function mapping states to signals is constant (i.e., non-informative). Although, from a nomenclature perspective, this understanding is somewhat inconsistent with how the terms 'observable' and 'unobservable' are  used in the literature to describe two disjoint classes of games. 
In what follows, we treat a strategy as a function from signals to (mixed) actions, and so, the dimension of our problem grows proportionally to the number of possible signals. \textcolor{black}{ Thus, to enable a direct implementation of our framework, we assume throughout that the number of possible signals is finite (if the actual information space is continuous then a suitable discretization of the information set is required).}

The new formulation suggests a practical method for the computation of equilibrium strategies in the observable GI/G/1 queue, a generalization of Naor's \cite{N1969} model, which we use as a guiding example in this section. The special case with Poisson arrivals, which itself is a generalization of \cite{N1969}, was thoroughly studied in \cite{K2011}. 

\subsection{Model description}
As before, we assume a renewal sequence of customer arrivals $\{A_n\}_{n\geq 1}$, each customer chooses an action from the set $\mathcal{A}=\{a_1, \dots, a_k\}$. The state of the system at any moment in time is represented by a vector $x$ of buffer contents in the state space $\mathcal{X}\subset \mathbb{R}_+^d$, in which $0^d\in\mathcal{X}$ is an element. Upon arrival, a customer is provided with an information signal which depends on the state at which they arrive. Let $\mathcal{I}$ be a countable set of information signals, thus, arriving at state $x\in\mathcal{X}$, a customer receives information $\psi(x)$, where $\psi:\mathcal{X}\to\mathcal{I}$ is a mapping from states to signals. We therefore reintroduce the concept of a strategy, now taking the interpretation of a  \emph{behavioral strategy}, i.e., a mapping $\mathbf{p}:\mathcal{I}\to \Delta$ from signals to distributions over actions. We denote by $\mathcal{P}$ the set of possible strategies. Hence, given $\mathbf{p}\in\mathcal{P}$, a customer who arrives at state $x\in\mathcal{X}$ will choose their action according to the distribution $\mathbf{p}(\psi(x)) = \big(p_1(\psi(x)), \dots, p_k(\psi(x))\big)\in \Delta$, which assigns a probability $p_j(\psi(x))$ for each action $a_j\in\mathcal{A}$. Consistent with the previous sections, we assume each strategy $\mathbf{p}\in\mathcal{P}$ gives rise to a regenerative system-state process $\{X_n(\mathbf{p})\}_{n\geq 1}$ embedded at arrival instants, with $0^d$ being a regenerative state and  $X_1(\mathbf{p})=0^d$. Thus, given $\mathbf{p}\in\mathcal{P}$, the stationary state $X(\mathbf{p})$, the cycle length $L(\mathbf{p})$, and the $r$-th moment of the cycle length $\ell^r(\mathbf{p})$, all follow the definitions introduced in Section \ref{sec:model}. Consistent with Sections \ref{sec:model} and \ref{sec:SA}, when facing state $x\in\mathcal{X}$ (hence a signal $\psi(x)$ is observed) and a realization $y$ of a random outcome $Y$, the value vector is $\mathbf{v}(x, y)=\big( v_1(x, y), \dots, v_k(x, y) \big)$, and  we let $\mathbf{\overline{v}}(X)= \E_Y \left[\mathbf{v}\big(X, Y\big) \: \big| \: X \right]$.

Assuming the state is drawn from the stationary distribution induced by $\mathbf{p}$, the signal observed is given by the r.v. $\psi(X(\mathbf{p}))$ supported on $\mathcal{I}$. We denote the (point) probability function of $\psi(X(\mathbf{p}))$ by $\xi_{\mathbf{p}}(s)=\P_\mathbf{p}(\psi(X(\mathbf{p}))=s )$ and its support by $\mathcal{I}(\mathbf{p})=\{s\in \mathcal{I} \text{ s.t. } \xi_{\mathbf{p}}(s)>0\}$. 
For conciseness, we ignore any irrelevant signals, i.e., we assume $\mathcal{I}=\bigcup_{\mathbf{p}\in\mathcal{P}}\mathcal{I}(\mathbf{p})$. 

For any signal $s\in \mathcal{I}(\mathbf{p})$ we define, analogous to \eqref{eq:up}, the ($k$-dimensional) \emph{conditional} expected value vector as
\begin{equation}\label{eq:up_i}
\mathbf{u}(\mathbf{p}\mid s)=\E_\mathbf{p}\bigg[ \mathbf{v}\big(X(\mathbf{p}), Y\big) \:\bigg|\: \psi(X(\mathbf{p}))=s  \bigg]=\frac{\E_\mathbf{p}\bigg[ \mathbf{v}\big(X(\mathbf{p}), Y\big)\cdot\mathit{1}(\psi(X(\mathbf{p}))=s )\bigg]}{\xi_{\mathbf{p}}(s)}.
\end{equation}
For completeness, when $s\notin \mathcal{I}(\mathbf{p})$ we let $\mathbf{u}(\mathbf{p}\mid s)= \mathbf{0}$. 
Finally, the best response set for a strategy $\mathbf{p}$ is given by
\begin{equation} \label{EQN:br-def-observable}
\mathcal{BR}(\mathbf{p})=\argmax_{\mathbf{q}\in\mathcal{P}} \sum_{s\in \mathcal{I}(\mathbf{p})} \xi_{\mathbf{p}}(s) \cdot \mathbf{u}(\mathbf{p}\mid s)'\mathbf{q}(s) ,
\end{equation}
and similarly to Definition \ref{def:SNE}, we have that $\mathbf{p}$ is a symmetric Nash equilibrium if $\mathbf{p}\in\mathcal{BR}(\mathbf{p})$.
The objective in the optimization problem of (\ref{EQN:br-def-observable}) can be thought of as the ex-ante expected utility (prior to receiving state information) of a customer who plays $\mathbf{q}$, when all other customers play $\mathbf{p}$. It can be further seen that the optimization problem is separable, in the sense that it can be solved independently for every $s\in\mathcal{I}(\mathbf{p})$. In other words, to constitute a best response, a strategy has to prescribe an optimal play, separately, under any information signal observed.

\subsection{Simulation and stochastic approximation}
Again, we construct our utility estimator by simulating a single regenerative cycle starting at the arrival of a customer to state $0^d$, with $X_j$, $j=1, \dots, L$, being the state at the $j$-th arrival in that cycle, and $L$ being the cycle length.
To apply the SA algorithm we construct an estimator for every $s\in\mathcal{I}$ defined by \begin{equation*}
\mathbf{G}(s) = \sum_{j=1}^{L} \mathbf{\overline{v}}(X_j)\cdot \mathit{1}(\psi(X_j)=s ).
\end{equation*}
Assuming $\ell^2(\mathbf{p})<\infty$, an immediate extension of Lemma \ref{lemma:unbiased_G} yields, for all $s\in\mathcal{I}$,
\begin{equation}\label{eq:unbiased-G-obs}
\E_{\mathbf{p}}[\mathbf{G}(s)] = \ell(\mathbf{p}) \cdot \xi_{\mathbf{p}}(s) \cdot \mathbf{u}(\mathbf{p}\mid s).
\end{equation}
Note that if $s\notin\mathcal{I}(\mathbf{p})$, then $\xi_{\mathbf{p}}(s)=0$, and equality in \eqref{eq:unbiased-G-obs} trivially holds as we have that
$\mathbf{G}(s)=0$ with probability 1.

In the case with state information, our algorithm constructs at each iteration a point estimate $\mathbf{G}(s)$ for every $s\in\mathcal{I}$. Thus, for practical reasons we make the assumption hereafter that $\mathcal{I}$ is a finite set. In many applications, a trivial bound on $|\mathcal{I}|$ can be derived, as demonstrated in Example \ref{exmp:obs-gg1} below. In other applications, different heuristics can be applied to cap the number of signals, with arguably ``minimal'' effect on the best-response function, however such considerations exceed the scope of the this paper. 

The progression of the SA algorithm is similar to the one described in the previous sections:
Let the step-size sequence be $\{\gamma_n\}_{n\geq 1}$ and the initial strategy be $\mathbf{p}^{(1)}\in\mathcal{P}$. At iteration $n$, with a corresponding strategy $\mathbf{p}^{(n)}$, we use one regeneration cycle to construct $|\mathcal{I}|$ estimators, $\{\mathbf{G}^{(n)}(s)\}_{s\in\mathcal{I}}$. The update at the next iteration for every $s\in\mathcal{I}$ is done by
\begin{equation}\label{eq:SA-iteration-obs}
    \mathbf{p}^{(n+1)}(s) = \pi_\Delta \left( \mathbf{p}^{(n)}(s) + \gamma_n \mathbf{G}^{(n)}(s)\right),\ n\geq 0,
\end{equation}
The following assumptions are sufficient for the convergence of the algorithm to an equilibrium strategy:

\renewcommand\theassumption{B\arabic{assumption}}
\setcounter{assumption}{0}
\begin{assumption}\label{assum:finite_El2-obs}
The second moment of the cycle length $L(\mathbf{p})$ is bounded for every strategy;  $\ell^2(\mathbf{p})<\infty$ for all $\mathbf{p}\in \mathcal{P}$.
\end{assumption}
\begin{assumption}\label{assum:finite_EG2-obs}
The second moment of the estimator $\mathbf{G}(s)$ is uniformly bounded on $\mathcal{P}$; \\ $\sup_{\mathbf{p}\in\mathcal{S}}\E_{\mathbf{p}} \Vert \mathbf{G}(s) \Vert^2<\infty$ for all $s\in \mathcal{I}$. 
\end{assumption}
\begin{assumption}\label{assum:finite_u-obs} For all $s\in \mathcal{I}$,  $\xi_{\mathbf{p}}(s) \cdot \mathbf{H}\mathbf{u}(\mathbf{p}\mid s)$ is continuous with respect to $\mathbf{p}$, and there exists a function $u^*:\mathcal{P}\to\mathbb{R}$ such that for all  $s\in\mathcal{I}$, $-\xi_{\mathbf{p}}(s) \cdot \mathbf{H}\mathbf{u}(\mathbf{p}\mid s) = \nabla_{\mathbf{p}(s)}u^*$.
\end{assumption}
\noindent Theorem \ref{thm:convergence-obs} introduces the key result of the section, which is a generalization of Theorem \ref{thm:convergence}:
\begin{theorem}\label{thm:convergence-obs}
Suppose Assumptions \ref{assum:finite_El2-obs}--\ref{assum:finite_u-obs} and \ref{assum:gamma_n} are satisfied. As $n\to\infty$, $\mathbf{p}^{(n)}\asarrow \mathcal{S}^e$, where $\mathcal{S}^e\subseteq\mathcal{P}$ is a set of equilibrium strategies, namely, $\mathbf{p}\in\mathcal{BR}(\mathbf{p})$ for every $\mathbf{p}\in \mathcal{S}^e$.
\end{theorem}
The proof of Theorem \ref{thm:convergence-obs} is similar to that of Theorem \ref{thm:convergence}, after introducing a suitable surrogate best-response function, analogous to the one defined in Equation (\ref{EQN:f-def}). This is described in further elaboration in Appendix \ref{sec:extensions-obs}. Noticeably, Theorem \ref{thm:convergence-obs} also implies that an equilibrium strategy exists. In fact, if one only cares to prove that an equilibrium strategy exists, then only the continuity of $\xi_{\mathbf{p}}(s) \cdot \mathbf{H}\mathbf{u}(\mathbf{p}\mid s)$ is needed (see Appendix \ref{sec:extensions-obs}). 

As in Section \ref{sec:conv}, the existence of a potential function $u^*$, as referred to in Assumption \ref{assum:finite_u-obs}, ensures that (with probability 1) in the limit, the iterate will not oscillate between non-equilibrium points. In the absence of Assumption \ref{assum:finite_u-obs}, even a non-stochastic variant of the algorithm can exhibit such undesirable oscillations in the limit.
Although the implementation of the algorithm does not require $u^*$ to be specified, verifying Assumption \ref{assum:finite_u-obs} can be difficult in practical settings. It is nevertheless the case that if the limit $\lim_{n\to\infty}\mathbf{p}^{(n)}$ exists then it must satisfy the equilibrium condition, regardless of whether Assumption \ref{assum:finite_u-obs} holds or not.

\begin{example} \label{exmp:obs-gg1}
Consider a first-come first-served GI/G/$1$ queue with observable queue lengths: The state $x=(x^{[1]}, x^{[2]}, \dots)$ is a vector of residual service times, where $x^{[i]}$ is the residual service of the $i$-th customer in line, which, for $i=1$, is that of the one in service. For simplicity, we assume here that service times are strictly positive iid random variables with mean $1/\mu < \infty$. Customer utility when joining the queue is linear in the waiting time, and is equal to zero when balking. Formally,  $\mathbf{v}(x, y)=(v_1(x, y), v_2(x, y))=(R-C\cdot(y+\Vert x \Vert_1), 0)$ for positive constants $R,C>0$, implying that for a random state $X$, $\overline{\mathbf{v}}(X)=(R-C\cdot(1/\mu+\Vert X \Vert_1), 0)$. 

The information available to a customer arriving at state $x$ is the number of jobs in the system, $\Vert x\Vert_0=|\{i \text{ s.t. } x^{[i]} > 0\}|$. Importantly, following the FCFS assumption, any customer not in service has an expected residual service time of $1/\mu$. Hence, from an arriving customer's standpoint, the (state-dependent) utility function $\overline{\mathbf{v}}$ can be further simplified into $\overline{\mathbf{v}}(X) = (R-C\cdot(\Vert X\Vert_0/\mu+X^{[1]}), 0)$. It is therefore immediate that when $\Vert x\Vert_0 \geq R\mu / C$, an arriving customer strictly prefers balking over joining. By defining $K = \lfloor R\mu / C \rfloor $ we can thus \emph{assume} that customers balk when observing $K+1$ (or more) customers in queue, and we restrict attention to decision epochs, which are the arrival instants of customers to states in $\mathcal{X}=\{x \text{ s.t. } \Vert x\Vert_0\leq K \}$. This implies that at decision epochs, the state-information function $\psi(x)=\Vert x\Vert_0$ defines a finite set of possible information signals (queue lengths), $\mathcal{I}=\{0, \dots, K\}$. A strategy $\mathbf{p}$ is a mapping from $\mathcal{I}$ to the two-dimensional simplex, $\mathbf{p}(s)=(p(s), 1-p(s)) \in \Delta$, with $p(s)\in[0,1]$ being the probability of joining when $s\in\mathcal{I}$ customers are observed. Hence, $\xi_{\mathbf{p}}(s)$ is the stationary probability of observing $s\in\mathcal{I}$ customers in queue at a decision epoch, provided all customers play according to $\mathbf{p}$. 

Given $\mathbf{p}$, the dynamics of the process imply that if customers balk with probability $1$ at queue length ${s}$, i.e., if $p(s)=0$, then any queue length $\tilde{s}>s$ is transient. Put formally, we have for all $\mathbf{p}\in\mathcal{P}$, $\mathcal{I}(\mathbf{p}) = \{ s\in\mathcal{I} \text{  s.t } \forall \tilde{s}<s, p(\tilde{s})>0 \}$. An equilibrium, adapted to this example, is a strategy $\mathbf{p}\in\mathcal{P}$ such that for all $s\in\mathcal{I}(\mathbf{p})$, if $u_1(\mathbf{p}\mid s)>0$ then $p(s)=1$, and if $u_1(\mathbf{p}\mid s)<0$ then $p(s)=0$. We note that if $\mathbf{p}$ is an equilibrium, there could be at most one $s\in\mathcal{I}(\mathbf{p})$ satisfying $u_1(\mathbf{p}\mid s)<0$, which also constitutes a threshold on the length of the queue (as long as customers play according to $\mathbf{p}$). 

Figure \ref{fig:obs-GG1} depicts a simulation in which we set $C=1$ and $R=1.7$, implying that $K=1$ and $\mathcal{I}=\{0,1\}$. Thus, every customer that observes two (or more) other customers automatically balks. In both panels, interarrival times are exponentially distributed with mean $1$, which is equal to the mean service time, yet the service distribution varies between the two panels. In panel (a), services are exponentially distributed, hence it is a particular case of Naor's observable M/M/1 model (\cite{N1969}) in which the unique equilibrium, $\{\mathbf{p}^e(s)\}_{s\in \{0,1\}}=\{(p^e(s), 1-p^e(s))\}_{s\in \{0,1\}}$ is easily found, with $p^e(0) = 1$ and $p^e(1) = 0$, i.e., join if and only if the system is empty. In panel (b), services are uniformly distributed over $[0,2]$. We note that the uniform distribution is of a decreasing mean residual lifetime, and by Proposition 3.1 in Kerner (\cite{K2011}) the equilibrium exists uniquely.
In both experiments we run the simulation for $N=10^5$ iterations of the algorithm, with $\gamma_n=2/n$ and initial strategy $\mathbf{p}^{(1)}(s)=(1/2, 1/2)$ for all $s\in\{0,1\}$. Interestingly, in the right panel, the equilibrium to which the algorithm converges is of a \emph{randomized threshold} type; when a customer observes another customer in service, they join with probability $p^{(N)}(1)\approx 0.37$, in accordance with the analysis and intuition provided in \cite{K2011}.
\end{example}

\begin{figure}[H]
\centering
     \begin{subfigure}[b]{0.45\textwidth}
         \centering
         \includegraphics[scale=.7]{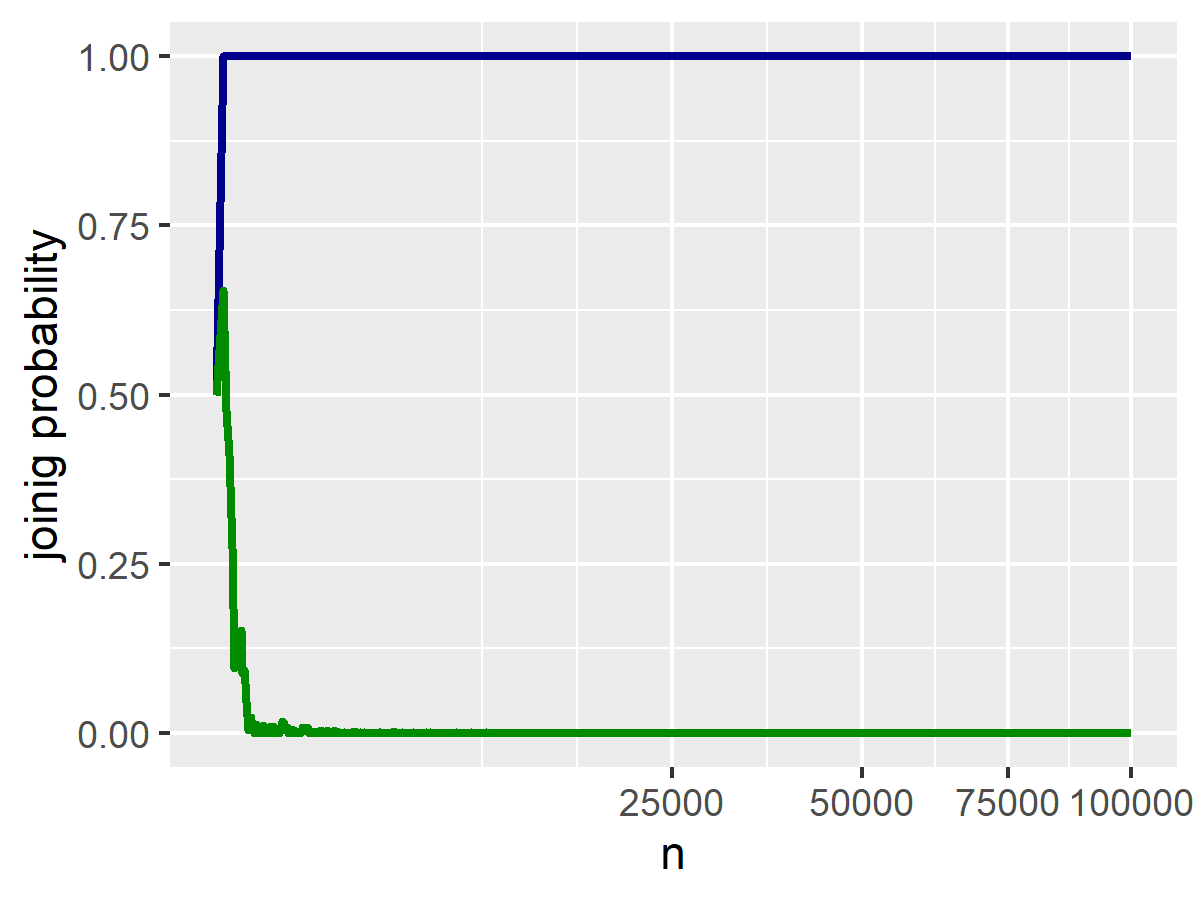}
         \caption{exponential service}
     \end{subfigure}
     \hfill
     \begin{subfigure}[b]{0.45\textwidth}
         \centering
         \includegraphics[scale=.7]{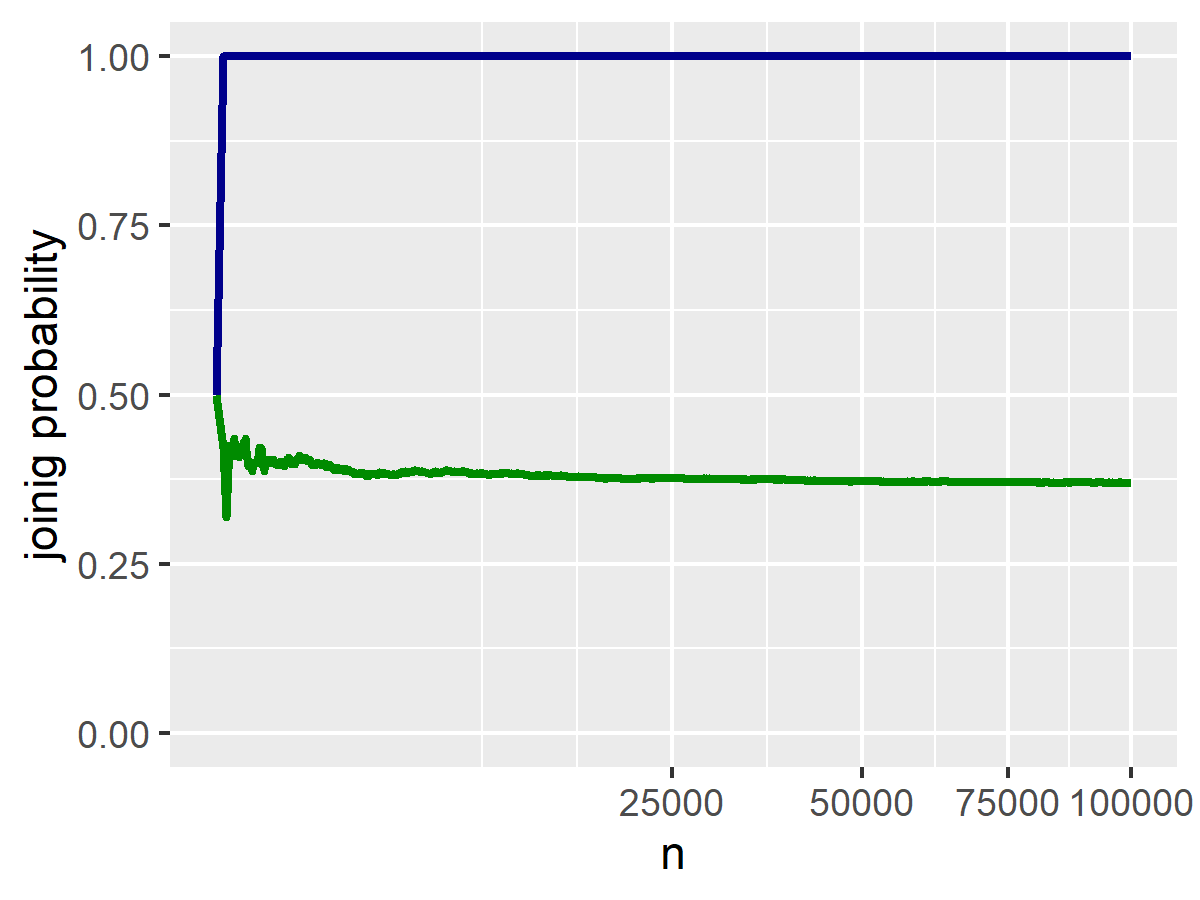}
         \caption{uniform service}
     \end{subfigure}
\caption{The convergence of the SA algorithm as $n\to\infty$. The joining probability  $p^{(n)}(s)$ is plotted vs. $n$ on a square-root scale, for $s=0,1$. In both panels,
the blue curve depicts ${p}^{(n)}(0)$ and the green curve depicts ${p}^{(n)}(1)$.
} \label{fig:obs-GG1}
\end{figure}

\begin{remark}
It is possible that some information signals are transient under some strategies but recurrent under different strategies. For instance, in the observable GI/G/1 queue, any strategy $\mathbf{p}(s)=(p(s), 1-p(s))$ with $\overline{s}\in \mathcal{I}(\mathbf{p})$ such that $p(\overline{s})=0$ (meaning customers balk when observing queue length $\overline{s}$), imply that the probability of observing a queue longer than $\overline{s}$ is zero, i.e., for all $s_0>\overline{s}$, $s_0 \notin \mathcal{I}(\mathbf{p})$. As a matter of fact, the equilibrium condition $\mathbf{p} \in \mathcal{BR}(\mathbf{p})$ does not impose any restriction on customer decisions in the (null) event they receive a zero-probability signal. 
Hassin and Haviv \cite{HH2002} introduce an adaptation of the concept of sub-game perfect Nash equilibria to queueing games, which is a refinement of the equilibrium concept, adding the requirement that customers play optimally also when arriving to transient states. In our setup, convergence to a sub-game perfect equilibrium cannot be guaranteed: Assuming the algorithm converges to an equilibrium $\mathbf{p}^e$ with $s_0 \notin \mathcal{I}(\mathbf{p}^e)$, the $s_0$-th (vector) component of the strategy will be updated, if at all, only finitely often; the tail of the sequence $\{\mathbf{p}^{(n)}(s_0)\}_{n\geq 0}$ is a constant vector (almost surely); and the limiting strategy $\mathbf{p}^e$ is possibly suboptimal at $s_0$.
\end{remark}

\section{Concluding remarks} \label{sec:conclusion}

This work introduces a robust stochastic-approximation algorithm that computes symmetric Nash equilibria in a general class of queueing games. The method involves simulating the system once and updating the strategy at regeneration times using the Robbins-Monro algorithm. As opposed to classical stochastic-optimization settings, a Nash equilibrium of a game with a discrete action space is not given by a first-order gradient condition. Therefore, an equivalent condition for a Nash equilibrium is formulated based on the root of a continuous function. This enables the construction of an iterative fixed-point method resembling those used in stochastic approximation. Results from renewal theory are used in order to construct an unbiased estimator of the total utility observed during a regeneration period. Verifiable conditions for almost-sure convergence of the algorithm are further provided. The method is shown to be useful for various interesting applications. 

We suspect that the suggested algorithm may be further used in order to find a socially optimal strategy, by considering the social cost of deviation instead of the individual optimality condition (see Haviv and Oz \cite{HO2021}). Loosely speaking, the concept of social cost of deviation basically allows one to reformulate the problem of identifying socially optimal strategies as a problem of finding equilibrium strategies in a game, and therefore deploy the SA algorithm to approximate these strategies. 

The framework developed here can be extended and improved in several ways, in particular, by introducing more accurate estimators for the utility function. A direct extension is to update the strategy after multiple regeneration cycles. Having several observations of regeneration cycles will reduce the variance of the estimator and will improve the accuracy of the iteration steps. Of course, this method entails an inherent trade-off because the algorithm can potentially ``waste time'' on simulating the process for strategies that are far off from the correct solution. Thus, a suitable dynamic choice of the simulation length at each iteration should be implemented.

The regeneration cycle approach comes along with both theoretical and practical advantages --  its implementation is straightforward and the estimator for $\mathbf{g}(\mathbf{p}^{(n)})$ is independent and unbiased. In this setting, the theory of stochastic approximations provides verifiable sufficient conditions for almost-sure convergence of the algorithm. An interesting future research direction is to explore the applicability of the SA framework to a more flexible iterative scheme that updates the strategy at every arrival and not just at regeneration times. This can be seen as a reinforcement-learning approach, although the resulting process is clearly not Markov. Implementation of such an algorithm should be just as easy, but theoretical analysis is much more involved. In particular, the estimation step is biased and is not independent of previous iterations. The theory on SA algorithms with correlated noise may be useful for constructing meaningful convergence conditions in this setting (see~\cite[Ch.~6]{KY2003}). This modification is appealing because it has the potential of coping with  strategies that drive the system out of the stability region. Proving that such a scheme of simple sequential interactions among agents eventually converges to an equilibrium will further strengthen the plausibility of equilibrium emergence in realistic systems and highlight the importance of studying equilibrium analysis in queues.

\section*{Acknowledgements}
The authors are grateful to Refael Hassin and Moshe Haviv for their advice and comments on the paper, as well as to Jim Dai and Binyamin Oz for several fruitful discussions. The authors would also like to thank the editors and reviewers for their detailed and helpful feedback. This research was supported in part by the Shenzhen Research Institute for Big Data International Postdoctoral Fellowship.

\section*{Authors}

\textbf{Ran I. Snitkovsky} is an Assistant Professor at the Coller School of Management, Tel Aviv University. He received his PhD in Operations Research in 2020 from the School of Mathematical Sciences, Tel Aviv University. He was a short-term visiting scholar at Tuck School of Business, and a postdoctoral fellow at Columbia Business School and Shenzhen Research Institute of Big Data, CUHK Shenzhen. His research revolves around the modeling and methodology of strategic, social, and behavioral interactions in congestion-prone systems, with a strong emphasis on economic and managerial insights.
\textbf{Contact details:} ran@tauex.tau.ac.il, Coller School of Management, Tel Aviv University, Tel Aviv 6997801, Israel.

\noindent\textbf{Liron Ravner} is a Senior Lecturer at the Department of Statisitcs, University of Haifa. He received his PhD in Statistics in the Hebrew Univsrsity of Jerusalem. His research interests focus on the intersection of Applied Probability, Statistics and Game Theory. 
\textbf{Contact details:} lravner@stat.haifa.ac.il, Department of Statistics, University of Haifa, Mount Carmel, Haifa 3498838, Israel.


\bibliographystyle{plain}
\newpage


\begin{appendices}

\section{Verification of approximate equilibrium for Section \ref{sec:motivation}}\label{sec:appVER}
For validation purpose, we estimate the mean virtual workloads for the two queues by simulating $2\times 10^6$ arrivals to each of the two GI/G/1 queues separately (assuming they are initially empty). The choice of simulation length follows the framework laid out in \cite{W1989}. Note that traffic intensities are relatively low in these two cases, and (recalling that $\mu_1=\mu_2=1$) are equal to $\lambda p^{(N)}_1 \approx 0.477$ in Server 1, and $\lambda p^{(N)}_2 \approx 0.3$ in Server 2. In the two cases below, inter-arrival times are distributed as $S=\sum_{i=1}^Z U_i$ where $U_i \sim \rm{Gamma}(0.1, 11)$, $Z$ is geometrically distributed with a different parameter at each case, and services are distributed as described:
\begin{enumerate}
\item $Z\sim {\rm Geom}(p^{(N)}_1)$ and services are distributed according to $F_1$ (which is $\rm{Beta}(10, 10)+0.5$);
\item $Z\sim {\rm Geom}(p^{(N)}_2)$ and services are distributed according to $F_2$ (which is  $\rm{Bernoulli}(0.1) \cdot 10$). 
\end{enumerate}
Under the algorithm's output strategy $\mathbf{p}^{(N)}$, we obtain (consistent) estimators for $\E[X^{[1]}]$ and $\E[X^{[2]}]$, in Cases 1 and 2, respectively, by averaging the observed workloads at arrival instants.
To construct asymptotic confidence intervals for the mean workloads we follow the CLT-based framework of \cite[Ch.~III]{AG2007}: We set the significance level $\alpha=0.005$. Under Case 1 we obtain an average virtual workload equal to 4.006 with an asymptotic $1-\alpha$ confidence interval $[3.996, 4.017]$. Under Case 2 we get a sample average equal to 4.005 with an asymptotic $1-\alpha$ confidence interval $[3.989, 4.021]$. This suggests that with more than $(1-\alpha)^2(\approx 0.99)$ certainty, our algorithm's output $\mathbf{p}^{(N)}$ is such that $(\E[v_1], \E[v_2]) \in \mathcal{I}_1 \times \mathcal{I}_2$, where $ \mathcal{I}_1 = [-0.017, 0.004]$ and $\mathcal{I}_2= [-0.021, 0.011]$. Hence, with high certainty (>99\%) our approximated solution $\mathbf{p}^{(N)}$ satisfies the criterion for an $\epsilon$-approximate Nash equilibrium (\cite{DMP2009}) for $\epsilon \leq \max_{(x,y) \in \mathcal{I}_1 \times \mathcal{I}_2} |x-y|= 0.028$.

\section{Proofs}\label{sec:appA}

\subsection{Proof of Lemma \ref{lem:equilibrium-equivalence}}\label{lem:equilibrium-equivalence-proof}
\begin{proof}
We show first that every symmetric Nash equilibrium $\mathbf{p}$ satisfies $\mathbf{p}=\mathbf{f}(\mathbf{p})$. Assume $\mathbf{p} \in \argmax_{\mathbf{q}\in\Delta} \mathbf{u(p)'q}$. Then any $\mathbf{q} \in \Delta$ satisfies $\mathbf{u(p)'q} \leq \mathbf{u(p)'p}$, or equivalently, $0 \leq \mathbf{u(p)'(p-q)}$. Assuming $\mathbf{q}\neq \mathbf{p}$, we have that 
\begin{align*}
\Vert \mathbf{p} + \mathbf{u}(\mathbf{p})-\mathbf{q} \Vert^2 &= \Vert \mathbf{p}-\mathbf{q} \Vert^2 + 2\mathbf{u}(\mathbf{p})'(\mathbf{p} - \mathbf{q}) + \Vert\mathbf{u}(\mathbf{p})\Vert^2 \\ 
&> \Vert\mathbf{u}(\mathbf{p})\Vert^2 = \Vert \mathbf{p} + \mathbf{u}(\mathbf{p})-\mathbf{p} \Vert^2,
\end{align*}
meaning that $\mathbf{p}= \pi_\Delta \big( \mathbf{p} + \mathbf{u}(\mathbf{p}) \big) = \mathbf{f}(\mathbf{p})$. 

We proceed to showing that if $\mathbf{p}$ is \emph{not} a symmetric equilibrium strategy then $\mathbf{p} \neq \mathbf{f}(\mathbf{p})$. Assume now that $\mathbf{p} \notin \argmax_{\mathbf{q}\in\Delta} \mathbf{u(p)'q}$ so that there exists some strategy $\mathbf{q}\in\Delta$, $\mathbf{q}\neq\mathbf{p}$, such that  $\mathbf{u(p)'(q-p)}>0$. 
Hence, we can choose $\theta \in (0,1]$ such that 
\[ 
\theta < \frac{2\mathbf{u(p)'(q-p)}}{\Vert\mathbf{q-p}\Vert^2}.
\]
Define $\tilde{\mathbf{q}} = (1-\theta)\mathbf{p}+\theta\mathbf{q}$, so that $ \mathbf{p} -\tilde{\mathbf{q}} = \theta(\mathbf{p-q})$. Note that $\Delta$ is a convex set, thus $\tilde{\mathbf{q}}\in\Delta$. Then,
\begin{align*}
\Vert \mathbf{p}+\mathbf{u(p)}-\tilde{\mathbf{q}}\Vert^2 &= \Vert \mathbf{p}-\tilde{\mathbf{q}}\Vert^2 + 2\mathbf{u(p)}'(\mathbf{p}-\tilde{\mathbf{q}}) +\Vert \mathbf{u(p)}\Vert^2 \\
&= \theta \left( \theta \Vert\mathbf{q}-\mathbf{p}\Vert^2
-2\mathbf{u(p)}'(\mathbf{q}-\mathbf{p})\right)+\Vert \mathbf{u(p)}\Vert^2 \\
& < \Vert \mathbf{u(p)}\Vert^2 = \Vert \mathbf{p}+\mathbf{u(p)}-\mathbf{p}\Vert^2,
\end{align*}
therefore $\mathbf{p}\neq \pi_\Delta \big( \mathbf{p} + \mathbf{u}(\mathbf{p}) \big)$,
meaning that $\mathbf{p} \neq \mathbf{f}(\mathbf{p})$.
\end{proof}

\subsection{Proof of Theorem \ref{thm:convergence} and auxiliary results}\label{sec:app_convergence-proof}
The proof of Theorem \ref{thm:convergence} essentially relies on casting the algorithm in \eqref{eq:SA_iteration} into a specific form, in which the conditions of
Theorems~2.4 in~\cite[Ch.~5.2]{KY2003} can be verified. Roughly speaking, this theorem relates the asymptotic behavior of the iterate $\mathbf{p}^{(n)}$ in \eqref{eq:SA_iteration} to solution paths of an appropriate ODE (Equation \eqref{eq:ODE_def} below), in which the strategy  ${\mathbf{p}}(\cdot)$ is understood as a continuous function of time. The proof will make use of the following result:
\begin{lemma}
Let $\mathbf{p}\in\Delta$ and $\mathbf{x}\in\mathbb{R}^k$, then 
\[ 
\pi_{\Delta}(\mathbf{p+x}) = \pi_{\Delta}(\mathbf{p+Hx}),
\]
where $\mathbf{H}$ is the orthogonal projection matrix onto $\mathcal{H}=\{\mathbf{x}\in \mathbb{R}^k \mid \mathbf{e}'\mathbf{x}=0\}$. \label{lem:aux-lemma-thm5}
\end{lemma}
\begin{proof}
Denote $\hat{\mathbf{p}}=\pi_{\Delta}(\mathbf{p+x})$, and suppose, by way of contradiction, that there exists some $\mathbf{q}\in\Delta$ that is closer than $\hat{\mathbf{p}}$ to $\mathbf{p+Hx}$, so that \[\Vert\mathbf{p+Hx-q}\Vert^2 < \Vert\mathbf{p+Hx}-\hat{\mathbf{p}}\Vert^2. \]
Note, since $\mathbf{p},\mathbf{q},\hat{\mathbf{p}}\in\Delta$, that $\mathbf{p-q}\in\mathcal{H}$ and $\mathbf{p}-\hat{\mathbf{p}}\in\mathcal{H}$, so that both $\mathbf{p-q}$ and $\mathbf{p}-\hat{\mathbf{p}}$
are orthogonal to $\mathbf{x-Hx}$. Using the (generalized) Pythagorean theorem, together with the above inequality we therefore have
\[\Vert\mathbf{p+x-q}\Vert^2=\Vert\mathbf{x-Hx}\Vert^2 + \Vert\mathbf{p+Hx-q}\Vert^2 < \Vert\mathbf{x-Hx}\Vert^2 + \Vert\mathbf{p+Hx}-\hat{\mathbf{p}}\Vert^2=\Vert\mathbf{p+x}-\hat{\mathbf{p}}\Vert^2, \]
which is in contradiction to the definition of $\hat{\mathbf{p}}$.
\end{proof}

To define the corresponding ODE let the set-valued map  $\mathcal{C}:\Delta\to 2^{\Delta}$
be defined as follows:
For an interior point $\mathbf{p}\in\Delta^\mathrm{o}$, let  $\mathcal{C}(\mathbf{p})=\{\mathbf{0}\}$; 
for a boundary point $\mathbf{p}\in \Delta \setminus \Delta^\mathrm{o}$, let $\mathcal{C}(\mathbf{p})$ be the infinite convex cone generated by the outer normals of the faces of $\Delta$ that contain $\mathbf{p}$.
Adopted to our framework, the projected ODE takes the form:
\begin{equation}\label{eq:ODE_def}
\mathbf{\dot{p}=Hu(p)+z}, \quad \text{and} \quad -\mathbf{z}(t)\in\mathcal{C}(\mathbf{p}(t)),\: \forall t\geq 0,
\end{equation}
where $\dot{\mathbf{p}}(\cdot)$ is the vector of time derivatives of $\mathbf{p}(\cdot)$; $\dot{\mathbf{p}}=(dp_1/dt, \dots, dp_k/dt)$. The value of $\mathbf{z}(\cdot)$ at each point in time is interpreted as the `minimum force' needed to prevent the path ${\mathbf{p}}(\cdot)$ from leaving $\Delta$. A \textit{stationary point} of \eqref{eq:ODE_def} is a point $\mathbf{p}^e \in \Delta$ such that $\mathbf{Hu}(\mathbf{p}^e)+\mathbf{z}=\mathbf{0}$ for some $\mathbf{z}$ such that $-\mathbf{z}\in \mathcal{C}(\mathbf{p}^e)$, meaning that the minimum force needed to push $\mathbf{p}^e+\mathbf{Hu}(\mathbf{p}^e)$ into $\Delta$ is given by $-\mathbf{Hu}(\mathbf{p}^e)$. Hence, if $\mathbf{p}^e$ is stationary,
\[ 
 \pi_{\Delta}(\mathbf{p}^e+\mathbf{Hu}(\mathbf{p}^e)) = \mathbf{p}+\mathbf{Hu}(\mathbf{p}^e)-\mathbf{Hu}(\mathbf{p}^e)= \mathbf{p}^e,
\]
and in light of Lemma \ref{lem:aux-lemma-thm5}, $\mathbf{p}^e=\mathbf{f}(\mathbf{p}^e)$, implying by Lemma \ref{lem:equilibrium-equivalence} that $\mathbf{p}^e$ is an equilibrium.

Define the \textit{limit set} of the ODE in \eqref{eq:ODE_def} as
\[
\mathcal{L}=\bigcup_{\mathbf{p}^0\in \Delta} \: \bigcap_{t\geq 0} \left\{ \mathbf{p}(s) \mid s \geq t,\: \mathbf{p}(0) = \mathbf{p}^0 \right\},
\]
that is, we consider solution paths of \eqref{eq:ODE_def} starting at an initial point $\mathbf{p}^0$, for each such path we pick all points visited infinitely many times, and then define $\mathcal{L}$ as the union of all such points for all possible initial values $\mathbf{p}^0\in \Delta$. Theorems~2.1 and 2.4 in~\cite[Ch.~5.2]{KY2003} provide sufficient conditions under which the stochastic iterate $\mathbf{p}^{(n)}$ converges (almost surely) to some subset of $\mathcal{L}$ as $n\to\infty$.

Clearly, if $\mathbf{p}^e$ is stationary, then a path starting at $\mathbf{p}^e$ will forever remain there, thus $\mathbf{p}^e\in\mathcal{L}$. However, in general $\mathcal{L}$ may include points that are not stationary, specifically when there are solutions of \eqref{eq:ODE_def} that do no converge to a unique point (e.g., circles). It is therefore of interest to introduce conditions under which such circular paths do not exist, and Assumption \ref{assum:finite_u} is one such condition. This is formulated in the lemma below:

\begin{lemma}\label{lem:limit-points-are-equilibria}
Let $\mathcal{S}^*$ be the set of stationary points of \eqref{eq:ODE_def}. If Assumption \ref{assum:finite_u} holds, then $\mathcal{S}^*=\mathcal{L}$, hence, every point in $\mathcal{L}$ is an equilibrium.
\end{lemma}
\begin{proof} As discussed, $\mathcal{S}^*\subseteq \mathcal{L}$ holds trivially, therefore it suffices to show that Assumption \ref{assum:finite_u} implies $\mathcal{L}\subseteq \mathcal{S}^*$.
Let $\nabla_{\mathbf{p}} u^*=-\mathbf{Hu}$, and let $\mathbf{p}(\cdot)$ be a (continuous) trajectory of \eqref{eq:ODE_def} starting at some $\mathbf{p}^0\in\Delta$. Our goal is to show that $\mathbf{p}(\cdot)$ converges to some set of stationary points, using $u^*(\cdot)$ as a Lyapunov function. Note that $u^*(\cdot)$ is continuous and therefore attains a finite minimum on $\Delta$, hence we can assume w.l.o.g that $u^*(\cdot)$ is positive on $\Delta$ (otherwise it can be increased by a constant without affecting the assumptions). For any $t\geq 0$, the term $\mathbf{p}(t)+\mathbf{Hu(p}(t))+ \mathbf{z}(t)$ is the projection of $\mathbf{p}(t)+\mathbf{Hu(p}(t))$ onto $\Delta$, hence
\begin{align*}
\Vert \mathbf{z}(t) \Vert
 &= \Vert \mathbf{p}(t)+\mathbf{Hu(p}(t)) - (\mathbf{p}(t)+\mathbf{Hu(p}(t))+ \mathbf{z}(t)) \Vert \\
 &\leq \Vert \mathbf{p}(t)+\mathbf{Hu(p}(t)) - \mathbf{p}(t) \Vert = \Vert \mathbf{Hu(p}(t)) \Vert. 
\end{align*}
Note that the inequality is strict in the case that $\mathbf{p}(t)$ is non-stationary. This, together with the Cauchy–Schwarz inequality implies, for any non-stationary $\mathbf{p}(t)$,
\[ 
-\Vert \mathbf{Hu(p}(t)) \Vert^2 <   -\Vert \mathbf{Hu(p}(t)) \Vert \cdot \Vert \mathbf{z}(t) \Vert \leq \big(\mathbf{Hu(p}(t))\big) ' \mathbf{z}(t).
\]
Denote $w(\mathbf{p}(t))=-(d/dt) u^*(\mathbf{p}(t))$. Using the gradient chain rule we therefore have, for non-stationary $\mathbf{p}(t)$,
\[
-w(\mathbf{p}(t)) =\frac{d}{dt} u^*(\mathbf{p}(t))= {\nabla_{\mathbf{p}}u^*(\mathbf{p}(t))}'\dot{\mathbf{p}}(t)=-\big(\mathbf{Hu(p}(t))\big)'(\mathbf{Hu(p}(t))+ \mathbf{z}(t)) < 0.
\]
Following the continuity of $\mathbf{Hu}(\cdot)$, $w(\cdot)$ is continuous on $\Delta$, and it can be seen from the above that for any $\mathbf{p}\in\Delta$, $w(\mathbf{p})=0$ if and only if $\mathbf{p}$ is stationary.
Furthermore, for all $t\geq 0$,
\[ 
0 \leq \int_{0}^t w(\mathbf{p}(s)) ds =
 u^*(\mathbf{p}^0) - u^*(\mathbf{p}(t)) \leq  u^*(\mathbf{p}^0), 
\]
where the last inequality is due to $u^*$ being positive.
It follows that $0\leq \int_{0}^\infty w(\mathbf{p}(s))ds < \infty$, thus, as $t\to\infty$, $w(\mathbf{p}(t))\to 0$, and by continuity of $w(\cdot)$ and $\mathbf{p}(\cdot)$,
\[ \mathbf{p}(t) \to \{\mathbf{p}^e \mid w(\mathbf{p}^e)=0 \},
\]
i.e., $\mathbf{p}(\cdot)$ converges to a set of stationary points, as desired.
\end{proof}

\begin{remark}\label{rem:conservative-u}
There are cases in which Assumptions \ref{assum:finite_u} can be easily verified. For example, if $k=2$, then each strategy $\mathbf{p}=(p, 1-p)\in\Delta$ is uniquely characterized by a single value $p\in[0,1]$ which parametrizes the whole simplex, and the matrix $\mathbf{H}$ is given by
\[
\mathbf{H} = \frac{1}{2}\cdot\begin{pmatrix}
1 & -1 \\
-1 & 1 
\end{pmatrix}.
\] 
Then for any $\mathbf{p}=(p, 1-p)\in\Delta$, denoting $\mathbf{s}=(s, 1-s)$ and $\tilde{\mathbf{e}}=\mathbf{e}_1-\mathbf{e}_2=(1, -1)$, we can simply define $u^*(\cdot)$ by
\[
u^*(\mathbf{p})=\int_{0}^{p} \mathbf{-\left(Hu(s)\right)}'\tilde{\mathbf{e}} \cdot ds = - \int_{0}^{p} \left( u_1(\mathbf{s}) - u_2(\mathbf{s})\right) ds = \int_{0}^{p} \left( u_2(s,1-s) - u_1(s,1-s)\right) ds.
\]

If $k>2$, then $u^*$ can be easily shown to exist if $\mathbf{u}(\cdot)$ itself is known to be conservative. In this case, for every parameterized curve $\mathbf{r}(s), s\in[a,b]\subseteq\mathbb{R}$, that lies entirely in $\Delta$, any tangential vector at a point along $\mathbf{r}$ must lie in $\mathcal{H}$; that is, $\dot{\mathbf{r}}(s)\in\mathcal{H}$ for all $s\in[a,b]$. Assuming  $-\mathbf{u}=\nabla_{\mathbf{p}}v^*$ for some $v^*:\Delta\to\mathbb{R}$, the line integral of $-\mathbf{Hu}(\cdot)$ along the curve $\mathbf{r}(\cdot)$ is given by
\[ 
\int_a^b -\big(\mathbf{Hu(r}(s))\big)' \dot{\mathbf{r}}(s) ds = \int_a^b -\mathbf{u(r}(s))' \mathbf{H} \dot{\mathbf{r}}(s) ds =  \int_a^b -\mathbf{u(r}(s))'  \dot{\mathbf{r}}(s) ds = v^*(\mathbf{r}(b))-v^*(\mathbf{r}(a)),
\]
where the first transition follows from the symmetry of $\mathbf{H}$, the second follows from $\dot{\mathbf{r}}(s)\in\mathcal{H}$, and the third is due to $\mathbf{u}(\cdot)$ being conservative.
This shows that the integral of $\mathbf{-Hu}(\cdot)$ along any curve in $\Delta$ depends only on its endpoints, hence $\mathbf{-Hu}(\cdot)$ is conservative.
\end{remark}

We next proceed to the proof of the theorem.
\subsubsection*{Proof of Theorem \ref{thm:convergence}}
\begin{proof}
Let $\epsilon_n=\gamma_n\ell(\mathbf{p}^{(n)})$, then the algorithm in \eqref{eq:SA_iteration} can be written as
\begin{equation}\label{eq:SA_modified}
\mathbf{p}^{(n+1)}=\pi_\Delta\left(\mathbf{p}^{(n)}+\epsilon_n \frac{\mathbf{G}^{(n)}}{\ell(\mathbf{p}^{(n)})}\right) = \pi_\Delta\left(\mathbf{p}^{(n)}+\epsilon_n \frac{\mathbf{H}\mathbf{G}^{(n)}}{\ell(\mathbf{p}^{(n)})}\right),
\end{equation}
where the second equality is due to Lemma  \ref{lem:aux-lemma-thm5}. Note that $\ell(\mathbf{p})\geq 1$, thus $\sum_{n=1}^\infty\epsilon_n\geq \sum_{n=1}^\infty\gamma_n=\infty$, in accordance with Equation (1.1) of \cite[Ch.~5.1]{KY2003}.
Assumption~\ref{assum:finite_El2} implies that $\ell(\mathbf{p})\leq M<\infty$ for all $\mathbf{p}\in\Delta$, hence by Assumption~\ref{assum:gamma_n} and the compactness of $\Delta$ we have that $\sum_{n=1}^\infty\epsilon_n^2\leq M^2\sum_{n=1}^\infty\gamma_n^2<\infty$ almost surely.  Therefore, condition (2.14) of Theorem~2.4 in \cite[Ch.~5.2]{KY2003} is satisfied.  

Next observe that the iterate $\mathbf{p}^{(n)}$ in \eqref{eq:SA_modified} is confined to the simplex, which is a compact and non-empty polyhedron. Therefore, Assumption~(A4.3.2) of \cite[Ch.~5.2]{KY2003} is satisfied. 
Assumptions~\ref{assum:finite_u} and~\ref{assum:finite_EG2} imply Assumptions (A2.1), (A2.3) of \cite[Ch.~5.2]{KY2003}.  By Lemma~\ref{lemma:unbiased_G} and the linearity of expectation we have that
\begin{align*}
\E_{\mathbf{p}}\left[ \frac{\mathbf{H}\mathbf{G}^{(n)}}{\ell(\mathbf{p}^{(n)})}\right]=\mathbf{H}\mathbf{u}(\mathbf{p}^{(n)}),
\end{align*}
hence Assumption (A2.7) of \cite[Ch.~5.2]{KY2003} is satisfied.  

Combining the above yields the conclusion of Theorem~2.4 in \cite[Ch.~5.2]{KY2003} that $\mathbf{p}^{(n)}$ as defined in \eqref{eq:SA_modified} converges to a set of limit points in $\mathcal{L}$, and together with Assumption \ref{assum:finite_u} we have, by Lemma \ref{lem:limit-points-are-equilibria}, that all points in $\mathcal{L}$ are equilibria, which concludes the proof of the theorem. 
\end{proof}

\subsection{Proof of  Proposition~\ref{prop:eps_NE}}\label{sec:app_norm}

\begin{proof}
Assume without loss that $\mathbf{u}$ is non-negative in all its coordinates and that $\min_{\mathbf{p}\in\Delta}\Vert\mathbf{u(p)}\Vert > 0$. This is possible since $\Vert \mathbf{u}(\cdot) \Vert_\infty$ admits a minimum on $\Delta$, and shifting all coordinates of $\mathbf{u}$ by the same constant neither affects the game nor the algorithm's progression. We
first note that if $\mathbf{u}$ is Lipschitz, then $\mathbf{u(p)'p}$, as a function of $\mathbf{p}$, is also (locally) Lipschitz on $\Delta$: Let $K_{\mathbf{u}}$ be a Lipschitz constant for $\mathbf{u}$ and let $M_{\mathbf{u}}$ be an upper bound on $\Vert \mathbf{u}(\cdot)\Vert$ over $\Delta$. Then for all $\mathbf{p,q}\in\Delta$
\begin{align*}
|\mathbf{u(p)'p} - \mathbf{u(q)'q}| &= |\mathbf{(u(p)-u(q))'p} + \mathbf{u(q)'(p-q)}| \\
&\leq \Vert\mathbf{u(p)-u(q)}\Vert\cdot\Vert \mathbf{p} \Vert + \mathbf{\Vert u(q)\Vert \cdot \Vert p-q \Vert} \\
&\leq \Vert\mathbf{u(p)-u(q)}\Vert + M_\mathbf{u}\mathbf{\Vert p-q \Vert} \\
&\leq (K_{\mathbf{u}}+M_{\mathbf{u}})\Vert\mathbf{p-q}\Vert.
\end{align*}
Thus, $K=K_{\mathbf{u}}+M_{\mathbf{u}}$ is a Lipschitz constant for both $\mathbf{u(p)'p}$ and $\mathbf{u(p)}$.
Note that $u^*$ is strongly convex, therefore there exists a unique equilibrium $\mathbf{p}^e$. Furthermore, $\mathbf{p}^e$ is a best response to itself, i.e., $p_i^e=0$ for any $i$ such that $u_i(\mathbf{p}^e)<\Vert\mathbf{u}(\mathbf{p}^e) \Vert_\infty$, therefore  $\mathbf{u}(\mathbf{p}^e)'\mathbf{p}^e = \Vert\mathbf{u}(\mathbf{p}^e)\Vert_\infty$. Thus, for all $\mathbf{p}\in\Delta$,
\begin{align*}
\Vert \mathbf{u}(\mathbf{p}^e) \Vert_\infty - \mathbf{u(p)'p} = \mathbf{u}(\mathbf{p}^e)'\mathbf{p}^e - \mathbf{u(p)'p} \leq  |\mathbf{u}(\mathbf{p}^e)'\mathbf{p}^e - \mathbf{u(p)'p}| \leq K\Vert\mathbf{p-p}^e\Vert ,  
\end{align*}
hence, $\Vert \mathbf{u}(\mathbf{p}^e) \Vert_\infty - K\Vert\mathbf{p-p}^e\Vert\leq \mathbf{u(p)'p}$.
For all $\mathbf{x}\in\mathbb{R}^k$, $\Vert \mathbf{x} \Vert_\infty \leq \Vert \mathbf{x} \Vert$, and using 
the triangle inequality together with the above yields, for all $\mathbf{p,q}\in\Delta$,
\begin{equation}\label{EQN:conv-rate-eps-eq}
\begin{split}
\mathbf{u(p)'q} & \leq \Vert \mathbf{u(p)} \Vert_\infty = \Vert \mathbf{u(p)} - \mathbf{u}(\mathbf{p}^e) + \mathbf{u}(\mathbf{p}^e)\Vert_\infty 
\leq \Vert\mathbf{u}(\mathbf{p}^e)\Vert_\infty+ \Vert \mathbf{u(p)} - \mathbf{u}(\mathbf{p}^e) \Vert_\infty \\
&\leq \Vert\mathbf{u}(\mathbf{p}^e)\Vert_\infty+ \Vert \mathbf{u(p)} - \mathbf{u}(\mathbf{p}^e) \Vert \leq \Vert\mathbf{u}(\mathbf{p}^e)\Vert_\infty+ K\Vert \mathbf{p} - \mathbf{p}^e \Vert \\
&= \Vert\mathbf{u}(\mathbf{p}^e)\Vert_\infty -K\Vert \mathbf{p} - \mathbf{p}^e \Vert+2K\Vert \mathbf{p} - \mathbf{p}^e \Vert \leq \mathbf{u(p)'p}+ 2K\Vert \mathbf{p} - \mathbf{p}^e \Vert.
\end{split}
\end{equation}

Following Assumptions \ref{assum:finite_El2}, for all $\mathbf{p}\in\Delta$, by Lemma \ref{lemma:unbiased_G}, $\E_{\mathbf{p}}  \mathbf{G} = \ell(\mathbf{p})\mathbf{u(p)}$. Since $\ell$ is a positive real function and  $\Vert\mathbf{u}(\cdot)\Vert$ is bounded away from $0$, using Jensen's inequality, 
\[
\ell(\mathbf{p}) = \frac{\Vert \E_{\mathbf{p}} \mathbf{G} \Vert}{\Vert \mathbf{u(p)}\Vert} \leq \frac{\E_{\mathbf{p}}\Vert \mathbf{G} \Vert}{\Vert \mathbf{u(p)}\Vert}  \leq \frac{\E_{\mathbf{p}}\Vert \mathbf{G} \Vert}{\min_{\mathbf{q}\in\Delta}\Vert \mathbf{u(q)}\Vert} .
\]
Therefore, by Assumption \ref{assum:finite_EG2}, we can choose constants $M_0,M_1>0 $ such that $\E_{\mathbf{p}} \Vert \mathbf{G} \Vert^2\leq M_0^2$ and $\ell(\mathbf{p})\leq M_1$ for all $\mathbf{p}\in\Delta$.
Setting $\gamma_n=\eta/n$, we rewrite the iterate $\mathbf{p}^{(n)}$ as in \eqref{eq:SA_modified} with the corresponding step size $\eta \ell(\mathbf{p}^{(n)})/n$, where we note that $\mathbf{H}\mathbf{G}^{(n)} / \ell(\mathbf{p}^{(n)})$ is unbiased for  $\mathbf{Hu}(\mathbf{p}^{(n)})$ and that $\eta \ell(\mathbf{p}^{(n)})/n \geq \eta /n $ (almost surely). Suppose that $\eta > 1/(2C)$, then
under the strong convexity assumption \eqref{eq:convex_potential}, we can invoke the result in \cite[Eq.~(2.8)]{NJLS2009}, by which we have for all $n\geq 1$ that
\begin{align*}
   \E\Vert\mathbf{p}^{(n+1)}-\mathbf{p}^e\Vert^2 & \leq \left(1-\frac{\eta \ell(\mathbf{p}^{(n)})}{n} \cdot 2C \right) 2 \E\Vert\mathbf{p}^{(n)}-\mathbf{p}^e\Vert^2+\left(\frac{\eta  \ell(\mathbf{p}^{(n)})}{n} \cdot M_0\right)^2 \\
   &\leq  \left(1-\frac{2 C \eta}{n}\right) 2 \E\Vert\mathbf{p}^{(n)}-\mathbf{p}^e\Vert^2+\left(\frac{\eta M_0  M_1}{n}\right)^2 .
\end{align*}
Define $M_\eta = \max\{ (\eta M_0 M_1)^2/(2C\eta-1) , 1 \}$. Then, similarly to \cite[Eq.~(2.9)]{NJLS2009}, by induction we have for all $n\geq 1$ that
\begin{align*}
   \E\Vert\mathbf{p}^{(n)}-\mathbf{p}^e\Vert^2\leq\frac{M_\eta}{n}.
\end{align*}
For any $\delta\in(0,1)$, using Markov's inequality we then get
\begin{align*}
 \P\left(2K\Vert \mathbf{p}^{(n)} - \mathbf{p}^e \Vert > \sqrt{n^{-\delta}}\right)= \P\left(\Vert \mathbf{p}^{(n)} - \mathbf{p}^e \Vert^2 > \frac{n^{-\delta}}{4K^2}\right) \leq \frac{4K^2\E \Vert \mathbf{p}^{(n)} - \mathbf{p}^e \Vert^2}{n^{-\delta}}\leq 4K^2M_\eta\cdot n^{\delta-1}.
\end{align*}
From \eqref{EQN:conv-rate-eps-eq}, for all $\mathbf{q}\in\Delta$,
\[
\mathbf{u}(\mathbf{p}^{(n)})'\mathbf{p}^{(n)} \geq \mathbf{u}(\mathbf{p}^{(n)})\apost \mathbf{q} - 2K\Vert \mathbf{p}^{(n)} - \mathbf{p}^e \Vert,
\]
hence, if $2K\Vert \mathbf{p}^{(n)} - \mathbf{p}^e \Vert\leq \sqrt{ n^{-\delta}}$, then $\mathbf{p}^{(n)}$ is a $\sqrt{ n^{-\delta}}$-equilibrium, and we conclude that
    \begin{align*}
        \P\left(\mathbf{u}(\mathbf{p}^{(n)})\apost\mathbf{p}^{(n)} \geq         \max_{\mathbf{q}\in\Delta}\mathbf{u}(\mathbf{p}^{(n)})\apost \mathbf{q} -\sqrt{ n^{-\delta}}\right)
        \geq  
        1-4\cdot K^2 M_\eta \cdot  n^{\delta-1},
    \end{align*}
which completes the proof.
\end{proof}

\section{Proofs for applications and examples}\label{sec:app_GGS}

\subsection{Proof of Lemma~\ref{lem:GG2_existence}}\label{lem:GG2_existence_proof} Recall that $1/\mu_m$ is the expected service time in queue $m$ and that $\mu_2\geq \mu_1$. For any strategy $\mathbf{p}=(p_1,p_2,p_3)$ the arrival rate to queue $m=1,2$ is $\lambda p_i$. Let $w_m(p_m)$ denote the expected stationary waiting time (service inclusive) in queue $m=1,2$ when strategy $\mathbf{p}=(p_1,p_2,p_3)$ is played by all customers, where we define $w_m(p_m)=\infty$ for any strategy such that $\lambda p_m>\mu_m$. 
Clearly, if $R-C/\mu_2<0$ then in equilibrium no customer will join either queue, hence $\mathbf{p}^e=(0,0,1)$ is the unique Nash equilibrium. 

If $R-C/\mu_2>0>R-C/\mu_1$, then in equilibrium, customers never join queue 1, but a positive fraction of them join queue 2. The function $w_2(p_2)$ is continuous and increasing with respect to $p_2$ (see \cite{W1974}) such that $w_2(0)=1/\mu_2$. Therefore, if $R-Cw_2(1)\geq 0$ then $\mathbf{p}^e=(0,1,0)$ is the unique equilibrium, and if $R-Cw_2(1)< 0$ then there exists some $p_2\in(0,1)$ such that $R-Cw_2(p_2)=0$, i.e., $\mathbf{p}^e=(0,p_2,1-p_2)$ is the unique equilibrium. 

Finally, we address the case $R-C/\mu_2>R-C/\mu_1>0$. If $R-C/w_2(1)>R-C/\mu_1$, then in equilibrium all customers join queue 2 and $\mathbf{p}^e=(0,1,0)$ is the unique equilibrium. If $R-C/w_2(1)\leq R-C/\mu_1$ then an equilibrium strategy must satisfy $p_1,p_2>0$ and $w_1(p_1)=w_2(p_2)$. As before, $w_m(p_m)$ is continuous and increasing in $p_m$, $m=1,2$, hence there are two possible cases, (a) There exists a pair $p_1+p_2<1$ such that $R-C/w_2(p_2)=R-Cw_1(p_1)=0$ and $\mathbf{p}^e=(p_1,p_2,1-p_1-p_2)$ is the unique Nash equilibrium; (b) There exists a pair $p_1+p_2=1$ such that $R-C/w_2(p_2)=R-Cw_1(p_1)>0$ and $\mathbf{p}^e=(p_1,p_2,0)$. This solution is unique because of the monotonicity of the waiting times implies that there is at most one solution $p_1$ to the equation $w_1(p_1)=w_2(1-p_1)$. \qed

\subsection{Proof of Proposition \ref{prop:GGn}}\label{prop:GGn-proof}
\begin{proof}
In light of Theorem \ref{thm:convergence}, to prove our Proposition \ref{prop:GGn}, it suffices to verify that Assumptions \ref{assum:finite_El2}--\ref{assum:finite_u} are satisfied in the underlying model. 

First, we note that $\mathbf{u}(\mathbf{p})$ is a separable function, in the sense that $u_i(\mathbf{p})$ is only a function of $p_i$ for every $i=1,\ldots,k$. By \cite{W1974}, $\mathbf{u}(\cdot)$ is continuous, and can be further expressed as the gradient of the sum of the antiderivatives, 
\[
\mathbf{u}(\mathbf{p}) = \nabla_{\mathbf{p}} \left(\sum_{i=1}^k \int u_i(p_i)dp_i\right).
\]
Thus, Assumption \ref{assum:finite_u} holds.
In what follows, we verify that Assumptions \ref{assum:finite_El2} and \ref{assum:finite_EG2} also hold. To this aim, we construct a coupling between the original system and a single-server queue, which we describe next. 

Given a strategy $\mathbf{p}=(p_1, \dots, p_{k})$, consider a single-server, FCFS queue, which we call the \emph{coupled system}, with identical arrival process $\{T_n\}_{n\geq 1}$ as in the original system. Consider the $j$-th arriving customer, whose arrival time (in both the original and the coupled systems) is $T_j$. We assume that if this customer chooses to join queue $m$ in original system (which occurs with probability $p_m$) with service demand $Y_m$, then customer $j$ in the coupled system with equal service demand $Y_m$ joins at the same time. If customer $j$ in the original system balks, then we say that in the coupled system, customer $j$ joins with 0 service demand. 

Recall our notation $X_j(\mathbf{p})=(X_j^{[1]}(\mathbf{p}), \dots, X_j^{[k-1]}(\mathbf{p}))$ for the vector of workloads in the original system observed by customer $j$, and let the (univariate) r.v. $\tilde{X}_j(\mathbf{p})$ represent its coupled counterpart, namely the workload observed by the $j$-th arriving customer in the coupled system. We assume both the original and the coupled system start empty, $X_1^{[m]}(\mathbf{p})=\tilde{X}_1(\mathbf{p})=0$ for all $m\in\{1,\dots,k-1\}$.  It can be seen from the construction that $X_j^{[m]}(\mathbf{p}) \leq \tilde{X}_j(\mathbf{p})$ with probability one, for all $m\in\{1,\dots,k-1\}$ and $j\geq 0$.
Furthermore, recall that $L(\mathbf{p})=\inf\{n\geq 1:\ X_{n+1}(\mathbf{p})=0^{k-1}(\mathbf{p})\}$ denotes the length of a random  regenerative cycle in the original system and let $\tilde{L}(\mathbf{p})=\inf\{n\geq 1:\ \tilde{X}_{n+1}(\mathbf{p})=0\}$ be its coupled counterpart. Then $L(\mathbf{p})\leq \tilde{L}(\mathbf{p})$ with probability one. Denoting by $\tau(\mathbf{p})$ the departure time of the $L(\mathbf{p})$-th customer in the original system and by $\tilde{\tau}(\mathbf{p})$ the departure time of the $\tilde{L}(\mathbf{p})$-th customer in the coupled system, it is immediate that the busy period in the original system, $B(\mathbf{p})=\tau(\mathbf{p}) - T_1$, and the busy period in the coupled system, $\tilde{B}(\mathbf{p})=\tilde{\tau}(\mathbf{p}) - T_1$, satisfy $B(\mathbf{p}) \leq \tilde{B}(\mathbf{p})$ almost surely. Since for all $m\in\{1,\dots,k-1\}$ and $j \in \{1, \dots, L(\mathbf{p})\}$ we know that $X_j^{[m]}(\mathbf{p}) \leq B(\mathbf{p})$, we conclude that $X_j^{[m]}(\mathbf{p}) \leq \tilde{B}(\mathbf{p})$ with probability one.

For convenience, assume without loss that $\mu_1 \leq \mu_2 \leq \dots \leq \mu_{k-1}$. For any strategy $\mathbf{p}=(p_1, \dots, p_k)$, let $\tilde{Y}(\mathbf{p})$ be a random variable that with probability $p_m$ is a random draw from the distribution $F_m$, $m\in\{1,\dots,k-1\}$, and with probability $p_k$ takes the value 0.
Hence, given $\mathbf{p}$, the coupled system forms a GI/G/1 queue in which service time is distributed similarly to $\tilde{Y}(\mathbf{p})$. Moreover, this queue is stable due to $\lambda < \mu_1$. By assumption, $\E[Y_m^4] <\infty$, and it follows that $\E[\tilde{Y}(\mathbf{p})^4] <\infty$. It is known (see \cite{T1985a}) that this implies $\Ep \tilde{L}(\mathbf{p})^4<\infty$ as well as $\Ep \tilde{B}(\mathbf{p})^4 < \infty$. Thus, $\ell^2(\mathbf{p}) \leq \Ep \tilde{L}(\mathbf{p})^2 < \infty$ and Assumption \ref{assum:finite_El2} follows.

Let $\kappa >0$ be a constant such that $\nu_m$ is $\kappa$-Lipschitz for all $m\in\{1,\dots,k-1\}$, and let $\xi=\max_{m=1,\dots,k-1}\nu_m(0) + \kappa \E[Y_m] $. We have, for every $m\in\{1,\dots,k-1\}$ and $j \in \{1, \dots, L(\mathbf{p})\}$,
\begin{align*}
\overline{v}_m(X_j(\mathbf{p})) &= \E [\nu_m(X_j^{[m]}(\mathbf{p}) + Y_m) \mid X_j^{[m]}(\mathbf{p})]  \\
&\leq \nu_m(0) + \kappa X_j^{[m]}(\mathbf{p}) + \kappa \E[Y_m] \leq \xi  + \kappa X_j^{[m]} (\mathbf{p})\leq \xi  + \kappa \tilde{B}(\mathbf{p})
\end{align*}
almost surely. Thus,
\[
\Vert \mathbf{G} \Vert^2=\Bigg\Vert \sum_{j=1}^{L(\mathbf{p})}  \mathbf{\overline{v}}(X_j(\mathbf{p}))\Bigg\Vert^2 \leq \Bigg\Vert\sum_{j=1}^{L(\mathbf{p})}  (\xi + \kappa \tilde{B}(\mathbf{p}))\mathbf{
e} \Bigg\Vert^2 = k\cdot{L}(\mathbf{p})^2(\xi + \kappa \tilde{B})^2 \leq k\cdot\tilde{L}(\mathbf{p})^2(\xi + \kappa \tilde{B}(\mathbf{p}))^2,
\]
and together with H\"older's inequality, 
\begin{align*}    
\E_{\mathbf{p}}\Vert \mathbf{G}  \Vert^2 \leq k \E_{\mathbf{p}} [\tilde{L}(\mathbf{p})^2(\xi + \kappa\tilde{B}(\mathbf{p}))^2]\leq k \cdot(\E_{\mathbf{p}} [\tilde{L}(\mathbf{p})^4])^\frac{1}{2} (\E_{\mathbf{p}} [(\xi+\kappa\tilde{B}(\mathbf{p}))^4])^\frac{1}{2}<\infty,
\end{align*}
where the last inequality follows from the finiteness of the first four moments of $\tilde{L}(\mathbf{p})$ and $\tilde{B}(\mathbf{p})$. To prove that Assumption \ref{assum:finite_EG2} holds, it suffices to show that
the first four moments of $\tilde{L}(\mathbf{p})$ and $\tilde{B}(\mathbf{p})$ are continuous functions of $\mathbf{p}$, which will imply that $\E_{\mathbf{p}}\Vert \mathbf{G} \Vert^2$ is uniformly bounded on $\Delta$. The remainder of the proof is therefore dedicated to proving the continuity of the moments. 

Let $\mathbf{p}^{(n)}$ be a convergent sequence of strategies, and denote $\mathbf{p}=\lim_{n\to\infty}\mathbf{p}^{(n)}$. Our goal is to prove $\E_{\mathbf{p}^{(n)}} \tilde{L}(\mathbf{p}^{(n)})^r \to \E_{\mathbf{p}} \tilde{L}(\mathbf{p})^r$ and $\E_{\mathbf{p}^{(n)}} \tilde{B}(\mathbf{p}^{(n)})^r \to \E_{\mathbf{p}} \tilde{B}(\mathbf{p})^r$ for $r\leq 4$. The convergence of $\mathbf{p}^{(n)}$ to $\mathbf{p}$ implies that $\tilde{Y}(\mathbf{p}^{(n)})\darrow\tilde{Y}(\mathbf{p})$ as $n\to\infty$.
Since the cycle length and busy period are continuous maps of the queue-length process and workload process, respectively, and since $H$ is a continuous distribution, we have by \cite{W1974} that $\tilde{L}(\mathbf{p}^{(n)})^r \darrow \tilde{L}(\mathbf{p})^r$ and $\tilde{B}(\mathbf{p}^{(n)})^r \darrow \tilde{B}(\mathbf{p})^r$. To show convergence of the means we shall use the dominated-convergence theorem based on the following construction:

Define $Z=\max_{m\in\{1,\dots, k-1\}} Y_m$. Since $\E Y_m^4<\infty$, we have $\E Z^4<\infty$, and clearly, $\E Z> 1/\mu_1$ (note that $Y_1, \dots Y_{k-1}$ are independent). For every $\epsilon \in (0,1)$ there exists some $N_\epsilon$ such that for all $n\geq N_\epsilon$ and $i\in \{1, \dots, k\}$
\[ {p}^{(n)}_i \leq (1-\epsilon)p_i+\epsilon. \]
Assume $\epsilon <  (1/\lambda-1/\mu_1) / (\E Z - 1/\mu_1)$.
Define $\bar{Y}(\mathbf{p})$ as a random variable taking the value $\tilde{Y}(\mathbf{p})$ w.p. $1-\epsilon$ and $Z$ otherwise. Hence $\E\bar{Y}(\mathbf{p})^4 <\infty$,
and
\[
\E\bar{Y}(\mathbf{p}) = (1-\epsilon)\E\tilde{Y}(\mathbf{p}) + \epsilon \E Z  \leq (1-\epsilon) \frac{1}{\mu_1} + \epsilon \E Z = \epsilon\cdot(\E Z -1/\mu_1) + 1/\mu_1 < 1/\lambda.
\]
Furthermore, for $n\geq N_\epsilon$ and $i\in\{1,\dots, k-1\}$, we note that $\tilde{Y}(\mathbf{p}^{(n)})$ takes the (random) value $Y_i$ with probability $p^{(n)}_i$, whereas $\bar{Y}(\mathbf{p})$ takes a value equal or larger than $Y_i$ with probability $(1-\epsilon)p_i+\epsilon \geq p^{(n)}_i$. Hence,
\[ \tilde{Y}(\mathbf{p}^{(n)}) \leq_{\mathrm{st}}  \bar{Y}(\mathbf{p}). \]
Define $\bar{L}$ and $\bar{B}$ as the cycle length and busy period of a GI/G/1 queue with interarrival distribution $H$ and  service distribution similar to that of $\bar{Y}(\mathbf{p})$. Thus, $\tilde{L}(\mathbf{p}^{(n)}) \leq_{\mathrm{st}} \bar{L}$ and $\tilde{B}(\mathbf{p}^{(n)}) \leq_{\mathrm{st}} \bar{B}$. 
Because $\E\bar{Y}(\mathbf{p})<1/\lambda$ and $\E\bar{Y}(\mathbf{p})^4<\infty$, it follows that $\E\bar{L}^r<\infty$ and $\E \bar{B}^r<\infty$ for any integer $r\leq 4$. Since for all $n\geq N_\epsilon$, $\tilde{L}(\mathbf{p}^{(n)})^r \leq_{\mathrm{st}} \bar{L}^r$ and $\tilde{B}(\mathbf{p}^{(n)})^r \leq_{\mathrm{st}} \bar{B}^r$, we conclude, using the dominated-convergence theorem, that
\[ 
\lim_{n\to\infty}\E_{\mathbf{p}^{(n)}} \tilde{L}(\mathbf{p}^{(n)})^r = \E_{\mathbf{p}} \tilde{L}(\mathbf{p})^r \qquad\text{and}\qquad \lim_{n\to\infty}\E_{\mathbf{p}^{(n)}} \tilde{B}(\mathbf{p}^{(n)})^r = \E_{\mathbf{p}} \tilde{B}(\mathbf{p})^r,
\]
hence, we have proven that $\E_{\mathbf{p}} \tilde{L}(\mathbf{p})^r$ and $\E_{\mathbf{p}} \tilde{B}(\mathbf{p})^r$ are continuous in $\mathbf{p}$.
\end{proof}

\subsection{Proof of Proposition \ref{prop:CR}}\label{pro/p:CR-proof}
\begin{proof}
As a special case of the model in \ref{sec:CR}, by setting $\mathbf{p}=\mathbf{e}_2$, we have that all customers join the queue of server 2, hence it can be modeled as an M/G/1 queue, which is stable due to $\lambda < \mu$. Then, by a coupling argument, for every  $\mathbf{p}\in\Delta$, the total workload process $X^{[1]}(t)+X^{[2]}(t)$ is dominated by that of the corresponding M/G/1 queue of  $\mathbf{p}=\mathbf{e}_2$, which, under the assumption $\E[Y^4]<\infty$, has finite cycle-length second moment, as explained in Example \ref{example:MG1} in Section \ref{sec:converg}. Hence, Assumption \ref{assum:finite_El2} holds. Following an explanation similar to that in Example~\ref{example:MG1} in Section \ref{sec:converg}, we conclude that Assumption \ref{assum:finite_EG2} is satisfied. Furthermore, we have in this example that $k=2$, hence Assumption \ref{assum:finite_u} holds, and together with the step size assumption \ref{assum:gamma_n} the conclusion of Theorem \ref{thm:convergence} follows. 
\end{proof}

\section{Unknown stability region}\label{sec:app_stability}

In strategic-queueing literature, it is often claimed, in a rather loose sense, that customer rationality imposes system stability, even when some strategies in the strategy space force the system out of its stability region. However, the current literature is lacking a rigorous mathematical formulation of a queueing game that reifies this intuitive argument. The framework presented in Section \ref{sec:model} does not formally cover models in which some strategies induce system instability, because customer utility is not well defined when the system is non-regenerative.
In this section we 
suggest a heuristic modification of the algorithm that attempts to deal with cases where some strategies render the system unstable, and the stability region of the system is not known a-priori. While we believe this method to be useful in practice, some important theoretical questions revolving stability in queueing games are left unsettled. Hopefully the results presented in this paper will motivate future research on this subject.

Suppose that the underlying system is positive recurrent only for strategies in some non-empty connected subset of the strategy space $\mathcal{S}\subset \Delta$. 
For the sake of the discussion here we will assume that $\mathcal{S}$ itself is unknown, but that $\Delta\setminus\mathcal{S}$ is compact, and that $\mathcal{S}$ contains an equilibrium point. Note that when $\mathcal{S}\neq \Delta$, the existence of an equilibrium in $\mathcal{S}$ does not follow directly from our Assumption \ref{assum:finite_u}. This can be easily seen through a degenerate variation of the Unobservable M/G/1 in Example \ref{example:MG1}, by setting $\lambda > \mu$, hence $\mathcal{S}=\{\mathbf{p}=(p, 1-p) \mid p \in [0, \mu / \lambda) \}$, and letting $C<0$ (namely, customers profit from waiting). Then, although $\mathbf{Hu}$ is a smooth conservative vector field on $\mathcal{S}$, any strategy in $\mathcal{S}$ must prescribe balking with some positive probability, while its best response is to join with probability 1. Furthermore,  $\mathcal{S}$ is non-compact (as is often the case for stability regions of queueing systems), and in this example an equilibrium does not exist even if customers are \textit{restricted} to choose a strategy in $\mathcal{S}$.

For our Definition \ref{def:SNE} of an equilibrium strategy to be valid, $\mathbf{u}$ has to be defined in the extended sense as
\[ 
\mathbf{u(p)} = \lim_{n\to\infty} \E \big[\overline{\mathbf{v}}(X_n(\mathbf{p}))\big]
\]
for all $\mathbf{p}\in\Delta$, including those strategies in $\Delta \setminus \mathcal{S}$. Moreover, this definition has to coincide with the stationary utility whenever a limiting distribution exists, i.e., $\overline{\mathbf{v}}$ has to be a function such that for all $\mathbf{p}\in\mathcal{S}$, if there exists a random variable $X(\mathbf{p})$ such that $X_n(\mathbf{p})\darrow X(\mathbf{p})$ as $n\to\infty$, then it must also hold that
\[ 
\mathbf{u(p)}=\lim_{n\to\infty} \E \big[\overline{\mathbf{v}}(X_n(\mathbf{p}))\big] = \E \big[\overline{\mathbf{v}}(X(\mathbf{p}))\big].
\]
Furthermore, to use our method, appropriate assumptions have to be introduced on the primitives that rule out the existence of equilibria in $\Delta \setminus \mathcal{S}$, and prevent the iterate from ``drifting'' outside $\mathcal{S}$. Such conditions should imply that after finitely many iterations, the iterate is absorbed in a compact subset $\tilde{\mathcal{S}}$ of $\mathcal{S}$ containing an equilibrium point. Rigorously formulating such conditions exceeds the scope of this paper. However, we conjecture that the following condition, which indeed holds in Example \ref{example:MG1} (with strictly positive waiting cost), is sufficient: For all $\mathbf{p}\notin\mathcal{S}$, $\mathbf{u(p)}'\mathbf{p}=-\infty$ and, additionally, for any sequence $\{\mathbf{p}_m\}_{m\geq 1}$ such that $\min_{\mathbf{q}\in\Delta\setminus\mathcal{S}}\Vert \mathbf{p}_m -\mathbf{q} \Vert\to0$, $\lim_{m\to\infty}\mathbf{u(p}_m)'\mathbf{p}_m=-\infty$.
We next present a modified version of the SA algorithm that terminates busy cycles when they exceed a threshold that increases as the number of iteration grows.

At iteration $n\geq 1$, given a strategy $\mathbf{p}^{(n)}$, let
\begin{equation}\label{eqn:G-tilde}
    \tilde{\mathbf{G}}^{(n)}=\sum_{j=1}^{L\wedge\beta_n} \mathbf{\overline{v}}(X_j),
\end{equation}
where $L\wedge\beta=\min\{L,\beta\}$ and $\{\beta_n\}_{n\geq 1}$ is an increasing sequence of integers such that $\lim_{n\to\infty}\beta_n=\infty$. The modified SA algorithm is defined as
\begin{equation}\label{eq:Iter_stability}
        \mathbf{p}^{(n+1)}=\pi_\Delta\left(\mathbf{p}^{(n)}+\gamma_n\tilde{\mathbf{G}}^{(n)}\right) .
\end{equation}
Clearly $\tilde{\mathbf{G}}^{(n)}$ is in general not an unbiased estimator of $\mathbf{g}(\mathbf{p}^{(n)})$, even if $\mathbf{p}^{(n)}\in\mathcal{S}$. However, we next argue that by carefully choosing the threshold
sequence $\{\beta_n\}$, 
$\tilde{\mathbf{G}}^{(n)}$ can be made asymptotically unbiased to $\mathbf{g}(\mathbf{p}^{(n)})$.
Note that the number of summands in \eqref{eqn:G-tilde} is finite for all $n$, thus, under natural conditions on $\overline{\mathbf{v}}$, $\tilde{\mathbf{G}}^{(n)}$ is integrable for all $\mathbf{p}\in\Delta$, and we define $\mathbf{g}_n(\mathbf{p})=\E_\mathbf{p} \tilde{\mathbf{G}}^{(n)}$.

Our goal therefore is to find an appropriate threshold sequence $\{\beta_n\}$. In our suggested heuristic we attempt to choose $\{\beta_n\}$ such that for every $\mathbf{p}\in\mathcal{S}$, the sequence $\{\mathbf{g}_n(\mathbf{p})\}_{n\geq 1}$  converges sufficiently fast to $\mathbf{g}(\mathbf{p})$ as $n\to\infty$.
Note that for all $\mathbf{p}\in\mathcal{S}$, $\mathbf{g}_n(\mathbf{p})$ can be written as
\begin{align*}
    \mathbf{g}_n(\mathbf{p})&=\E_{\mathbf{p}}\tilde{\mathbf{G}}^{(n)} =\E_{\mathbf{p}}\left[\sum_{j=1}^L\mathbf{\overline{v}}(X_j)\right]-\E_{\mathbf{p}}\left[\sum_{j=\beta_n+1}^L\mathbf{\overline{v}}(X_j)\mathbf{1}_{\{L>\beta_n\}}\right] \\
    &= \mathbf{g}(\mathbf{p})-\E_{\mathbf{p}}\left[\sum_{j=\beta_n+1}^L\mathbf{\overline{v}}(X_j)\mathbf{1}_{\{L>\beta_n\}}\right] \\ 
    &= \mathbf{g}(\mathbf{p})-\P_\mathbf{p}(L>\beta_n)\cdot \sum_{j=\beta_n+1}^L\E_{\mathbf{p}} \big[\mathbf{\overline{v}}(X_j)\mid L>\beta_n \big] =  \mathbf{g}(\mathbf{p})- \xi_n(\mathcal{\mathbf{p}}),
\end{align*}
where 
\[
\xi_n(\mathcal{\mathbf{p}}) = \P_\mathbf{p}(L>\beta_n)\cdot \sum_{j=\beta_n+1}^L\E_{\mathbf{p}} \big[\mathbf{\overline{v}}(X_j)\mid L>\beta_n \big].
\]
Assume that $\mathbf{Hu}$ is continuous and conservative on $\mathcal{S}$ (as in \ref{assum:finite_u}) and that $\{\gamma_n\}_{n\geq 1}$ satisfies \ref{assum:gamma_n}. Let  $\tilde{\mathcal{S}}\subset\mathcal{S}$ be a compact subset of $\mathcal{S}$ in which an equilibrium lies, such that $\sup_{\mathbf{p}\in\mathcal{S}}\E_{\mathbf{p}}\Vert \mathbf{G}\Vert^2<\infty$ (as in \ref{assum:finite_EG2}), and in addition, for all $\mathbf{p}\in\tilde{\mathcal{S}}$,
\begin{equation}\label{eqn:xi-n-condition}
\sum_{n=1}^\infty \gamma_n  \ell(\mathbf{p})\cdot \vert \xi_n({\mathbf{p}}) \vert < \infty.
\end{equation}
If for all sufficiently large $n$, the iterate remains (with probability 1) in $\tilde{\mathcal{S}}\subset\mathcal{S}$, 
then, we conjecture that the conclusion of Theorem \ref{thm:convergence} holds. Rigorous results for iteration-dependent functions can be found in \cite[Ch.~6]{KY2003}.

Noticeably, there is an inherent trade-off in the choice of the sequence $\{\beta_n\}_{n\geq 1}$: if $\{\beta_n\}_{n\geq 1}$ diverges too slowly, $\vert \xi_n({\mathbf{p}}) \vert$ diminishes slowly and may fail to satisfy the condition in \eqref{eqn:xi-n-condition}, whereas if $\{\beta_n\}_{n\geq 1}$ diverges very fast, there is a risk that the algorithm will spend long periods of time simulating arrivals to an unstable system. A reasonable choice for $\beta_n$ should satisfy, as an example, for all sufficiently large $n$ and $\mathbf{p}\in\mathcal{S}$,
\begin{equation}\label{eq:beta_tail}
    \P_\mathbf{p}(L>\beta_n)\leq  z(\mathbf{p})^{n}\ ,
\end{equation}
for some function $z:\mathcal{S}\to(0,1)$. If $z$ is bounded over the compact set  $\tilde{\mathcal{S}}$, then this  choice of $\{\beta_n\}_{n\geq 1}$, by the Borel-Cantelli lemma, ensures that the cycle length hits the threshold $\beta_n$ only finitely many times. This can be done, for example, if the cycle-length distribution has a light tail,
i.e., if for all $\mathbf{p}\in\tilde{\mathcal{S}}$ there exists an $\alpha$ such that $\P(L>x)\leq \alpha^x$ (the latter is implied, for example, in the M/G/1 queue, assuming that the service-time distribution is light tailed; see \cite{DMT1980}). Then taking $\beta_n=\kappa n$ for some constant $\kappa>0$ yields $\P_\mathbf{p}(L>\beta_n)\leq \alpha^{\kappa n}$.

In Figure \ref{fig:unstable-mm1} below we present results from a simulation of the Unobservable M/G/1 queue of Example \ref{example:MG1}, which demonstrate the convergence of the iterate to the true equilibrium depicted by the red dashed line. 
In this example, service-time are exponentially distributed with mean $1/\mu=1$ (thus, the queue is an M/M/1), and the (potential) arrival rate $\lambda$ is $2$, hence, $\lambda > \mu$ and $\mathcal{S}=\{\mathbf{p}=(p, 1-p) \mid p \in [0, 1 / 2) \}$. The utility parameters are $R=5$ and $C=1$, implying that the equilibrium joining probability is given by $p^e=0.4$. We implement the modified algorithm, for  $n=1,\dots, 10^6$, using $\tilde{\mathbf{G}}^{(n)}$ as the estimator at iteration $n$, with $\beta_n=n$, step size $\gamma_n=1/n$ and initial strategy $\mathbf{p}^{(1)}=(1/2, 1/2)$. In our experimentation, over the simulation horizon, the cycle length hits the threshold $\beta_n$ a few dozens times, with $n=6,175$ being the last iteration index for which the threshold was hit. 
The last iteration $n$ 
for which the the iterate exists the stability region is $n=5,059$, where we have $\mathbf{p}^{(n)}=(0.53,0.47)\notin \mathcal{S}$.
This provides a strong empirical evidence that the iterate is absorbed in a compact set in the domain of attraction of the equilibrium point $\mathbf{p}^e$.
\begin{figure}[H]
\centerline{\includegraphics[scale=.5]{"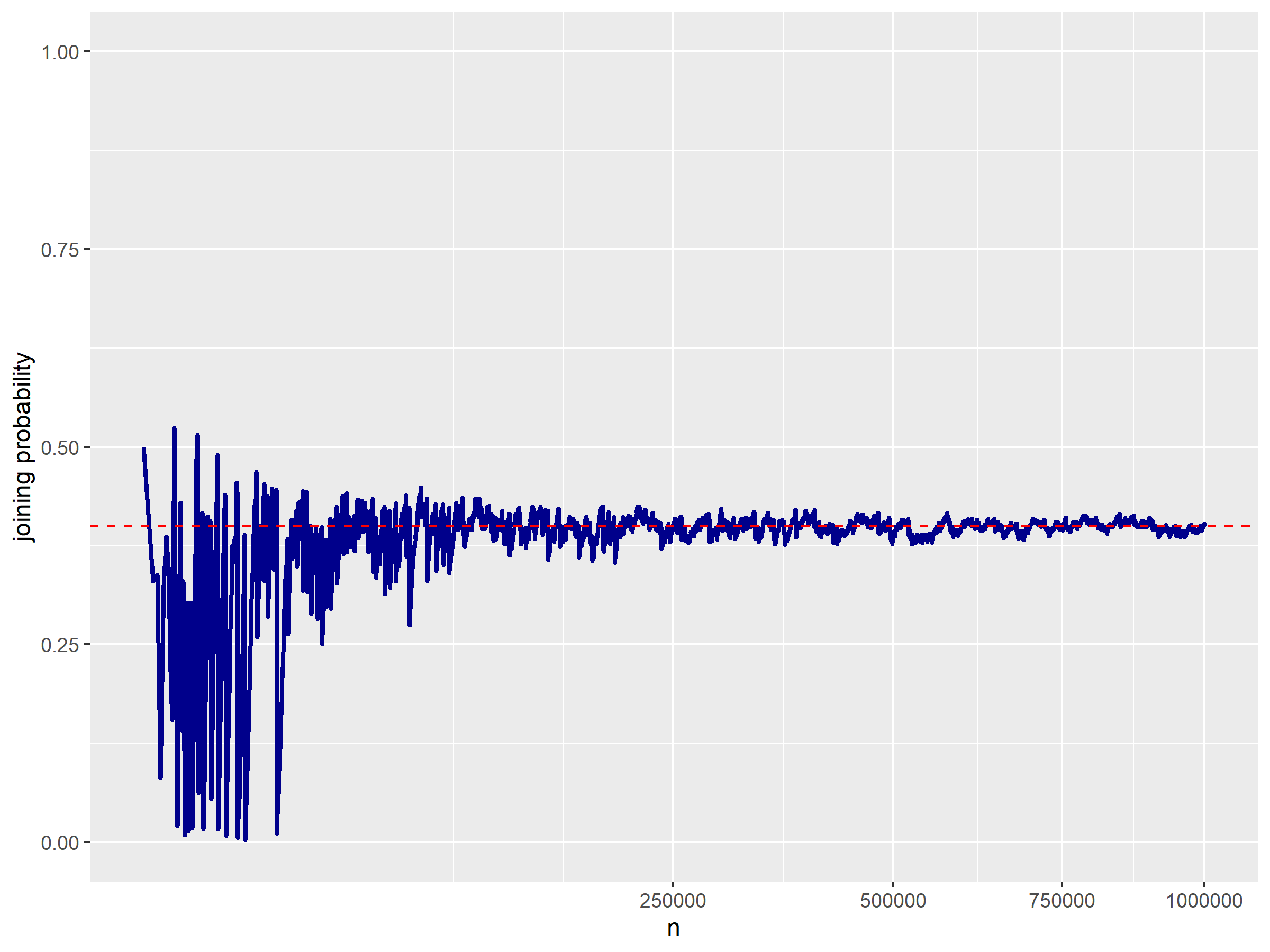"}}
\caption{Convergence of the modified SA algorithm as $n\to\infty$. The first coordinate of $\mathbf{p}^{(n)}=({p}^{(n)}, 1- {p}^{(n)})$, i.e., the joining probability at stage $n$, is plotted (in blued) vs. $n$ on a square-root scale. The red dashed line depicts the correct equilibrium joining probability, which is given by $p^e=0.4$.} \label{fig:unstable-mm1}
\end{figure}


\section{Extension to games with state information}\label{sec:extensions-obs}
In this section we extend our concept of the surrogate best-response function to the model with state information. Theorem~\ref{thm:convergence-obs} follows directly from this construction. 

First, we redefine the surrogate best-response function as
\begin{equation} \label{EQN:f-def-observable}
\mathbf{f}(\mathbf{p}) 
= \argmin_{\mathbf{q}\in\mathcal{P}} \sum_{s\in \mathcal{I}} \left\Vert \mathbf{p}(s) + \xi_\mathbf{p}(s)\cdot\mathbf{Hu}(\mathbf{p}\mid s)  -\mathbf{q}(s) \right\Vert^2.
\end{equation}
Similar to (\ref{EQN:br-def-observable}), the optimization problem in (\ref{EQN:f-def-observable}) is also separable, motivating the following definition, for all $s\in\mathcal{I}$:
\begin{equation} \label{EQN:fs-def}
\mathbf{f}(\mathbf{p}\mid s) = \pi_{\Delta} \left( \mathbf{p}(s) + \xi_\mathbf{p}(s)\cdot\mathbf{Hu}(\mathbf{p}\mid s) \right) = \pi_{\Delta} \left( \mathbf{p}(s) + \xi_\mathbf{p}(s)\cdot\mathbf{u}(\mathbf{p}\mid s) \right)
\end{equation}
which is a real function from $\Delta$ to itself. The second transition in \eqref{EQN:fs-def} follows Lemma \ref{lem:aux-lemma-thm5}.
We can therefore think of $\mathbf{f}$ as a vector valued function, $\mathbf{f}:\mathcal{P}\to\mathcal{P}$ (as opposed to $\mathcal{BR}$, which is a correspondence, $\mathcal{BR}:\mathcal{P}\to2^{\mathcal{P}}$),
and in addition, $\mathbf{p}=\mathbf{f(p)}$ if and only if $\mathbf{p}(s)=\mathbf{f(p}\mid s)$ for all $s\in \mathcal{I}$. 
\begin{lemma}\label{lem:aux-lemma-observable}
Suppose for every $s\in \mathcal{I}$, $\xi_\mathbf{p}(s)\cdot\mathbf{Hu}(\mathbf{p}\mid s)$ is continuous on $\mathcal{P}$ (as a function of $\mathbf{p}$). Then $\mathbf{f}(\cdot)$ is continuous on $\mathcal{P}$, and therefore admits a fixed point. In addition, any strategy is a fixed point for $\mathbf{f}$ if and only if it is an equilibrium, i.e., $\mathbf{p}=\mathbf{f}(\mathbf{p})$ if and only if $\mathbf{p}\in\mathcal{BR}(\mathbf{p})$, hence, an equilibrium exists.
\end{lemma}
\begin{proof}
Under the conditions of Lemma \ref{lem:aux-lemma-observable}, the objective function in the LHS of (\ref{EQN:f-def-observable})
is continuous in $(\mathbf{p},\mathbf{q})$ on the product space $\mathcal{P}\times\mathcal{P}$, and $\mathcal{P}$ is compact, thus by the maximum theorem $\mathbf{f}$ is continuous over $\mathcal{P}$. As a result, by Brouwer's fixed-point theorem, $\mathbf{f}$ admits a fixed point in $\mathcal{P}$. The rest of the proof is similar to that of Lemma \ref{lem:equilibrium-equivalence}:

Suppose $\mathbf{p}$ is an equilibrium, i.e., $\mathbf{p}\in\mathcal{BR}(\mathbf{p})$. Then for all $s\in \mathcal{I}$, the expression $\xi_\mathbf{p}(s)\mathbf{u}(\mathbf{p}\mid s)'\mathbf{r}$ is maximized at $\mathbf{r}=\mathbf{p}(s)$, hence $\xi_\mathbf{p}(s)\mathbf{u}(\mathbf{p}\mid s)'(\mathbf{p}(s) - \mathbf{q}(s))\geq 0$ for all $\mathbf{q}\in\mathcal{P}$. Consider a strategy $\mathbf{q} \neq \mathbf{p}$, thus, for any $s\in\mathcal{I}$ such that $\mathbf{q}(s) \neq \mathbf{p}(s)$, 
\begin{align*}
\Vert \mathbf{p}(s) + \xi_\mathbf{p}(s)\mathbf{u}(\mathbf{p}\mid s)-\mathbf{q}(s) \Vert^2 &= \Vert \mathbf{p}(s)-\mathbf{q}(s) \Vert^2 + 2\xi_\mathbf{p}(s)\mathbf{u}(\mathbf{p}\mid s)'(\mathbf{p}(s) - \mathbf{q}(s)) + \Vert\xi_\mathbf{p}(s)\mathbf{u}(\mathbf{p}\mid s)\Vert^2 \\ 
&> \Vert\xi_\mathbf{p}(s)\mathbf{u}(\mathbf{p}\mid s)\Vert^2 = \Vert \mathbf{p}(s) + \xi_\mathbf{p}(s)\mathbf{u}(\mathbf{p}\mid s)-\mathbf{p}(s) \Vert^2.
\end{align*}
Hence, $\mathbf{p}(s)=\mathbf{f}(\mathbf{p}\mid s)$, meaning that $\mathbf{p}=\mathbf{f}(\mathbf{p})$. Assume on the contrary that  $\mathbf{p}$ is not an equilibrium, thus we can choose $\mathbf{q}\in\mathcal{P}$ and $s\in \mathcal{I}(\mathbf{p})$, $\xi_{\mathbf{p}}(s)>0$, such that $\mathbf{u}(\mathbf{p}\mid s)'\mathbf{q}(s) > \mathbf{u}(\mathbf{p}\mid s)'\mathbf{p}(s)$.
Then a parameter $\theta \in (0,1)$ can be chosen such that 
\[\theta<\min\left\lbrace \frac{2\xi_{\mathbf{p}}(s)\mathbf{u(p}\mid s)'(\mathbf{q}(s)-\mathbf{p}(s))}{\Vert \mathbf{q}(s)-\mathbf{p}(s)\Vert^2}, 1 \right\rbrace, \] 
noting from the assumptions that the right-hand side is strictly positive. We then define
$\tilde{\mathbf{q}}=\mathbf{p}+\theta(\mathbf{q-p})$ to obtain, after rearrangement,
\[ 
\begin{split}
& \left\Vert \mathbf{p}(s) + \xi_\mathbf{p}(s)\cdot\mathbf{u}(\mathbf{p}\mid s)  -\tilde{\mathbf{q}}(s) \right\Vert^2 \\
 & \qquad = \theta \left( \Vert \mathbf{q}(s)-\mathbf{p}(s)\Vert^2 +  2\xi_{\mathbf{p}}(s)\mathbf{u(p}\mid s)'(\mathbf{q}(s)-\mathbf{p}(s)) \right) +  \Vert\xi_\mathbf{p}(s)\mathbf{u}(\mathbf{p}\mid s)\Vert^2 \\
& \qquad  < \Vert\xi_\mathbf{p}(s)\mathbf{u}(\mathbf{p}\mid s)\Vert^2 = \left\Vert \mathbf{p}(s) + \xi_\mathbf{p}(s)\cdot\mathbf{u}(\mathbf{p}\mid s)  -\mathbf{p}(s) \right\Vert^2,
\end{split}
\]
and therefore $\mathbf{p}(s) \neq \mathbf{f}(\mathbf{p}\mid s)$, implying that $\mathbf{p} \neq \mathbf{f}(\mathbf{p})$.
\end{proof}

\subsubsection*{Proof of Theorem \ref{thm:convergence-obs}}
\begin{proof}
Denote $|\mathcal{I}|=l$. With slight abuse of notation, we treat a strategy $\mathbf{p}\in\mathcal{P}$ hereafter as a real $k\times l$ matrix, whose $s$-th column, $(p_{1,s}, \dots p_{k,s})$, is given by the vector $\mathbf{p}(s)$ defined previously. Define the mapping $\bar{\mathbf{u}}:\mathcal{P}\to\mathbb{R}^{k\times l}$ such that $\bar{\mathbf{u}}(\mathbf{p})$ is a matrix whose $s$-th column, $(\bar{u}_{1,s}(\mathbf{p}), \dots, \bar{u}_{k,s}(\mathbf{p}))$, is given by
$\xi_{\mathbf{p}}(s)\mathbf{Hu(p}\mid s)$. Note that the strategy space $\mathcal{P}$ is the product space of $l$ copies of $\Delta$, hence it is a compact polyhedron in $\mathbb{R}^{k\times l}$. Furthermore, the projection of an element in $\mathbf{x}\in\mathbb{R}^{k\times l}$ onto $\mathcal{P}$ is computed simply by projecting each column of $\mathbf{x}$ onto $\Delta$.
Then the corresponding ODE for the algorithm in \eqref{eq:SA-iteration-obs} is
\begin{equation}\label{eq:ODE_def-obs}
\mathbf{\dot{p}=\bar{u}(p)+z}, \quad  -(z_{1,s}(t), \dots z_{k,s}(t) )\in\mathcal{C}((p_{1,s}(t), \dots p_{k,s}(t) )),\: \forall t\geq 0, 1\leq s\leq l,
\end{equation}
where $(z_{1,s}(t), \dots z_{k,s}(t))$ is the $s$-th column vector of $\mathbf{z}$, and $\mathcal{C}$ is the set-valued map defined in Section \ref{sec:app_convergence-proof}. 
Let $u^*:\mathcal{P}\to\mathbb{R}^l$ be the potential function, and let $\nabla u^*(\mathbf{p})$ be its gradient at $\mathbf{p}$, i.e., a real $k\times l$ matrix whose element at the $i$-th row and $s$-th column is $(\partial/\partial p_{i,s})u^*(\mathbf{p})$. Then Assumption \ref{assum:finite_u-obs} implies that $\nabla u^*(\mathbf{p})=-\bar{\mathbf{u}}(\mathbf{p})$. 
By repeating the same steps as in the proof of Lemma \ref{lem:limit-points-are-equilibria}, we have that each point $\mathbf{p}$ in the limit set of \eqref{eq:ODE_def-obs} is stationary, i.e., satisfies $\mathbf{p}=\pi_{\mathcal{P}}(\mathbf{p}+\mathbf{\bar{u}(p)})$, and therefore is an equilibrium strategy.

Similarly to the proof of Theorem \ref{thm:convergence}, we define  $\epsilon_n=\gamma_n\ell(\mathbf{p}^{(n)})$, observing that $\sum_{n=1}^\infty\epsilon_n= \infty$ and $\sum_{n=1}^\infty\epsilon_n^2 < \infty$ due to Assumptions \ref{assum:finite_El2-obs} and \ref{assum:gamma_n}. The algorithm in \eqref{eq:SA-iteration-obs} can be therefore written as
\begin{equation}
\mathbf{p}^{(n+1)}= \pi_{\mathcal{P}}\left(\mathbf{p}^{(n)}+\epsilon_n \frac{\mathbf{H}\mathbf{G}^{(n)}}{\ell(\mathbf{p}^{(n)})}\right),
\end{equation}
where $\mathbf{G}$ is a $k\times l$ matrix with $\mathbf{G}(s)$ as its $s$-th column. From \eqref{eq:unbiased-G-obs} we have that 

\begin{align*}
\E_{\mathbf{p}}\left[ \frac{\mathbf{H}\mathbf{G}^{(n)}}{\ell(\mathbf{p}^{(n)})}\right]=\bar{\mathbf{u}}(\mathbf{p}^{(n)}).
\end{align*}
Thus, as in the proof of Theorem \ref{thm:convergence}, we  invoke Theorem~2.4 in \cite[Ch.~5.2]{KY2003}, to conclude that $\mathbf{p}^{(n)}$ converges to a set of limit points of \eqref{eq:ODE_def-obs}, and since every limit point is an equilibrium strategy, we obtain the result.
\end{proof}

\textcolor{black}{
\section{Multiple customer types}\label{sec:heterogeneous-customers}
Many of our results can be easily extended to queueing games with heterogeneous types of customers interacting in the same system. Below we explain, through a representative example, how to utilize our framework to approximate an equilibrium strategy with two types of customers. The model we consider in this section is a natural extention of the Unobservable M/G/1 presented in Example \ref{example:MG1}; We assume customers arrive to the system according to a Poisson process with rate $\lambda$, and services are iid with mean $\frac{1}{\mu}$.
We consider two customer types, 1 and 2, indexed by $\theta\in\{1,2\}$, and assume that each customer's type is drawn independently of the system state. Similarly to the single-type model, we denote by $R_\theta$ and $C_\theta$ the service evaluation and waiting cost rate for type $\theta\in\{1,2\}$. As before, we let $X_j$ represent the virtual workload upon the $j$-th arrival instant. We redefine $\overline{\mathbf{v}}(X_j)$ as a mapping of a possible system state to a matrix in $\mathbb{R}^{2\times 2}$, with one conditional expected utility vector for each type,
\[ 
\overline{\mathbf{v}}(X_j)= \begin{pmatrix} \overline{\mathbf{v}}_1(X_j) & \overline{\mathbf{v}}_2(X_j)\end{pmatrix} = \begin{pmatrix} R_1-C_1\cdot(X_j+1/\mu) & R_2-C_2\cdot(X_j+1/\mu) \\ 0 & 0 \end{pmatrix},
\]
and the vector of expected utilities is again expressed by  
$\mathbf{u}(\mathbf{p})=\E_\mathbf{p} \big[\mathbf{\overline{v}}(X(\mathbf{p})) \big]$.
Knowing their type, each customer chooses whether to join or balk without observing the system state. Thus, a strategy profile, $\mathbf{p}\in \Delta^2$, is now represented by a pair of point distributions over actions, $\mathbf{p}=(\mathbf{p}_1 \quad \mathbf{p}_2 )\in \Delta^2$. For a matrix $\mathbf{m}\in \mathbb{R}^{2\times 2}$, let $(\mathbf{m})_\theta$ denote its $\theta$-th column, $\theta\in\{1,2\}$.
An equilibrium is then defined as a strategy profile $\mathbf{p}^e\in \Delta^2$ such that for every $\theta\in\{1,2\}$,
\[ 
(\mathbf{p}^e)_\theta\in\argmax_{\mathbf{q}\in \Delta}( \mathbf{u}(\mathbf{p}^e))_\theta'\mathbf{q}.
\]
}

\textcolor{black}{
Since customer types are independent of the system state, we can consider an arrival of an arbitrary customer (of any type) to an idling server as a time point of system regeneration. Each iteration of the SA algorithm includes a simulation of a single regeneration cycle and adaptation of the strategy profile.
By indexing customers of a single regeneration cycle by $j\in\{1,2,\dots, L\}$ (where $L$ is the cycle length) and denoting by $X_j$ the system state upon arrival of the $j$-th customer, we can construct our utility estimator, 
\[ \mathbf{G}=\sum_{j=1}^L \mathbf{\overline{v}}(X_j). \]
}

\textcolor{black}{
Consistent with the notation in previous sections, we let $\mathbf{p}^{(n)}$, $\mathbf{G}^{(n)}$, and $\gamma_n$
denote the strategy profile, utility estimator, and step size, respectively, at the $n$-th iteration of the algorithm. Setting an arbitrary initial strategy $\mathbf{p}^{(1)}\in\Delta^2$, the update scheme is given by
\begin{equation}
    \mathbf{p}^{(n+1)} = \pi_{\Delta^2} \left( \mathbf{p}^{(n)} + \gamma_n \mathbf{G}^{(n)}\right),\ n\geq 1.
\end{equation}
}

\textcolor{black}{
Figure \ref{fig:2types-mm1} below depicts the convergence of the iterate to the (unique) equilibrium for an M/M/1 queue with proportion $\alpha=0.3$ of the customers being type 1 and the rest type 2. The arrival process is Poisson with (total) rate of $\lambda=0.7$ arrivals per unit time, and services are exponentially distributed with parameter 1. The cost and reward parameters for the two types are given by $R_1 =3$, $C_1 =1$, $R_2 =5$ and $C_2 =2$. The blue curve corresponds to the joining probability of type-1 customers (i.e., the first component of $\mathbf{p}^{(n)}_1$), while the green curve corresponds to the joining probability of type-2 customers (i.e., the first component of $\mathbf{p}^{(n)}_2$), as functions of the iteration number, $n$. In this example, $\frac{R_1}{C_1}>\frac{R_2}{C_2}$, implying that type-2 customers join the system only when type-1 customers join with probability 1. In this specific example, the unique equilibrium strategy profile can be easily calculated as we have that the joining probability of type-2 customers in equilibrium (assuming all type-1 customers join) is equal to $\left(\mu-\alpha\lambda-\frac{C_2}{R_2}\right) / ((1-\alpha)\lambda)$, so that 
\[ 
(\mathbf{p}^e)_1=\begin{pmatrix} 1 \\ 0 \end{pmatrix}, \quad (\mathbf{p}^e)_2=\begin{pmatrix} 0.796 \\ 0.204 \end{pmatrix}.
\]
For $\gamma_n=\frac{1}{2n}$ and initial joining probability of $0.5$ for each type, after $10^6$ iterations, we obtain an approximation error of the order of $10^{-4}$.}
\begin{figure}[H]
\centerline{\includegraphics[scale=.5]{"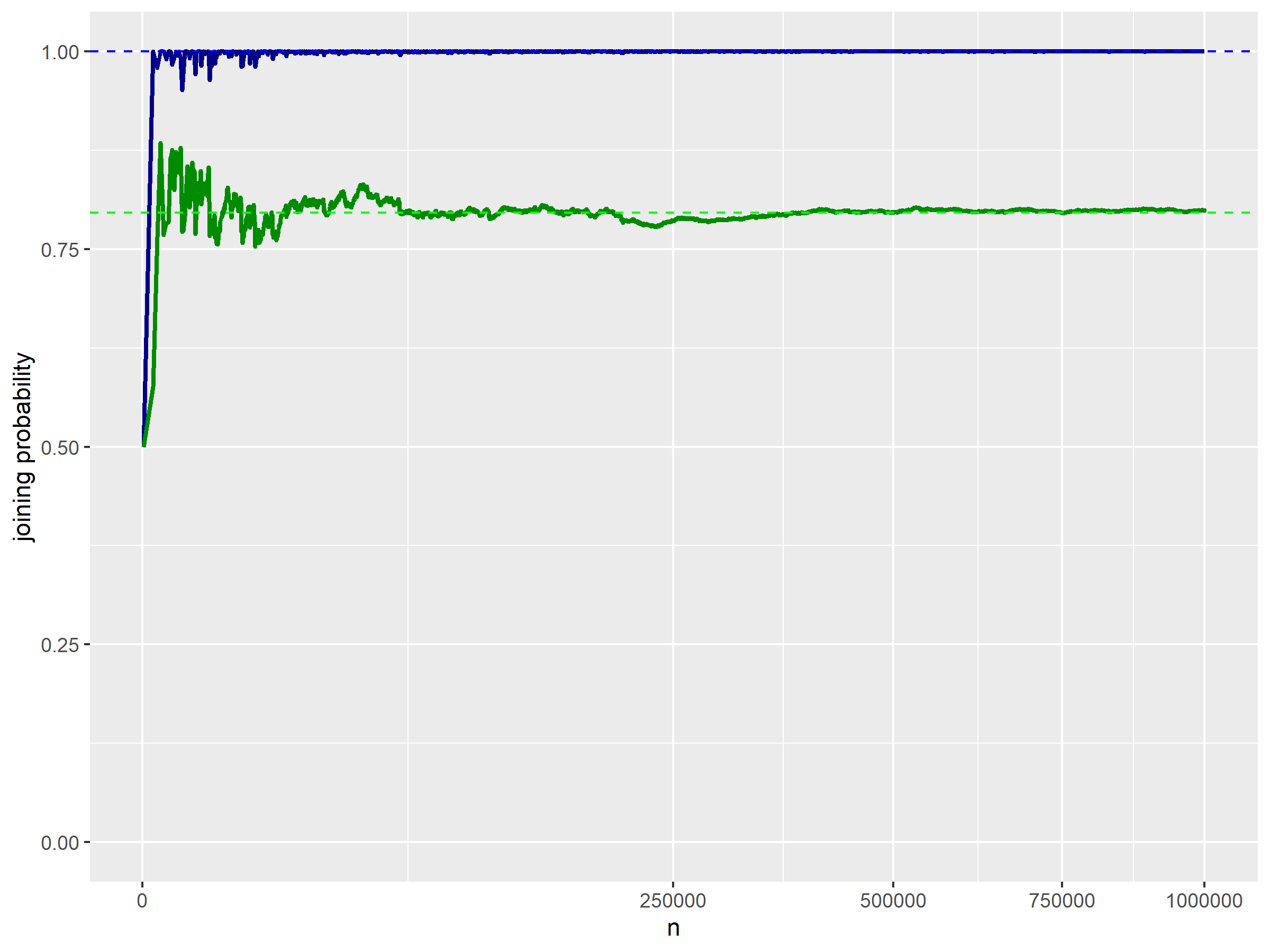"}}
\caption{\textcolor{black}{Convergence as $n\to\infty$ of the SA algorithm with multiple customer types. The joining probability at stage $n$, for type-1 customer (in blue) and type-2 customers (in green), is plotted vs. $n$ on a square-root scale. The blue dashed line depicts the correct equilibrium joining probability for type 1, which is given by $1$, and the green dashed line depicts the correct equilibrium joining probability for type 2, which is given by $0.796$.}} \label{fig:2types-mm1}
\end{figure}

\end{appendices}

\end{document}